\documentclass[11pt,a4paper]{article}

\usepackage[margin=1in]{geometry}
\usepackage{natbib}
\usepackage[T1]{fontenc}
\usepackage[utf8]{inputenc}
\usepackage{amsmath,amssymb,amsthm,mathtools,mathrsfs}
\usepackage{microtype}
\usepackage{graphicx}
\usepackage{float}
\usepackage[hidelinks]{hyperref}
\usepackage{url}
\usepackage{booktabs}
\usepackage{xcolor}
\usepackage{placeins}
\usepackage{longtable,array}

\newtheorem{theorem}{Theorem}[section]
\newtheorem{proposition}[theorem]{Proposition}
\newtheorem{corollary}[theorem]{Corollary}
\newtheorem{definition}[theorem]{Definition}
\newtheorem{remark}[theorem]{Remark}

\newenvironment{proofsketch}{\begin{proof}[Proof sketch]}{\end{proof}}

\newcommand{\R}{\mathbb{R}}
\newcommand{\Sd}{\mathbb{S}^{d-1}}
\newcommand{\cP}{\mathcal{P}}
\newcommand{\cM}{\mathcal{M}}

\newcommand{\mF}{\mathrm{F}}
\newcommand{\mN}{\mathrm{N}}
\newcommand{\mJ}{\mathrm{J}}
\newcommand{\mI}{\mathrm{I}}
\newcommand{\MF}{\mathcal{F}}
\newcommand{\MN}{\mathcal{N}}
\newcommand{\MJ}{\mathcal{J}}
\newcommand{\Wtwo}{\mathsf{W}_2}
\newcommand{\Wg}{\mathsf{W}_{\gamma}}
\newcommand{\Wgbb}{\mathsf{W}_{\gamma}^{\mathrm{BB}}}
\newcommand{\WgM}{\mathsf{W}_{\gamma}^{\mathrm{M}}}
\newcommand{\UWg}{\mathsf{UW}_{\gamma}}
\newcommand{\WQ}[1]{\mathsf{W}^{#1}}
\newcommand{\BWg}{\mathsf{B}_{\gamma}}
\newcommand{\tr}{\operatorname{tr}}
\newcommand{\divv}{\operatorname{div}}
\newcommand{\divS}{\operatorname{div}_{\Sd}}
\newcommand{\nablaS}{\nabla_{\Sd}}
\newcommand{\Id}{\mathrm{Id}}
\newcommand{\diag}{\operatorname{diag}}
\newcommand{\norm}[1]{\left\lVert #1 \right\rVert}
\newcommand{\abs}[1]{\left\lvert #1 \right\rvert}
\newcommand{\ip}[2]{\left\langle #1,#2 \right\rangle}
\newcommand{\cPol}{\mathcal P_\gamma}
\newcommand{\cK}{\mathcal K_\gamma}
\newcommand{\cA}{\mathcal A_\gamma}
\newcommand{\Sppd}{\mathbb S_{++}^d}
\newcommand{\KL}{\mathrm{KL}}
\hypersetup{pdftitle={Muon Dynamics as a Spectral Wasserstein Flow},
  pdfauthor={Gabriel Peyre}}

\title{Muon Dynamics as a Spectral Wasserstein Flow}
\author{Gabriel Peyr\'e\\CNRS and ENS, PSL Universit\'e\\\texttt{gabriel.peyre@ens.fr}}

\date{}

\begin{document}

\maketitle

\begin{abstract}
Gradient normalization can stabilize deep-learning optimization, and spectral
normalizations are especially natural for matrix-shaped parameter blocks;
Muon is the motivating example. We study an idealized deterministic,
continuous-time, vanishing-momentum version of this idea in the mean-field
regime, where wide models are represented by probability measures on parameter
space. Starting from normalized matrix flows, we introduce Spectral
Wasserstein distances indexed by norms $\gamma$ on positive semidefinite
matrices: the trace norm recovers the classical $\Wtwo$ distance, the operator
norm gives the Muon geometry, and Schatten norms interpolate between them. We
develop a static Kantorovich formulation, a max--min robust-cost
representation, and Gaussian reductions extending the Bures formula. For
monotone norms, we prove equivalence with a Benamou--Brenier formulation and
obtain a gradient-flow interpretation of normalized mean-field training
dynamics. Numerical experiments illustrate the resulting flows for MMD
objectives, Gaussian reductions, two-layer ReLU models, and shallow attention.
\end{abstract}

\section{Introduction}
\label{sec:intro}

Spectrally normalized optimizers such as Muon replace the Euclidean geometry
of gradient descent by a matrix-aware geometry adapted to block-shaped
parameters. This paper develops a mathematical framework for such optimizers
in very wide, possibly infinite-width, models. The central object is a
Spectral Wasserstein geometry in which their idealized dynamics become
mean-field gradient flows.

\paragraph{Normalized gradient descent and mean field dynamics.}
On the optimization side, a growing literature studies normalized first-order methods. The recent framework of \citet{PethickXieAntonakopoulosEtAl2025} is particularly relevant because it treats norm-constrained linear minimization oracles as a general language for normalized gradient methods and includes spectral normalizations as special cases. Earlier works such as \citet{CutkoskyMehta2020} and \citet{MurraySwensonKar2019} show how normalized gradient methods change the optimization dynamics even in nonconvex settings. For deep architectures, matrix-aware normalizations are especially natural: Shampoo is an early influential example of tensor/matrix-aware preconditioning \citep{GuptaKorenSinger2018}, and Muon has become a leading example of spectral normalization in large-scale training \citep{KellerJordan2024,LiuSuYaoEtAl2025}. This connects naturally with the mean-field description of wide neural networks through probability measures on parameter space, which underlies the landscape analysis of two-layer networks by \citet{MeiMontanariNguyen2018}, the optimal-transport convergence analysis of over-parameterized models by \citet{ChizatBach2018}, and the metric-gradient-flow viewpoint developed in \citet{AmbrosioGigliSavare2008}. Our work keeps this mean-field perspective but changes the underlying metric from Euclidean Wasserstein geometry to matrix-aware Spectral Wasserstein geometries.

\paragraph{Generalized optimal distances.}
On the transport side, our work is closest in spirit to generalized coupling costs and weak transport, as developed for instance by \citet{GozlanRobertoSamsonTetali2017}, \citet{BackhoffBeiglbockPammer2019}, and \citet{BackhoffPammer2022}. It is also related to covariance-dependent transport costs \citep{BurgerErbarHoffmannMatthesSchlichting2023}, to matrix-valued optimal transport \citep{NingGeorgiouTannenbaum2015,ChenGeorgiouTannenbaum2018}, and to the Bures--Wasserstein geometry of covariance matrices \citep{BhatiaJainLim2019}. Another useful perspective is to robustify the Wasserstein distance by optimizing over the cost itself. 
The closest antecedent is the subspace robust Wasserstein distance of \citet{PatyCuturi2019}, which corresponds to the special case of our spectral Wasserstein distance $\Wg$ obtained when $\gamma$ is a Ky Fan norm~\citep{fan1951maximum}.
Maximizing over costs is also connected to metric learning in Wasserstein discriminant analysis \citep{FlamaryCuturiCourtyRakotomamonjy2018}, ground metric learning \citep{CuturiAvis2014}, and congestion models for transportation networks \citep{CarlierJimenezSantambrogio2008}. The opposite direction, minimizing over costs, appears, for instance, in \citet{SebbouhCuturiPeyre2023}; this leads to a concave minimization problem useful for Gromov--Wasserstein-type structure. By contrast, our robustification is a concave maximization over admissible quadratic costs, leading to a convex static problem.

\paragraph{Contributions.}
The construction has three steps: a matrix normalization defines a descent
direction; its force-covariance formula extends that direction to arbitrary
measures; and a transport distance supplies the corresponding geometry.
We first derive spectrally normalized matrix flows and their mean-field
continuity equations in Section~\ref{sec:flows}. We then introduce the static
Spectral Wasserstein cost in Section~\ref{sec:static}, compare it with $\Wtwo$
in Proposition~\ref{prop:bounds}, establish its max--min robust-cost
representation in Theorem~\ref{thm:minmax}, and derive its Gaussian reduction
in Proposition~\ref{thm:gaussian}. For monotone norms,
Section~\ref{sec:dynamic} proves the equivalence between the static and
Benamou--Brenier formulations in Theorem~\ref{thm:staticdynamic}, yielding the
metric property in Corollary~\ref{cor:metric} and explicit geodesics in
Corollary~\ref{cor:geodesic}. Section~\ref{sec:wasserstein-flows} identifies the
normalized PDE as the corresponding metric gradient flow and discusses MLP,
shallow-attention, MMD, and linear-network examples. A Python computational
toolbox, illustrative notebooks, and code to reproduce the numerical
experiments are available at
\url{https://github.com/gpeyre/muon-dynamics}.

\textbf{Scope.} This paper does not attempt to explain why Muon-type
algorithms outperform vanilla gradient descent for particular architectures,
especially attention layers. Its goal is to provide the mathematical
formalism needed to investigate such questions in the regime of very wide
networks. The resulting transport geometries also raise theoretical questions
that extend beyond neural-network training. The finite-particle identity is
exact, while the gradient-flow characterization is conditional on stated
regularity assumptions. We do not prove convergence of particle systems to
the limiting PDE, global PDE well-posedness, or convergence of practical
stochastic, momentum-based Muon implementations.

\section{Spectral Gradient Flow and Mean Field Dynamics}
\label{sec:flows}

\paragraph{Spectral gradient flows on matrices.}

This section introduces the normalized dynamics studied in the paper, whose
geometry is governed by a matrix gauge $\gamma$. The ``Muon-type'' descent
corresponds to the case where $\gamma$ is the operator norm.
We impose monotonicity throughout the main text. It is used in the
static--dynamic equivalence of Theorem~\ref{thm:staticdynamic} and is crucial
for the Spectral Wasserstein gradient-flow interpretation in
Section~\ref{sec:wasserstein-flows}. Appendix~\ref{app:monge} explains what can
fail without monotonicity, while Appendix~\ref{app:root-norm-gauges} gives a
different extension beyond convex gauges.

\begin{definition}[Monotone gauge $\gamma$ and norm $\mN(V)$]
\label{def:monotone-norm}
Unless explicitly stated otherwise, $\gamma$ denotes the restriction to $\mathbb S_+^d$ of a
norm on $\mathbb S^d$. We assume that this restriction is monotone: whenever
$0\preceq S\preceq S'$, one has $\gamma(S)\le\gamma(S')$. 
The matrix norm induced by $\gamma$ is $\mN(V)\coloneqq \gamma(V^\top V)^{1/2}$; for $\gamma=\gamma_p$, one has $\mN(V)=\norm{V}_{S_{2p}}$. 
As shown in Proposition~\ref{prop:N-norm}, $\mN(V)$ is indeed a norm when $\gamma$ is monotone.
\end{definition}

The canonical examples used throughout this article are the Schatten gauges: if $\lambda(S)$ denotes the vector of eigenvalues of $S\succeq 0$, then $\gamma_p(S)\coloneqq \|S\|_{S_p}=\|\lambda(S)\|_{\ell_p}$, $p\in[1,\infty]$.
Other examples include the Ky Fan norms~\citep{fan1951maximum},
$\gamma(S)\coloneqq \sum_{i=1}^k \lambda_i(S)$, obtained by summing the $k$
largest eigenvalues. They interpolate between the operator norm $(k=1)$ and
the trace norm $(k=d)$ and recover the subspace-robust Wasserstein setting of
\citet{PatyCuturi2019}. Although the main text assumes that $\gamma$ is convex,
many metric and Benamou--Brenier results extend to the larger class of possibly
nonconvex root-norm gauges introduced in
Appendix~\ref{app:root-norm-gauges}.

\begin{definition}[Linear Minimization Oracle (LMO)]
\label{def:matrix-lmo}
For $G\in\R^{n\times d}$, the matrix LMO associated with $\gamma$ is any minimizer
$\mJ_\gamma(G)\in\operatorname*{argmin}_{V}
\{\langle G,V\rangle_F+\frac12\mN(V)^2\}$.
\end{definition}

The name is a slight abuse of the usual terminology: a classical LMO minimizes
a linear functional on a fixed ball. Here it is followed by an optimal scalar
step. Indeed, if $\mN^*$ is the dual norm and $G\ne0$, then
\begin{equation}\label{eq:lmo-unit-ball}
 U_G\in\operatorname*{argmin}_{\mN(U)\le1}\langle G,U\rangle_F
 \quad\Longrightarrow\quad
 \mJ_\gamma(G)=\mN^*(G)U_G.
\end{equation}
This follows by minimizing $-r\mN^*(G)+r^2/2$ over $r\ge0$;
$\mJ_\gamma(0)=0$. We use the same convention for the Measure-LMO below.

Let $\mF:\R^{n\times d}\to\R$ be a smooth objective.
Following \citet{PethickXieAntonakopoulosEtAl2025}, it yields both the discrete step and the ODE flow
\begin{equation}\label{eq:ode-basic}
	X_{k+1}=X_k+\tau\mJ_\gamma(\nabla\mF(X_k))
	\quad\text{and}\quad
	\dot X_t=\mJ_\gamma(\nabla\mF(X_t)).
\end{equation}
For $p=1$, $\mJ_{\gamma_1}(G)=-G$, so one recovers the usual gradient flow;
for $p=\infty$, $\mJ_{\gamma_\infty}$ is the trace-scaled partial polar factor
of $G$, giving a simplified ``Muon'' ODE. Explicit Schatten formulas are
given in Appendix~\ref{app:matrix-recap}.
The scalar factor matters: the unscaled polar update and this squared-norm
flow follow the same deterministic paths, away from stationary points, but
with different time parametrizations. Our metric fixes the latter clock.

\paragraph{Mean-field lift and normalized continuity equation.}
Given particles $x_1,\dots,x_n\in\R^d$, write
$X\coloneqq[x_1^\top;\ldots;x_n^\top]\in\R^{n\times d}$ and
$\mu_X\coloneqq n^{-1}\sum_i\delta_{x_i}$. We consider
permutation-invariant objectives lifted from measures,
$\mF(X)\coloneqq n\MF(\mu_X)$. This scaling applies, in particular, to the
two-layer MLP and attention models discussed in
Section~\ref{sec:wasserstein-flows}.

\begin{definition}[Measure-LMO]
\label{def:duality}
Let $\mu\in\cP_2(\R^d)$. For $v\in L^2(\mu;\R^d)$, define
$\MN_\mu(v)^2\coloneqq \gamma\!\left(\int v(x)v(x)^\top\,d\mu(x)\right)$.
For $g\in L^2(\mu;\R^d)$, the measure-LMO is any member of
$\MJ_\gamma^\mu(g) \in 
\operatorname*{argmin}_{v\in L^2(\mu;\R^d)}
\left\{ \int g(x)\cdot v(x)\,d\mu(x) +\frac12\MN_\mu(v)^2 \right\}$.
\end{definition}

For an empirical measure $\mu_X=n^{-1}\sum_i\delta_{x_i}$ transported by
particle velocities stacked into a matrix $V$, one has
$\MN_{\mu_X}(v)^2=n^{-1}\mN(V)^2$. If $G$ denotes the matrix with rows
$g(x_i)^\top$, then the rows of $\MJ_\gamma^{\mu_X}(g)$ coincide with those of
$\mJ_\gamma(G)$.
For a more general measure $\mu$, however, it is not a priori clear how to compute
$\MJ_\gamma^{\mu}(g)$. The following proposition reduces this infinite-dimensional
LMO to a finite-dimensional matrix problem.

\begin{proposition}[Finite-dimensional Measure-LMO]
\label{prop:measure-lmo-reduction}
Let $S_g\coloneqq\mathcal S_\mu(g)=\int g(x)g(x)^\top\,d\mu(x)$ and let
$P_g$ be the orthogonal projector onto $\operatorname{range}(S_g)$.  If
$S_g=0$, then $\MJ_\gamma^\mu(g)=0$. Otherwise, a valid Measure-LMO is given by
\[
    A_g\in
    \operatorname*{argmin}_{A=P_gA}
    \left\{
        \tr(S_g^{1/2}A)+\frac12\gamma(A^\top A)
    \right\},
    \qquad
    \MJ_\gamma^\mu(g)(x)=A_g^\top S_g^{\dagger/2}g(x).
\]
\end{proposition}

\begin{proof}
Associate to each
field $h\in L^2(\mu;\R^d)$ the operator
$T_h:\R^d\to L^2(\mu)$, $(T_ha)(x)=a\cdot h(x)$.  Then
$T_h^*T_h=\int hh^\top d\mu$ and
$\int g\cdot h\,d\mu=\langle T_g,T_h\rangle_{\rm HS}$.  Writing the polar
decomposition $T_g=U_gS_g^{1/2}$, set $A=U_g^*T_h$ and
$B=(I-U_gU_g^*)T_h$. Then $P_gA=A$,
$T_h^*T_h=A^\top A+B^*B$, and
$\langle T_g,T_h\rangle_{\rm HS}=\tr(S_g^{1/2}A)$.
Monotonicity of $\gamma$ shows that deleting $B$ cannot increase the
objective. Conversely, every $A=P_gA$ is realized by $T_h=U_gA$, or
$h(x)=A^\top S_g^{\dagger/2}g(x)$. The finite-dimensional objective is
continuous and coercive, so a minimizer exists. If $S_g=0$, then $g=0$
$\mu$-almost everywhere and the unique optimal field is zero.
\end{proof}
For Schatten gauges, this finite-dimensional LMO gives the explicit formula in
Corollary~\ref{cor:schatten-measure-selectors}; in particular, it covers the
Muon endpoint.
For force and displacement fields, we use ``covariance'' to mean the
uncentered second-moment matrix. Covariances of Gaussian laws retain their
usual centered meaning.

Following \citet{AmbrosioGigliSavare2008}, define the Wasserstein force and
the first variation by
\[
    g_\mu(x)\coloneqq
    \nabla_x\frac{\delta\MF}{\delta\mu}(x),
    \qquad
    \left.\frac{d}{d\varepsilon}\MF(\mu+\varepsilon\sigma)
    \right|_{\varepsilon=0}
    =
    \int \frac{\delta\MF}{\delta\mu}\,d\sigma
\]
for signed perturbations $\sigma$ of zero total mass for which the perturbed
measures are admissible, with one-sided derivatives when necessary.
The force $g_\mu$ is a cotangent quantity, not the spectral velocity:
the latter is $\MJ_\gamma^\mu(g_\mu)$.
The normalized mean-field dynamics is then the PDE 
\begin{equation}
\label{eq:measure-spectral-flow}
\partial_t\mu_t+\divv\!\bigl(\mu_t \MJ_\gamma^{\mu_t}(g_{\mu_t})\bigr)=0.
\end{equation}
When $\gamma(S)=\tr(S)$, $\MJ_\gamma^\mu(g)=-g$, and this reduces to the
usual formal $\Wtwo$ gradient-flow PDE. The particle interpretation requires
the following transport chain rule, not just a first variation along mixtures.

\begin{proposition}[Finite-width interpretation]
\label{cor:particle}
Assume $\mF(X)=n\MF(\mu_X)$ is $C^1$ and, at every empirical state $\mu$,
$\left.\frac{d}{dh}\MF((\Id+hw)_\#\mu)\right|_{h=0}
=\int g_\mu\cdot w\,d\mu$ for every field $w$ on its support.
Let $X_t$ be an absolutely continuous particle curve and set $\mu_t=\mu_{X_t}$,
using matching matrix and measure LMO selections.
As long as the atoms remain pairwise distinct, $X_t$ solves
ODE~\eqref{eq:ode-basic} if and only if $\mu_{X_t}$ solves
\eqref{eq:measure-spectral-flow} weakly. The same statement holds through
collisions when coincident atoms are assigned the same selected velocity.
The proof is in Appendix~\ref{app:matrix-recap}.
\end{proposition}

Extending the global-in-time PDE well-posedness theory of
\citet{ChizatBach2018}, which covers the trace case $p=1$ under smoothness and
growth assumptions on the feature map, to general Schatten geometries remains
a difficult open problem.
The Measure-LMO nevertheless defines the same evolution rule for empirical
and continuous distributions, without a prescribed number of particles.
For suitable continuous families it can be substantially simpler: affine
velocities preserve Gaussian distributions and close the dynamics on their
moments (Appendix~\ref{app:gaussian-linear}).
The major question left open by the construction above is whether
\eqref{eq:measure-spectral-flow} is the gradient flow of an actual distance
on probability measures. The next two sections define this distance and
prove a Benamou--Brenier formula with velocity action $\MN_\mu^2$.
Different velocities may generate the same continuity equation, so metric
speed need not equal the norm of an arbitrary representative.
Theorem~\ref{prop:formalGF} establishes equality for the selected descent
velocity under its stated hypotheses.

\begin{remark}[Connection with Muon]
Appendix~\ref{sec:muon-extensions} relates the normalized-flow viewpoint to
practical ingredients of Muon. Finitely many Newton--Schulz iterations replace
the exact LMO by a polynomial spectral map generated by an explicit
potential. Momentum can be modeled by an exponentially averaged force
variable, leading in the mean-field limit to a phase-space Liouville equation;
this is no longer a gradient flow on the position marginal alone, and thus lies
partly outside the Spectral Wasserstein geometry developed below.
Blockwise normalization is covered by block gauges, with different gauges
for independently scaled LMOs and for raw polar updates.
The same appendix distinguishes these deterministic extensions from noisy
minibatch updates, for which normalization and expectation do not commute.
\end{remark}

\section{Static Spectral Wasserstein Geometry}
\label{sec:static}

\paragraph{Kantorovich cost.}
This section introduces the Spectral Wasserstein transport cost. The key point is that the correct object is coupling-based and depends on a norm $\gamma$ acting on the global displacement covariance.

\begin{definition}[Static Spectral Wasserstein]
\label{def:static}
For $\mu,\nu\in\cP_2(\R^d)$, let $\Pi(\mu,\nu)$ be the set of probability measures on $\R^d\times\R^d$ with marginals $\mu$ and $\nu$, and define
$$
\Wg(\mu,\nu)^2
\coloneqq 
\inf_{\pi\in\Pi(\mu,\nu)}
\gamma\!\left(
\int_{\R^d\times\R^d}(y-x)(y-x)^\top\,d\pi(x,y)
\right).
$$
\end{definition}

The Monge restriction, the possible strictness of the coupling relaxation, and the obstruction to the triangle inequality for non-monotone norms are discussed in Appendix~\ref{app:monge}. In the monotone case, the full metric property is proved later through the dynamic formulation in Corollary~\ref{cor:metric}.

\paragraph{Comparison with $\Wtwo$ and cost-robust formulation.}

We now show that $\Wg$ has the same topology as $\Wtwo$. The proof, explicit
Schatten constants, and sharpness examples are given in
Appendix~\ref{app:static-details}.

\begin{proposition}[Norm comparison]
\label{prop:bounds}
Let $0<c_\gamma\le C_\gamma<\infty$ satisfy
$c_\gamma\tr(S)\le \gamma(S)\le C_\gamma\tr(S)$ for every $S\succeq0$;
such constants always exist. Then
$\sqrt{c_\gamma}\,\Wtwo(\mu,\nu)\le \Wg(\mu,\nu)
\le \sqrt{C_\gamma}\,\Wtwo(\mu,\nu)$.
\end{proposition}

The comparison bounds of Proposition~\ref{prop:bounds} already imply
separation and topological equivalence with $\Wtwo$. What is not immediate from
the static formulation alone is the triangle inequality. To address this, and to
clarify the geometry of $\Wg$, Theorem~\ref{thm:minmax} below recasts $\Wg$ as a
cost-robust Wasserstein distance. This reformulation has two immediate
consequences: a conditional Brenier theorem in the Monge setting
(Appendix~\ref{app:monge}) and the fact that $\Wg$ is a distance
(Appendix~\ref{subsec:metric-static}). The latter can also be obtained from the
Benamou--Brenier formulation; see Corollary~\ref{cor:metric}.

\begin{definition}[Representing set $\cK$]
\label{def:representing-set}
A representing set for $\gamma$ is a convex compact set
$\cK\subset \mathbb S_+^d$ such that, for every $S\succeq0$,
\[
    \gamma(S)=\max_{Q\in\cK}\tr(QS).
\]
%Equivalently, $\gamma$ is the support function of $\cK$ restricted to the positive semidefinite cone.
%
As recalled in Appendix~\ref{app:matrix-recap}, monotonicity of $\gamma$ on
$\mathbb S_+^d$ allows the representing set to be chosen inside
$\mathbb S_+^d$.
Remark~\ref{rem-kgamma-shatten} gives examples of such sets
$\cK$ for the Schatten norms $\gamma=\gamma_p$.
\end{definition}

\begin{theorem}[Max-min and cost-robust representation]
\label{thm:minmax}
For every positive semidefinite matrix
$Q\in\mathbb S_+^d$, let $\WQ{Q}$ denote the quadratic transport functional
associated with the cost $\ip{Q(x-y)}{x-y}$, namely
\[
    \WQ{Q}(\mu,\nu)^2
    \coloneqq
    \inf_{\pi\in\Pi(\mu,\nu)}
    \int_{\R^d\times\R^d} (y-x)^\top Q (y-x)\,d\pi(x,y).
\]
In particular, $\WQ{\Id}=\Wtwo$. 
Then, for any representing set $\cK$ for $\gamma$ as in
Definition~\ref{def:representing-set}, for every $\mu,\nu\in\cP_2(\R^d)$,
$$
\Wg(\mu,\nu)^2
=
\max_{Q\in\cK}
\inf_{\pi\in\Pi(\mu,\nu)}
\int (y-x)^\top Q (y-x)\,d\pi(x,y)
=
\max_{Q\in\cK}\WQ{Q}(\mu,\nu)^2.
$$ 
Denoting
$\Sigma_\pi\coloneqq\int (y-x)(y-x)^\top\,d\pi(x,y)$, any saddle point
$(Q^\star,\pi^\star)$ satisfies
$Q^\star\in\partial\gamma(\Sigma_{\pi^\star})$, where the subdifferential is
understood on the cone $\mathbb S_+^d$.
\end{theorem}

\begin{proofsketch}
Details are given in Appendix~\ref{app:static-details}. 
The support
representation of $\gamma$ gives
$\gamma(\Sigma_\pi)=\max_{Q\in\cK}\tr(Q\Sigma_\pi)$, with
$\tr(Q\Sigma_\pi)=\int (y-x)^\top Q(y-x)\,d\pi(x,y)$. Hence
$\Wg(\mu,\nu)^2=\inf_{\pi\in\Pi(\mu,\nu)}\max_{Q\in\cK}F(\pi,Q)$, where
$F(\pi,Q)\coloneqq\int (y-x)^\top Q(y-x)\,d\pi(x,y)$.
The conclusion follows by exchanging the infimum and the maximum. Indeed,
$\Pi(\mu,\nu)$ is convex, $\cK$ is compact and convex, and $F$ is affine in both
variables, continuous in $Q$, and lower semicontinuous in $\pi$. Sion's minimax
theorem therefore gives
$\inf_{\pi}\max_{Q}F(\pi,Q)=\max_{Q}\inf_{\pi}F(\pi,Q)$. The inner infimum is
precisely $\WQ{Q}(\mu,\nu)^2$, yielding
$\Wg(\mu,\nu)^2=\max_{Q\in\cK}\WQ{Q}(\mu,\nu)^2$.
\end{proofsketch}

\begin{corollary}[Kantorovich potentials]\label{cor:kantorovich-potentials}
For $\mu,\nu\in\cP_2(\R^d)$,
\[
 \Wg(\mu,\nu)^2
 =\max_{Q\in\cK}
   \sup_{\varphi(x)+\psi(y)\le (x-y)^\top Q(x-y)}
   \left(\int\varphi\,d\mu+\int\psi\,d\nu\right),
\]
where the supremum ranges over integrable admissible potentials.
\end{corollary}
The proof is in Appendix~\ref{app:matrix-recap}.

\paragraph{Discrete setting.}
This subsection spells out the finite-dimensional version used in computations and shows how the abstract coupling problem becomes a concrete convex program. We consider two discrete measures $\mu\coloneqq\sum_{i=1}^n a_i\delta_{x_i}$ and $\nu\coloneqq\sum_{j=1}^m b_j\delta_{y_j}$, with weights $a_i,b_j\ge 0$ summing to one. As in classical Kantorovich transport, a coupling is represented by a transport matrix $P\in\R_+^{n\times m}$ satisfying $P\mathbf 1_m=a$ and $P^\top\mathbf 1_n=b$. The discrete static Spectral Wasserstein problem can then be written as the convex optimization problem
$$
\Wg(\mu,\nu)^2=
\min_{P\ge 0,\;P\mathbf 1_m=a,\;P^\top\mathbf 1_n=b}
\gamma\!\big(
\sum_{i=1}^n\sum_{j=1}^m P_{ij}(y_j-x_i)(y_j-x_i)^\top
\big).
$$
For $\gamma=\gamma_p$ and $p=1$, this reduces to the usual linear
optimal-transport problem with quadratic cost, hence to a classical assignment
problem in the equal-weight case. For $p=\infty$, it admits a semidefinite
programming formulation through an epigraph constraint on the largest
eigenvalue of the displacement covariance.
Figure~\ref{fig:static-pair} illustrates couplings obtained for the trace and
operator objectives. The trace assignments are exact; the operator panels
use a rank-one alternating heuristic, not a certified SDP optimizer
(Appendix~\ref{app:extra-static-dynamic}). Proposition~\ref{prop:large-mean-schatten}
explains this effect in the regime where the two measures have well-separated
means: while the trace case remains isotropic, every Schatten geometry with
$p>1$ asymptotically selects a quadratic cost aligned with the mean
displacement direction, making the matching primarily sensitive to
displacements parallel to the vector joining the two means.
As explained in Appendix~\ref{app:monge}, for $p=\infty$ an active matrix
$Q^\star$ is often rank one, so the conditional Brenier argument neither
ensures uniqueness nor forces every optimizer to be induced by a map.
Figure~\ref{fig:static-pair} illustrates the resulting freedom in directions
orthogonal to the active one; see also Remark~\ref{rmk:muon-insight}.
Additional image-defined shape pairs are shown in Appendix~\ref{app:extra-static-dynamic}.

\begin{figure}[!ht]
\centering
\begin{tabular}{@{}c@{}cc@{}c@{}}
\includegraphics[width=.24\linewidth]{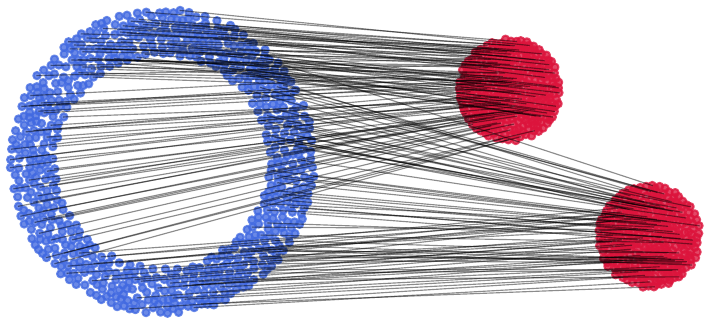}
&
\includegraphics[width=.24\linewidth]{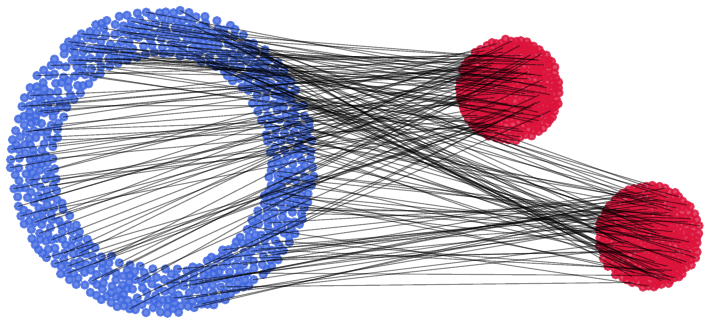}
&
\includegraphics[width=.24\linewidth]{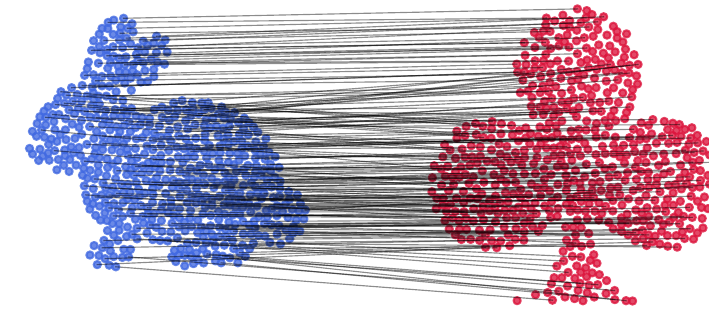}
&
\includegraphics[width=.24\linewidth]{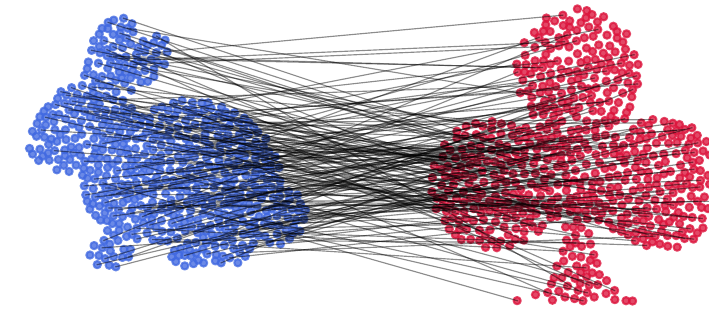}
\\[-.2em]
\scriptsize (a) $p=1$
&
\scriptsize (b) $p=\infty$
&
\scriptsize (c) $p=1$
&
\scriptsize (d) $p=\infty$\\[-2mm]
\end{tabular}
\caption{Static coupling illustrations for two source-target pairs: exact
trace assignments ($p=1$) and heuristic operator couplings ($p=\infty$).
Only a subsample of matched segments is displayed.}
\label{fig:static-pair}
\end{figure}

\paragraph{Gaussian setting and Spectral Bures distance.}\label{subsec:gaussians}

Gaussian marginals compress the transport problem to a finite-dimensional optimization over admissible covariance blocks.
The proof of Proposition~\ref{thm:gaussian} and the closed commuting formula
for Schatten norms are collected in Appendix~\ref{app:gaussians}.

\begin{proposition}[Gaussian reduction]
\label{thm:gaussian}
Let $\mu=\mathcal N(m_0,\Sigma_0)$ and
$\nu=\mathcal N(m_1,\Sigma_1)$. Then
\[
\Wg(\mu,\nu)^2
=
\inf\nolimits_{K:\,
\left(\begin{smallmatrix}
\Sigma_0&K\\ K^\top&\Sigma_1
\end{smallmatrix}\right)\succeq0}
\gamma\!\left(
(m_1-m_0)(m_1-m_0)^\top
+\Sigma_0+\Sigma_1-K-K^\top
\right).
\]
\end{proposition}

\section{Dynamic Formulation and Geodesics}
\label{sec:dynamic}
We now turn to the Benamou--Brenier formulation
\citep{BenamouBrenier2000}. Throughout this section, $\gamma$ is monotone.
Appendix~\ref{app:matrix-recap} then allows us to choose
$\cK\subset\mathbb S_+^d$; this positive-semidefinite representation is the
key ingredient in the static--dynamic equivalence below.

\begin{definition}[Dynamic Spectral Wasserstein]
\label{def:dynamic}
For $\mu_0,\mu_1\in\cP_2(\R^d)$, define
\begin{equation}
\label{eq:dynamic-spectral-wasserstein}
 \Wgbb(\mu_0,\mu_1)^2
 \coloneqq
 \inf_{(\mu_t,v_t)}
 \int_0^1
 \gamma\!\left(\int v_t(x)v_t(x)^\top\,d\mu_t(x)\right)dt,
\end{equation}
where the infimum is over narrowly continuous curves
$t\mapsto\mu_t\in\cP_2(\R^d)$ and Borel velocity fields satisfying
$\partial_t\mu_t+\divv(\mu_t v_t)=0$ and the endpoint conditions
$\mu_{t=0}=\mu_0$, $\mu_{t=1}=\mu_1$.
\end{definition}
The equivalent convex momentum formulation is given in Appendix~\ref{app:dynamic-details}.
The following theorem is the main result of the paper. Its proof, together
with proofs of the subsequent corollaries, is given in
Appendix~\ref{app:dynamic-details}.

\begin{theorem}[Static-dynamic equivalence]
\label{thm:staticdynamic}
If $\gamma$ is monotone, then
$\Wgbb(\mu_0,\mu_1)=\Wg(\mu_0,\mu_1)$.
\end{theorem}

\begin{proofsketch}
The detailed argument appears in Appendix~\ref{app:dynamic-details}. For the
dynamic-to-static inequality, let $(\mu_t,v_t)$ solve the continuity equation.
By the superposition principle, there is a probability measure $\eta$ on
absolutely continuous paths $\omega$ such that
$(e_t)_\#\eta=\mu_t$ and
$\dot\omega_t=v_t(\omega_t)$ for a.e. $t$, where
$e_t(\omega)=\omega_t$. The endpoint coupling
$\pi=(e_0,e_1)_\#\eta$ then satisfies
$\omega_1-\omega_0=\int_0^1v_t(\omega_t)\,dt$ for $\eta$-a.e. $\omega$.
By monotonicity and Appendix~\ref{app:matrix-recap}, the dual support set
$\cK$ may be chosen in $\mathbb S_+^d$. Hence, for every $Q\in\cK$, Jensen's
inequality gives
$\langle Q,(\omega_1-\omega_0)(\omega_1-\omega_0)^\top\rangle
\le \int_0^1 \langle Q,v_t(\omega_t)v_t(\omega_t)^\top\rangle\,dt$.
Integrating with respect to $\eta$ and taking the supremum over $Q\in\cK$
yields the static cost of $\pi$ bounded by the dynamic action. Taking the
infimum over admissible $(\mu_t,v_t)$ gives one inequality.
Conversely, let $\pi$ be an optimal static coupling, set
$z_t(x,y)=(1-t)x+ty$, and define $\mu_t=(z_t)_\#\pi$. An Eulerian velocity
field for this curve is $v_t(z)=\mathbb E_\pi[y-x\mid z_t=z]$.
Conditional Jensen's inequality yields
$\int v_t v_t^\top\,d\mu_t\preceq\int (y-x)(y-x)^\top\,d\pi$
for a.e. $t$.
Monotonicity of $\gamma$ therefore bounds the dynamic action by the static
cost of $\pi$, proving the reverse inequality.
\end{proofsketch}

\begin{corollary}[Metric properties]
\label{cor:metric}
$\Wg$ is a distance on $\cP_2(\R^d)$ inducing the same topology as $\Wtwo$.
\end{corollary}

% \begin{proof} Symmetry is obvious. Separation follows from Proposition~\ref{prop:bounds}. The triangle inequality follows from the dynamic formulation by concatenation and time rescaling of admissible curves. \end{proof}

\begin{corollary}[Geodesics]
\label{cor:geodesic}
If $\pi_*$ is any optimal coupling for $\Wg(\mu_0,\mu_1)$, then $\mu_t=((1-t)x+ty)_\#\pi_*$ is a constant-speed geodesic and $\Wg(\mu_s,\mu_t)=|t-s|\,\Wg(\mu_0,\mu_1)$ for every $0\le s\le t\le1$.
\end{corollary}

\begin{figure}[!ht]
\centering
\resizebox{\linewidth}{!}{%
\begin{tabular}{@{}cccccccccc@{}}
\includegraphics[width=.098\linewidth]{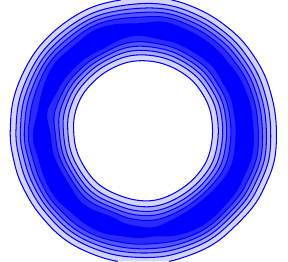}
&
\includegraphics[width=.098\linewidth]{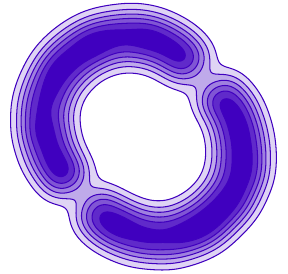}
&
\includegraphics[width=.098\linewidth]{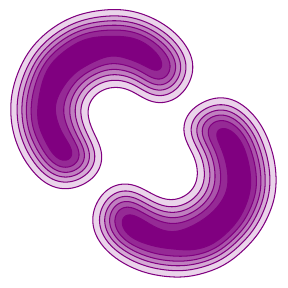}
&
\includegraphics[width=.098\linewidth]{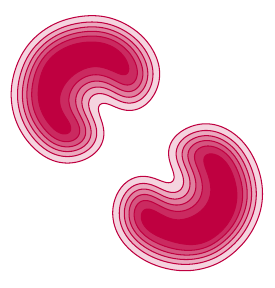}
&
\includegraphics[width=.098\linewidth]{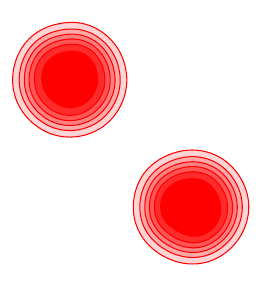}
&
\includegraphics[width=.098\linewidth]{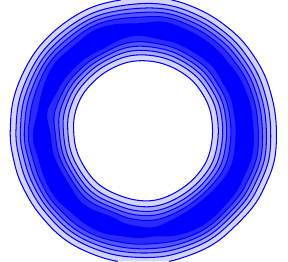}
&
\includegraphics[width=.098\linewidth]{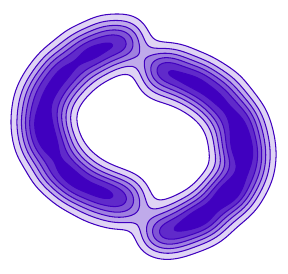}
&
\includegraphics[width=.098\linewidth]{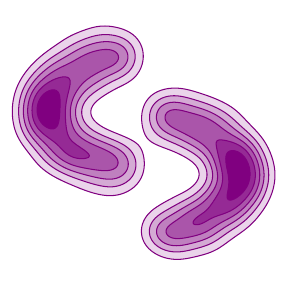}
&
\includegraphics[width=.098\linewidth]{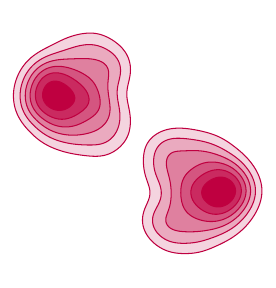}
&
\includegraphics[width=.098\linewidth]{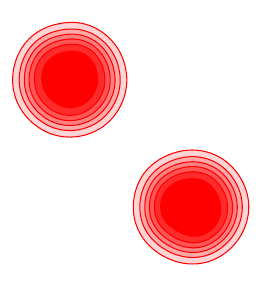}
\\[-.25em]
&
&
\normalsize $p=1$
&
&
&
&
&
\normalsize $p=\infty$
&
&\\[-2mm]
\end{tabular}
}
\caption{Displacement-interpolation density snapshots induced by the computed static couplings (exact for $p=1$, heuristic for $p=\infty$), for $p=1$ (first five panels) and for $p=\infty$ (last five panels). We use $n=1500$ particles. As in Figure~\ref{fig:static-pair}, the first and last snapshots correspond to the same source and target distributions ($t=0$ and $t=1$).}
\label{fig:static-interp-dynamic}
\end{figure}

Figure~\ref{fig:static-interp-dynamic} complements Figure~\ref{fig:static-pair}: once a coupling is fixed, its displacement interpolation is linear in time at the particle level, but the macroscopic density evolution differs markedly between $p=1$ and $p=\infty$. Appendix~\ref{app:extra-static-dynamic} shows additional examples.
For Schatten gauges these effects do not privilege the coordinate axes:
for every orthogonal matrix $R$, simultaneous rotation of the two measures
conjugates each displacement covariance by $R$, leaving $\Wg$ unchanged.
The exact transport problem is rotation invariant; the displayed heuristic
couplings may additionally reflect the solver initialization.

\section{Spectral Wasserstein Gradient Flows}
\label{sec:wasserstein-flows}

This section is written directly for arbitrary measures.
Proposition~\ref{cor:particle} already identifies its empirical version with
the normalized matrix ODE, hence with the idealized Muon dynamics at
$p=\infty$. The different role of Theorem~\ref{prop:formalGF} below is to
identify this evolution as a curve of maximal slope for a transport metric,
including when the distribution is not empirical.
% Recall the normalized PDE \eqref{eq:measure-spectral-flow}. Before Section~\ref{sec:dynamic}, this was only a normalized transport equation; the Benamou--Brenier identity now identifies $\MN_\mu$ as the infinitesimal norm of $\Wg$.

\paragraph{Metric interpretation of the normalized PDE.}

We now show that \eqref{eq:measure-spectral-flow} is the gradient flow of
$\MF$ for the metric $\Wg$, in the metric-gradient-flow sense of
\citet{AmbrosioGigliSavare2008}.
Recall that the metric derivative of an absolutely continuous curve
$t\mapsto\mu_t$ is
$|\mu'|_{\gamma}(t)\coloneqq\lim_{h\to0}\Wg(\mu_{t+h},\mu_t)/|h|$, whenever
the limit exists, and the local slope of $\MF$ is
$|\partial\MF|_{\gamma}(\mu)\coloneqq
\limsup_{\nu\to\mu}(\MF(\mu)-\MF(\nu))_+/\Wg(\mu,\nu)$.
We use the energy-identity formulation of a curve of maximal slope: the curve
and its energy are locally absolutely continuous, and, for every $s\le t$,
\[
 \MF(\mu_s)-\MF(\mu_t)
 =\frac12\int_s^t\bigl(|\mu'|_\gamma(r)^2
                       +|\partial\MF|_\gamma(\mu_r)^2\bigr)\,dr.
\]
The velocity action is minimized over representatives of the same measure
derivative, formally
$\inf\{\MN_\mu(v):-\divv(\mu v)=\dot\mu\}$.
We will use the speed bound supplied by BB and prove equality for the LMO
flow, without assuming equality for every velocity field.

\begin{theorem}[$\Wg$ gradient flow]
\label{prop:formalGF}
Fix $T<\infty$ and $\mu_0\in\cP_2(\R^d)$. Let $v_t(x)$ be continuous in
$(t,x)$, uniformly globally Lipschitz in $x$, with uniform linear growth on
$[0,T]$. Let $X_t$ be its characteristic flow and
$\mu_t=(X_t)_\sharp\mu_0$. Assume that $\MF$ is finite on a neighborhood of
this curve and admits forces $g_t=g_{\mu_t}\in L^2(\mu_t;\R^d)$ such that,
for constants $C,\delta>0$, every $t\in[0,T]$, and every
$\pi\in\Pi(\mu_t,\nu)$ with $\int|y-x|^2d\pi\le\delta$,
\begin{equation}\label{eq:transport-taylor}
 \left|\MF(\nu)-\MF(\mu_t)
       -\int g_t(x)\cdot(y-x)\,d\pi(x,y)\right|
 \le C\int|y-x|^2\,d\pi(x,y).
\end{equation}
If $v_t\in\MJ_\gamma^{\mu_t}(g_t)$, then $\mu_t$ is a curve of maximal
slope for $\MF$ with respect to $\Wg$. More precisely, almost everywhere,
\[
 |\mu'|_\gamma(t)=|\partial\MF|_\gamma(\mu_t)
 =\MN_{\mu_t}^*(g_t)=\MN_{\mu_t}(v_t),\qquad
 \frac{d}{dt}\MF(\mu_t)=-\MN_{\mu_t}^*(g_t)^2,
\]
where $\MN_\mu^*$ is the dual norm under the $L^2(\mu)$ pairing.
\end{theorem}

\begin{proofsketch}
The detailed proof is in Appendix~\ref{sec:proof-formal-gamma-flow}.
The transport expansion \eqref{eq:transport-taylor} and barycentric
projection of an optimal coupling give
$|\partial\MF|_\gamma(\mu_t)\le\MN_{\mu_t}^*(g_t)$.
Perturbing $\mu_t$ by $(\Id+hw)_\sharp\mu_t$, for arbitrary
$w\in L^2(\mu_t)$, proves the opposite inequality by duality.
The characteristic flow gives the chain rule
$\frac{d}{dt}\MF(\mu_t)=\int g_t\cdot v_t\,d\mu_t
=-\MN_{\mu_t}^*(g_t)^2$.
The dynamic formulation bounds the metric speed by
$\MN_{\mu_t}(v_t)=\MN_{\mu_t}^*(g_t)$; the slope inequality and the chain
rule give the reverse bound. These equalities imply the energy dissipation
identity.
\end{proofsketch}

These are verification assumptions on an existing curve, not an existence
or uniqueness theorem. They require neither a density nor invertibility of
the force covariance. A concrete smooth class satisfying
\eqref{eq:transport-taylor} is
$\MF(\mu)=\Phi(\int\varphi_1d\mu,\ldots,\int\varphi_md\mu)$,
with $C^2$ features of bounded Hessian and $\Phi\in C^2$ on a neighborhood
of the relevant moments; see Appendix~\ref{sec:proof-formal-gamma-flow}.
Regularity of the selected velocity must still be checked. In particular,
the theorem does not automatically cover ReLU kinks or changes of rank at
the operator endpoint.

\begin{remark}[Static and infinitesimal views of Muon]\label{rmk:muon-insight}
For $\gamma_\infty(S)=\lambda_{\max}(S)$, the optimal cost certificate
$Q^\star\in\partial\gamma_\infty(\Sigma_{\pi^\star})$ of
Theorem~\ref{thm:minmax} is supported on the top eigenspace of the
displacement covariance. If this eigenspace is spanned by a unit vector
$u^\star$, then $Q^\star=u^\star(u^\star)^\top$ and $\pi^\star$ minimizes
the projected cost $\langle y-x,u^\star\rangle^2$.
Infinitesimally, write $S_g=\int gg^\top d\mu$ and let $P_g$ project onto
its range. Corollary~\ref{cor:schatten-measure-selectors} gives the selector
and its velocity covariance:
\[
 v=\MJ_{\gamma_\infty}^\mu(g)=-\tr(S_g^{1/2})S_g^{\dagger/2}g,
 \qquad \int vv^\top d\mu=\tr(S_g^{1/2})^2P_g.
\]
Thus the flow equalizes the velocity covariance on the active force
subspace, rather than restricting motion to one direction. This is the
mean-field counterpart of Muon's matrix orthogonalization.
\end{remark}

\paragraph{Mean-field models for two-layer MLPs.}
\label{sec:mlp-meanfield}

We now specialize the abstract gradient-flow PDE to the standard mean-field parameterization of two-layer MLPs.
The mean-field predictor is
$H_\mu(z)\coloneqq\int_{\R^d}\phi(z,x)\,d\mu(x)$, and the training objective
is $\MF(\mu)\coloneqq R(H_\mu)$ for some loss $R$.
For a two-layer network with scalar output and activation $\sigma$, one may take parameter $x=(u,v)\in\R^d$ and feature $\phi(z,x)\coloneqq u\,\sigma(v\cdot z)$, where $u$ is the outer weight and $v$ is the inner weight.
This is the standard mean-field representation of a two-layer multilayer perceptron \citep{MeiMontanariNguyen2018,ChizatBach2018}. 
The notation is a deliberate departure from usual machine-learning conventions: the network input is denoted by $z$, while the trainable variable is denoted by $x$ so that it matches the transport notation used throughout the paper. 
Here $d$ is the sum of the block dimensions of $u$ and $v$. The experiment
normalizes the single $n\times d$ particle matrix. Separate normalization of
inner and outer weight blocks is also covered, using the sum or maximum block
gauges of Appendix~\ref{sec:muon-extensions}, depending on the scalar
normalization of each update. A basic example is the quadratic risk
$R(H)=\frac12\int \abs{H(z)-y^\star(z)}^2\,d\rho(z)$, or its empirical
version on a dataset. Because $H_\mu$ depends linearly on $\mu$, this makes
$\MF(\mu)$ a quadratic interaction functional. These energies are of MMD-type
\citep{GrettonBorgwardtRaschSchoelkopfSmola2012}, which is why the MMD experiment
below is a natural test case for the Spectral Wasserstein flow.
In Figure~\ref{fig:mlp-trajectories}, we train a two-dimensional to
one-dimensional ReLU network, with $\sigma(s)=\max(s,0)$, particles
$x=(u,v_1,v_2)\in\R^3$, and covariates $z\in\R^2$. The target
$y^\star$ is generated by a three-neuron teacher on Gaussian covariates
$z\sim\mathcal N(0,\Id_2)$.
We display both runs on their respective unit-normalized nominal horizons,
chosen by hand. This is a qualitative comparison, not a calibration of the
intrinsic gradient-flow clocks: the implementation also clips large Euler
steps. Appendix~\ref{sec:numerical-protocol} records the exact normalization,
safeguards, and an initialization-only dissipation check.
We display the particle evolution in the reduced two-dimensional plane
$|u|v$. The rays associated with the teacher neurons are shown as dashed
half-lines in Figure~\ref{fig:mlp-trajectories}; additional loss and circular
histogram plots are in Appendix~\ref{sec:archived-diagnostics}.
Changing $p$ visibly changes the paths and the loss profile on these
normalized plots. We do not infer an intrinsic convergence-rate or
wall-clock advantage from this comparison.
The positively two-homogeneous structure of ReLU motivates a spherical
reduction of the energy and homogeneous velocity action. The corresponding
regularity and selector assumptions are stated in
Appendix~\ref{app:homogeneous}; the nonsmooth ReLU case requires its own
chain-rule justification.

\begin{figure}[!th]
\centering
\begin{tabular}{@{}c@{\hspace{.02\linewidth}}c@{}}
\includegraphics[width=.40\linewidth]{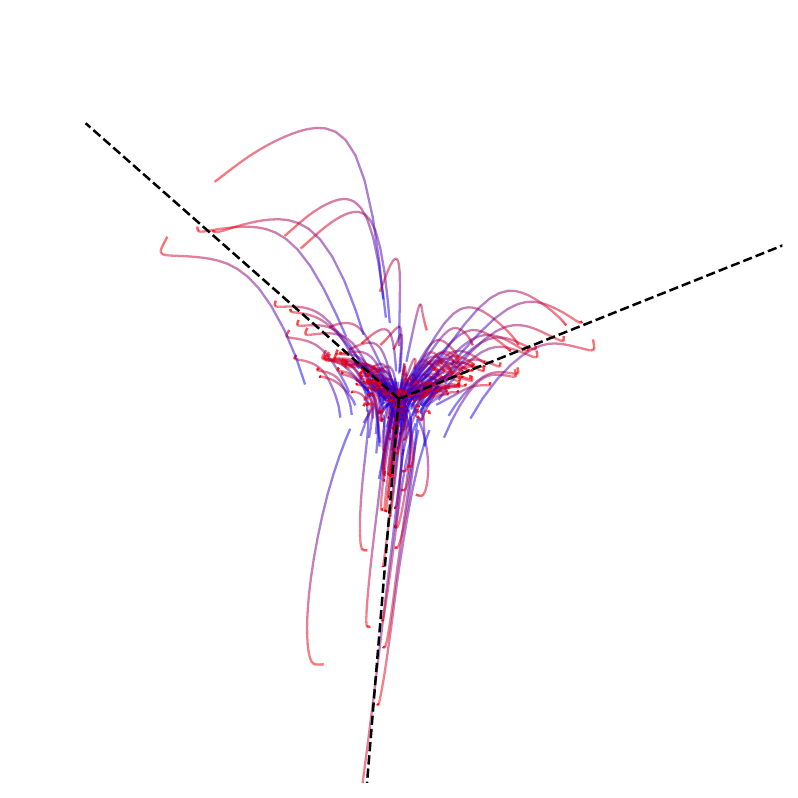}
&
\includegraphics[width=.40\linewidth]{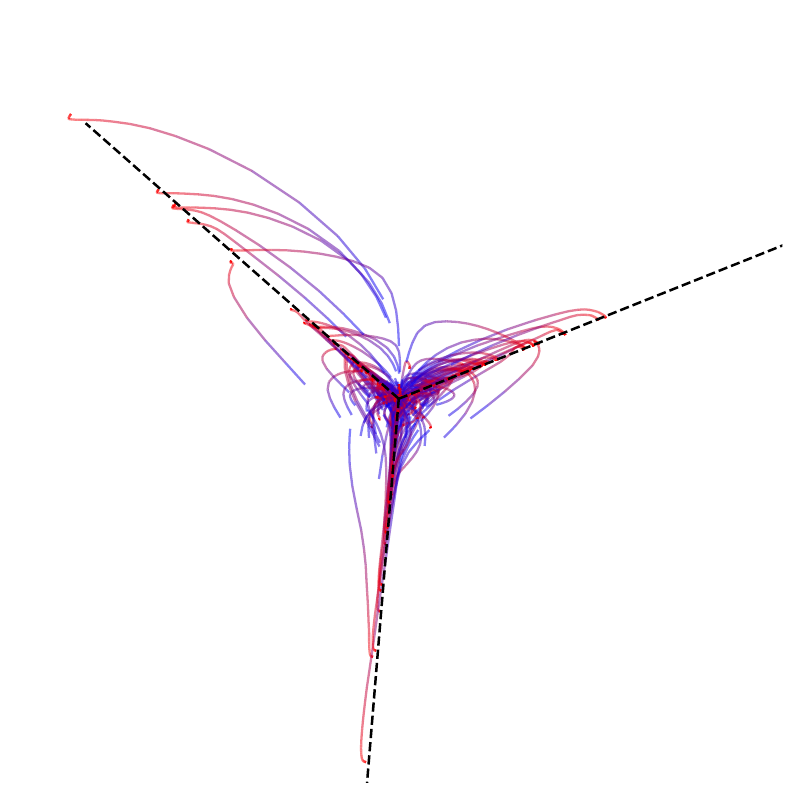}
\\
\small (a) $p=1$
&
\small (b) $p=\infty$
\end{tabular}
\caption{Two-layer ReLU MLP trajectories in the reduced parameter plane
$(|u|v_1,|u|v_2)$. Dashed rays indicate the teacher neurons.
The archived runs use separately normalized nominal horizons; neither
intrinsic flow times nor wall-clock runtimes are matched.}
\label{fig:mlp-trajectories}
\end{figure}

\paragraph{Mean-field shallow attention.}
Using the same spectral-flow construction, we also consider a shallow mean-field multi-head attention model. This is a useful stress test because attention uses softmax nonlinearities, precisely the regime where Muon-type spectral normalization is empirically attractive. A datum is a sequence $z=(z_1,\ldots,z_m)$ of tokens in $\R^{d_T}$. A head is a particle $x=(Q,K,V,O)$, with $Q,K,V,O\in\R^{d_H\times d_T}$, so $x\in\R^d$ with $d=4d_Hd_T$. Its elementary feature is the softmax attention map $\phi(z,x)_j\coloneqq O^\top\sum_{i=1}^m \frac{\exp(\langle Qz_j,Kz_i\rangle/\sqrt{d_H})}{\sum_{\ell=1}^m\exp(\langle Qz_j,Kz_\ell\rangle/\sqrt{d_H})}Vz_i$, and the mean-field layer is again $H_\mu(z)=\int\phi(z,x)d\mu(x)$. For an empirical measure $\mu_X$, this is a multi-head attention layer averaged over an arbitrary number $n$ of heads, which is the natural mean-field analogue of increasing the number of heads without fixing a finite architecture. 
The measure-level oracle requires $g_\mu\in L^2(\mu)$, which is not automatic
from $\mu\in\cP_2$. For bounded covariates and targets, the output has at most
quadratic parameter growth, but derivatives in $Q$ and $K$ can have cubic
growth. A finite sixth parameter moment is a sufficient integrability
condition for the least-squares force. Preservation of this moment bound and
regularity of the selected velocity remain additional requirements for a
rigorous flow application.

Figure~\ref{fig:attention-trajectories} in Appendix~\ref{sec:archived-diagnostics}
reports an experiment with $m=16$
tokens, $d_T=4$, and $d_H=3$: we fit a random five-head teacher by least squares
using a student with $n=200$ heads, over five paired trials with new teacher
and student initializations. The unfolding is a single $200\times48$ particle
matrix. The query/key scale controls the teacher's attention sparsity,
measured by the effective number of attended tokens
$N_{\mathrm{eff}}\coloneqq(\sum_i w_i^2)^{-1}$ averaged over attention rows.
The dense high-temperature regime has $N_{\mathrm{eff}}\approx16$, while the
sparse low-temperature regime has $N_{\mathrm{eff}}\approx2.1$.
The plots use separately normalized nominal horizons, with the safeguards
specified in Appendix~\ref{sec:numerical-protocol}; they are not calibrated
intrinsic flow times. The different loss profiles illustrate the dependence
on the geometry without establishing a convergence-rate ordering.

\paragraph{MMD gradient flows.}
As a simple model for generative fitting and sampling, in the spirit of MMD-based generative models \citep{LiSwerskyZemel2015,BinkowskiSutherlandArbelGretton2018}, we also minimize $\MF(\mu)=\operatorname{MMD}(\mu,\nu)^2$ for the energy-distance kernel $k(x,y)=-\|x-y\|_2$. For empirical measures $\mu=n^{-1}\sum_i\delta_{x_i}$ and $\nu=m^{-1}\sum_j\delta_{y_j}$, with $n=m=200$ particles in $\R^2$, this objective is $\operatorname{MMD}(\mu,\nu)^2=n^{-2}\sum_{i,i'}k(x_i,x_{i'})+m^{-2}\sum_{j,j'}k(y_j,y_{j'})-2(nm)^{-1}\sum_{i,j}k(x_i,y_j)$, with a small smoothing of $\|x-y\|_2$ in the implementation. The same initial cloud evolves with $p=1$ and $p=\infty$.
Strictly speaking, $-\|x-y\|_2$ is conditionally positive definite on
zero-mass signed measures, not a positive-definite kernel on all measures.
Adding the usual anchoring terms gives a positive-definite kernel with the
same quadratic energy on $\mu-\nu$ \citep{SejdinovicSriperumbudurGrettonFukumizu2013}.
Figure~\ref{fig:mmd-traj} isolates the effect of the spectral tangent norm without the additional architectural structure of a neural network. The trace-norm flow reacts locally to the MMD force field, while the operator-norm flow produces more globally coordinated trajectories; the scalar loss curves remain close on the separately normalized horizons.
A plausible explanation for the early outward motion in the operator case
is amplification of small transverse force components: the canonical
selector multiplies an active covariance mode of eigenvalue $\lambda_i$ by
$\tr(S_g^{1/2})/\sqrt{\lambda_i}$. Its velocity covariance is isotropic on
the active range, as in Remark~\ref{rmk:muon-insight}. This is a mechanism
consistent with the figure, not a separate proof of its transient behavior.

\begin{figure}[!ht]
\centering
\begin{tabular}{@{}c@{\hspace{.02\linewidth}}c@{}}
\includegraphics[width=.48\linewidth]{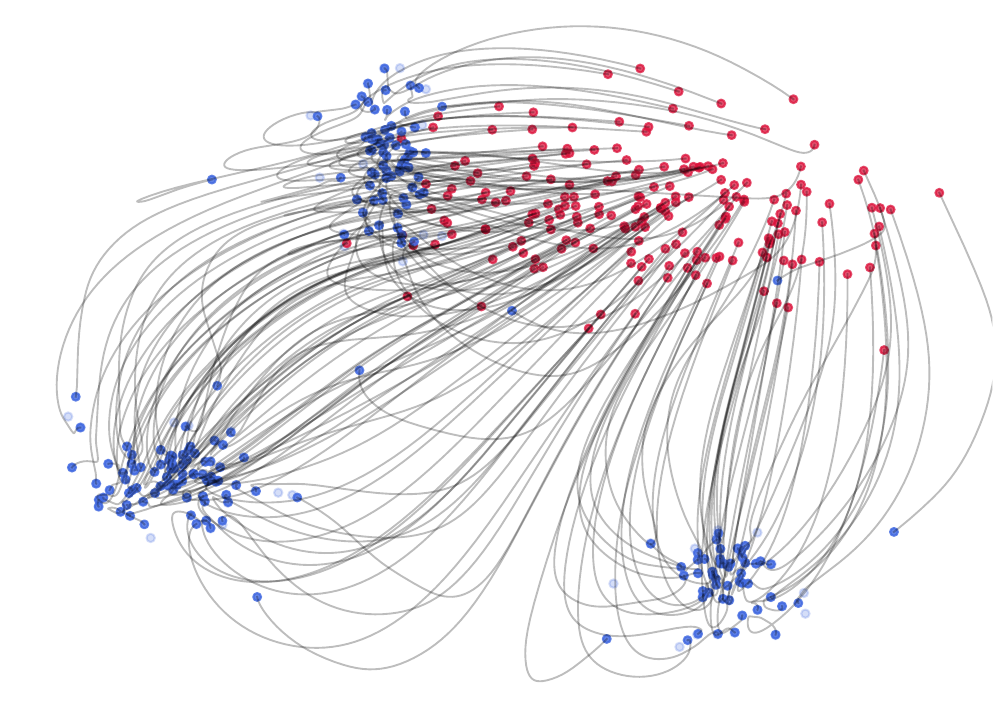}
&
\includegraphics[width=.48\linewidth]{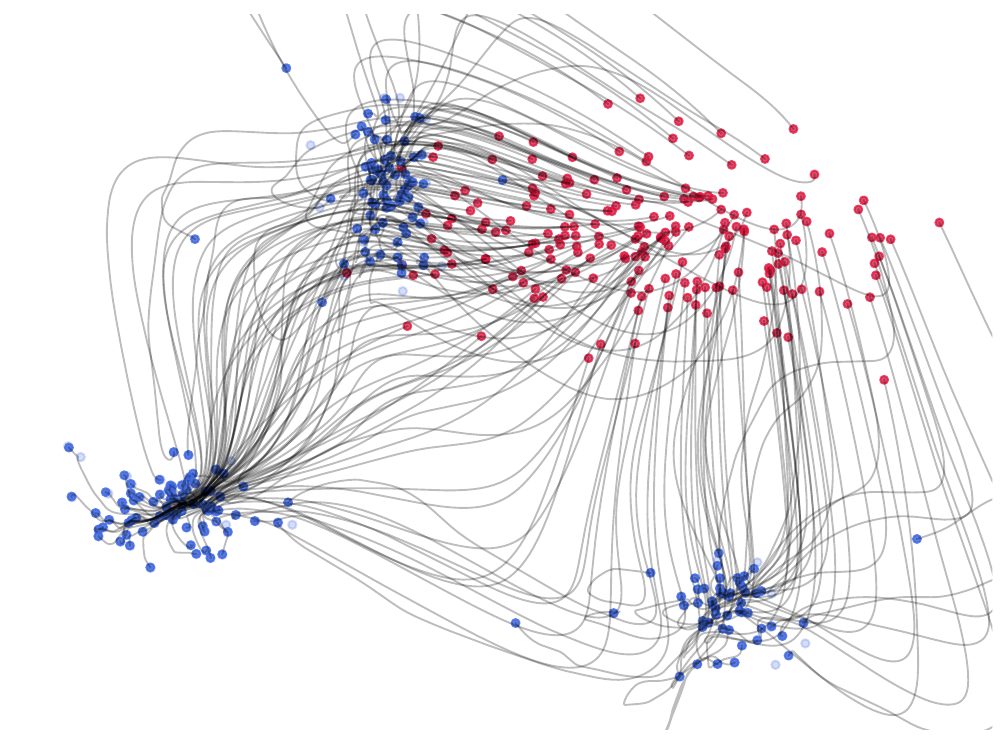}\\
\small (a) $p=1$ & \small (b) $p=\infty$
\end{tabular}
\caption{MMD particle trajectories on separately normalized nominal horizons.
The corresponding loss curves appear in
Figure~\ref{fig:mmd-loss}, Appendix~\ref{sec:archived-diagnostics}.}
\label{fig:mmd-traj}
\end{figure}

\paragraph{Gaussian-preserving linear networks.}\label{sec:gaussian-linear}
A particularly tractable example is a linear two-layer network,
$H_\mu(z)=\int u(v^\top z)\,d\mu(u,v)$, trained with quadratic loss.
Its force is linear in the parameter $x=(u,v)$. The canonical Measure-LMO
preserves this property, because it multiplies the force by a matrix
depending only on its covariance. An initially Gaussian distribution
therefore remains Gaussian along any sufficiently regular solution, and the
PDE closes on its mean and covariance. This illustrates a concrete advantage
of the measure formulation: infinitely many neurons can be described by a
finite-dimensional moment system. Appendix~\ref{app:gaussian-linear}
derives this system, its modal reductions, and the $p$-dependent implicit
covariance bias illustrated in Figures~\ref{fig:gaussian-leaves}
and~\ref{fig:gaussian-coupling}.

\paragraph{Geodesic convexity.}

Geodesic convexity is a useful structural benchmark for metric gradient
flows \citep{AmbrosioGigliSavare2008}, but is generally too strong for the
MLP and MMD objectives above. Appendix~\ref{app:convexity} studies convex
potentials and interactions along static displacement geodesics, and relative
entropy along a selected geodesic; Proposition~\ref{prop:basic-geodesic-convexity}
collects the basic examples. Convexity along every metric geodesic is stronger
and can fail at the Muon endpoint, where geodesics need not be unique.

A weaker route to convergence is a Polyak--Lojasiewicz (P-L) inequality,
which requires $|\partial\MF|_\gamma^2\ge2\kappa(\MF-\MF_\star)$, with
$\MF_\star=\inf\MF> -\infty$, without geodesic convexity. If
$\gamma(S)\le C_\gamma\tr(S)$, then
$|\partial\MF|_\gamma\ge C_\gamma^{-1/2}|\partial\MF|_2$;
thus a $\Wtwo$ P-L constant $\kappa$ transfers as $\kappa/C_\gamma$
(Proposition~\ref{prop:pl-transfer}). For every Schatten $p\ge1$, including
Muon, one can take $C_\gamma=1$. Writing $E(t)=\MF(\mu_t)-\MF_\star$,
a global $\kappa$-P-L flow with $E(0)<\infty$ satisfying
$E'=-|\mu'|_\gamma^2=-|\partial\MF|_\gamma^2$ obeys
$E(t)\le E(0)e^{-2\kappa t}$. If $\MF$ is also lower semicontinuous, then
$\Wg(\mu_t,\mu_\infty)\le\sqrt{2E(0)/\kappa}\,e^{-\kappa t}$ for some
minimizer $\mu_\infty$ (Proposition~\ref{prop:gamma-pl-convergence}).
Appendix~\ref{app:pl-poincare} gives the proofs and the local Poincar\'e
analysis for relative entropy. These are conditional guarantees, not a
convergence theorem or a trajectory-wise speed comparison for the networks above.

\section{Conclusion}
\label{sec:conclusion} 

The main message of this paper is that matrix-normalized optimizers for
mean-field neural models are naturally encoded by Spectral Wasserstein
distances. In this view, matrix normalization is not merely a preconditioning
device, but the local steepest-descent geometry generated by a spectral transport
cost. This clarifies the transport side, through coupling-based costs, geodesics,
and Gaussian reductions, and the optimization side, through a continuum
interpretation of normalized matrix updates.
The resulting picture is that the choice of spectral gauge shapes which
directions of motion are emphasized or suppressed in the mean-field dynamics,
thereby providing a geometric language for comparing matrix-normalized
optimizers. A major open question is the rigorous study of the Spectral
Wasserstein gradient-flow PDE: global-in-time existence, stability, and the
avoidance of local minima for shallow neural-network and MMD losses.

\section*{Acknowledgments and Disclosures}
The author thanks Tony Silveti-Falls, from whom he learned about the intricate
details of Muon and with whom he had insightful discussions about
Newton--Schulz approximation. This work was supported by the European Research
Council (ERC project WOLF) and the French government under the management of
Agence Nationale de la Recherche as part of the ``France 2030'' program,
reference ANR-23-IACL-0008 (PRAIRIE-PSAI).
The author has no competing financial interests to disclose.

\bibliographystyle{plainnat}
\bibliography{references}

\appendix

\section{Matrix Recap and Explicit LMOs}
\label{app:matrix-recap}
This appendix collects the matrix facts used in Section~\ref{sec:flows}. The canonical polar representing set is $\cPol\coloneqq\{Q\in\mathbb S^d:\tr(QS)\le\gamma(S)\text{ for every }S\succeq0\}$; finite-dimensional convex duality gives $\gamma(S)=\max_{Q\in\cPol}\tr(QS)$ for every $S\succeq0$.

We first show that the use of positive gauges on the cone
$\mathbb S_+^d$ is not essentially more restrictive than working with norms on
the full space of symmetric matrices: under monotonicity, every such gauge
admits a canonical norm extension.

\begin{proposition}[Norm extension of a monotone positive gauge]
\label{prop:gauge-norm-extension}
Let $\gamma:\mathbb S_+^d\to\R_+$ be a monotone positive gauge, i.e.,
$\gamma$ is convex, positively one-homogeneous, $\gamma(S)>0$ for every
nonzero $S\succeq0$, and $0\preceq S\preceq S'$ implies
$\gamma(S)\le\gamma(S')$. Define, for $A\in\mathbb S^d$,
\[
    \|A\|_\gamma
    \coloneqq
    \inf\Bigl\{
        \gamma(P)+\gamma(N):
        P,N\succeq0,\ A=P-N
    \Bigr\}.
\]
Then $\|\cdot\|_\gamma$ is a norm on $\mathbb S^d$, and its restriction to
$\mathbb S_+^d$ coincides with $\gamma$.
\end{proposition}

\begin{proof}
We first note that the admissible set is nonempty for every
$A\in\mathbb S^d$: for instance, the spectral decomposition
$A=A_+-A_-$ gives $A_+,A_-\succeq0$. Hence $\|A\|_\gamma$ is well-defined and
finite.

Positive homogeneity follows directly from the definition. If $a\ge0$, then
\[
    \|aA\|_\gamma
    =
    \inf_{A=P-N}
    \{\gamma(aP)+\gamma(aN)\}
    =
    a\|A\|_\gamma.
\]
If $a<0$, then $aA=(-a)N-(-a)P$ whenever $A=P-N$, so the same argument gives
$\|aA\|_\gamma=|a|\|A\|_\gamma$. In particular, $\|-A\|_\gamma=\|A\|_\gamma$.

We next prove the triangle inequality. Let $A=P_1-N_1$ and $B=P_2-N_2$ with
$P_i,N_i\succeq0$. Then
\[
    A+B=(P_1+P_2)-(N_1+N_2),
\]
and $P_1+P_2,N_1+N_2\succeq0$. By subadditivity of $\gamma$, which follows
from convexity and one-homogeneity,
\[
    \gamma(P_1+P_2)+\gamma(N_1+N_2)
    \le
    \gamma(P_1)+\gamma(P_2)+\gamma(N_1)+\gamma(N_2).
\]
Taking the infimum over all decompositions of $A$ and $B$ yields
$\|A+B\|_\gamma\le\|A\|_\gamma+\|B\|_\gamma$.

It remains to prove nondegeneracy. Monotonicity and subadditivity imply
$|\gamma(S)-\gamma(T)|\le\gamma(I)\|S-T\|_{\rm op}$ on the PSD cone,
so $\gamma$ is continuous. Since it is strictly positive on the
compact section $\{S\succeq0:\tr(S)=1\}$, there exists $c>0$ such that
$\gamma(S)\ge c\,\tr(S)$ for every $S\succeq0$. Therefore, for every
decomposition $A=P-N$,
\[
    \gamma(P)+\gamma(N)
    \ge
    c\,\tr(P+N).
\]
Moreover, from $A=P-N$ we have $P+N\succeq A$ and $P+N\succeq -A$. Let
$E_+$ and $E_-$ denote the spectral projectors of $A$ associated with its
positive and negative eigenvalues. Taking traces after compression to these
two subspaces gives
\[
 \tr(E_+(P+N))\ge\tr(A_+),
 \qquad
 \tr(E_-(P+N))\ge\tr(A_-).
\]
Since $P+N\succeq0$, its trace on the kernel of $A$ is nonnegative, and hence
$\tr(P+N)\ge\tr|A|$. Therefore
\[
    \gamma(P)+\gamma(N)\ge c\,\tr|A|.
\]
Taking the infimum gives $\|A\|_\gamma\ge c\,\tr|A|$. Thus
$\|A\|_\gamma=0$ implies $A=0$.

Finally, let $S\succeq0$. The decomposition $S=S-0$ gives
$\|S\|_\gamma\le\gamma(S)$. Conversely, if $S=P-N$ with $P,N\succeq0$, then
$P=S+N\succeq S$, and by monotonicity,
$\gamma(P)\ge\gamma(S)$. Hence
\[
    \gamma(P)+\gamma(N)\ge\gamma(P)\ge\gamma(S).
\]
Taking the infimum over all such decompositions gives
$\|S\|_\gamma\ge\gamma(S)$. Therefore $\|S\|_\gamma=\gamma(S)$ for every
$S\succeq0$.
\end{proof}

The next proposition characterizes the monotone norms for which representing sets may be chosen inside the positive-semidefinite cone.
\begin{proposition}[PSD representation and monotonicity]
\label{prop:psdpolar}
The following are equivalent:
\begin{itemize}
\item $\gamma$ is monotone on $\mathbb S_+^d$, namely
$$
0\preceq S\preceq T \Longrightarrow \gamma(S)\le \gamma(T);
$$
\item there exists a convex compact representing set $\cK\subset \mathbb S_+^d$ such that
$$
\gamma(S)=\max_{Q\in \cK}\tr(QS)
\qquad\text{for every }S\succeq 0.
$$
\end{itemize}
Moreover, when these properties hold, the canonical choice
$$
\cK=\cPol\cap \mathbb S_+^d
$$
is admissible.
\end{proposition}

\begin{proof}
If $\cK\subset \mathbb S_+^d$, then for $0\preceq S\preceq T$ one has $\tr(QS)\le \tr(QT)$ for every $Q\in\cK$, hence $\gamma(S)\le \gamma(T)$ by taking suprema.

Conversely, assume $\gamma$ is monotone and use
Proposition~\ref{prop:gauge-norm-extension} to extend it to a norm on
$\mathbb S^d$. Define
$$
\cK^{+}\coloneqq
\{Q\in\mathbb S_+^d:\ \tr(QS)\le \gamma(S)\text{ for every }S\succeq0\}.
$$
It is immediate that $\cK^{+}$ is convex, closed, bounded, and contained in
$\cPol\cap\mathbb S_+^d$. We claim that it represents $\gamma$. Fix
$S\succeq0$. If $S\succ0$, choose a subgradient
$Q\in\partial\gamma(S)$ in the ambient space of symmetric matrices. Since
$\gamma$ is convex and one-homogeneous, $\tr(QS)=\gamma(S)$ and
$\tr(QT)\le\gamma(T)$ for every $T\succeq0$. Moreover $Q\succeq0$: indeed,
for every $R\succeq0$ and every sufficiently small $\varepsilon>0$ such that
$S-\varepsilon R\succeq0$, the subgradient inequality gives
$$
\gamma(S-\varepsilon R)\ge \gamma(S)-\varepsilon\tr(QR),
$$
while monotonicity gives $\gamma(S-\varepsilon R)\le\gamma(S)$. Hence
$\tr(QR)\ge0$ for all $R\succeq0$, i.e., $Q\succeq0$. Thus $Q\in\cK^{+}$
and $\gamma(S)\le\sup_{P\in\cK^{+}}\tr(PS)$. The reverse inequality follows
from the definition of $\cK^{+}$. For singular $S\succeq0$, apply the
previous argument to $S+\varepsilon I$. Compactness gives a convergent
subsequence of supporting matrices; passing to the limit and using continuity
of $\gamma$ yields equality at $S$. Therefore
$\gamma(S)=\max_{Q\in\cK^{+}}\tr(QS)$ for every $S\succeq0$.
\end{proof}

\begin{proposition}[Norm property of $\mN$]\label{prop:N-norm}
Assume that $\gamma$ is monotone and admits a positive semidefinite
representing set $\cK\subset\mathbb S_+^d$, so that
$\gamma(S)=\max_{Q\in\cK}\tr(QS)$ for every $S\succeq0$. Define
$\mN(V)\coloneqq\gamma(V^\top V)^{1/2}$. Then $\mN$ is a seminorm. If, in
addition, $\gamma(S)>0$ for every nonzero $S\succeq0$, then $\mN$ is a norm.
\end{proposition}

\begin{proof}
For every matrix $V$, one has
\[
    \mN(V)^2
    =
    \gamma(V^\top V)
    =
    \max_{Q\in\cK}\tr(QV^\top V)
    =
    \max_{Q\in\cK}\tr(VQV^\top).
\]
Since $Q\succeq0$, we may write
$\tr(VQV^\top)=\|VQ^{1/2}\|_F^2$. Hence
\[
    \mN(V)
    =
    \max_{Q\in\cK}\|VQ^{1/2}\|_F.
\]
For each fixed $Q\in\cK$, the map $V\mapsto\|VQ^{1/2}\|_F$ is a seminorm.
Therefore, for any matrices $V,W$ and scalar $a$,
\[
    \mN(aV)
    =
    \max_{Q\in\cK}\|aVQ^{1/2}\|_F
    =
    |a|\mN(V),
\]
and
\[
    \mN(V+W)
    =
    \max_{Q\in\cK}\|(V+W)Q^{1/2}\|_F
    \le
    \max_{Q\in\cK}
    \bigl(\|VQ^{1/2}\|_F+\|WQ^{1/2}\|_F\bigr)
    \le
    \mN(V)+\mN(W).
\]
Thus $\mN$ is a seminorm. If moreover $\gamma(S)>0$ for every nonzero
$S\succeq0$, then $\mN(V)=0$ implies $\gamma(V^\top V)=0$, hence
$V^\top V=0$, and therefore $V=0$. Thus $\mN$ is a norm.
\end{proof}

\begin{remark}[$\mathcal K_\gamma$ for the Schatten $p$-norms]\label{rem-kgamma-shatten}
For $\gamma=\gamma_p$, a convenient supporting set is obtained from the dual
Schatten unit ball. Namely, if $p\in[1,\infty]$ and $p'\in[1,\infty]$ is the
conjugate exponent, $1/p+1/p'=1$, one may take
\[
    \mathcal K_{\gamma_p}
    =
    \mathcal P_{\gamma_p}\cap\mathbb S_+^d
    \coloneqq
    \Bigl\{
        G\in\mathbb S_+^d:
        \|G\|_{S_{p'}}\leq 1
    \Bigr\}.
\]
One can, however, choose a smaller supporting set in the endpoint cases.
For $p=1$, since $\|A\|_{S_1}=\operatorname{tr}(A)$ on $\mathbb S_+^d$, it is
enough to take $\mathcal K_{\gamma_1}=\{I\}$. For $p=\infty$, one may take the
set of density matrices,
\[
    \mathcal K_{\gamma_\infty}
    =
    \Bigl\{
        G\in\mathbb S_+^d:
        \operatorname{tr}(G)=1
    \Bigr\}.
\]
\end{remark}

The next proposition gives the explicit LMO associated with each Schatten geometry.
\begin{proposition}[Explicit Schatten selectors]
\label{prop:schatten}
Let $\gamma=\gamma_p=\norm{\cdot}_{S_p}$, let $r=2p$, and let
$q=r'=2p/(2p-1)$ be the dual exponent, with $q=1$ for $p=\infty$. If
$$
G=U\diag(\sigma_i)W^\top
$$
is a compact singular value decomposition over the positive singular values,
then
$$
\mJ_p(G)
=
-\norm{G}_{S_q}^{2-q}\,
U\diag(\sigma_i^{q-1})W^\top.
$$
For $G=0$, set $\mJ_p(G)=0$.
In particular,
\begin{itemize}
\item $p=1$: $\mJ_1(G)=-G$, which recovers ordinary gradient flow;
\item $p=2$: $\mJ_2(G)=-\norm{G}_{S_{4/3}}^{2/3}\,U\diag(\sigma_i^{1/3})W^\top$, which we call the Frobenius-intermediate normalization;
\item $p=\infty$:
\begin{equation}
\label{eq:muon-selector}
\mJ_\infty(G)
=
-\norm{G}_{S_1}\,UW^\top
=
-\tr((G^\top G)^{1/2})\,G\,(G^\top G)^{\dagger/2},
\end{equation}
which is the Muon normalization.
\end{itemize}
\end{proposition}

\begin{proof}
The displayed formula is the negative gradient of
$\frac12\norm{G}_{S_q}^2$ when $1<q\le2$. At the endpoint $q=1$, the compact
factor $UW^\top$ is a canonical partial-polar subgradient; it also fixes
the value on the null singular subspaces.
\end{proof}

\subsection{Proofs and dual formulas for measure LMOs}
\paragraph{Proof of Proposition~\ref{prop:measure-lmo-reduction}.}
\begin{proof}
Associate to every vector field $h\in L^2(\mu;\R^d)$ the Hilbert--Schmidt
operator $T_h:\R^d\to L^2(\mu)$ defined by $(T_ha)(x)=a\cdot h(x)$. Then
$$
T_h^*T_h=\int hh^\top\,d\mu,
\qquad
\int g\cdot h\,d\mu=\langle T_g,T_h\rangle_{\rm HS}.
$$
Let $S_g=T_g^*T_g$ and write the polar decomposition
$$
T_g=U_gS_g^{1/2},
\qquad
U_g^*U_g=P_g,
$$
where $P_g$ is the orthogonal projector onto $\operatorname{range}(S_g)$.
Let $\Pi_g=U_gU_g^*$ be the orthogonal projector onto the range of $T_g$ in
$L^2(\mu)$. For any $h$, replacing $T_h$ by $\Pi_gT_h$ leaves the linear term
unchanged and decreases the covariance:
$$
\langle T_g,\Pi_gT_h\rangle_{\rm HS}
=
\langle T_g,T_h\rangle_{\rm HS},
\qquad
(\Pi_gT_h)^*(\Pi_gT_h)=T_h^*\Pi_gT_h\preceq T_h^*T_h.
$$
By monotonicity of $\gamma$, this projection cannot increase the LMO
objective. Hence it is enough to minimize over operators of the form
$T_h=U_gA$ with $A=P_gA$. For such fields,
$$
\int g\cdot h\,d\mu=\tr(S_g^{1/2}A),
\qquad
\int hh^\top\,d\mu=A^\top A.
$$
Finally, the operator $T_h=U_gA$ corresponds to the vector field
$h(x)=A^\top S_g^{\dagger/2}g(x)$. This proves the finite-dimensional
reduction and the lifted selector.
\end{proof}

\begin{corollary}[Schatten measure selectors]
\label{cor:schatten-measure-selectors}
Let $\gamma=\gamma_p=\|\cdot\|_{S_p}$ and set
$q=(2p)'$, with $q=1$ for $p=\infty$. Then
$$
\MJ_{\gamma_p}^\mu(g)(x)
=
-\|S_g^{1/2}\|_{S_q}^{2-q}\,S_g^{(q-2)/2}g(x),
\qquad
S_g=\int gg^\top\,d\mu,
$$
where the power is taken on $\operatorname{range}(S_g)$ and is zero on its
kernel; when $S_g=0$, the selector is zero. In particular, $p=1$ gives
$\MJ_{\gamma_1}^\mu(g)=-g$, while
$p=\infty$ gives
$$
\MJ_{\gamma_\infty}^\mu(g)(x)
=
-\tr(S_g^{1/2})\,S_g^{\dagger/2}g(x).
$$
\end{corollary}

\begin{proof}
By Proposition~\ref{prop:measure-lmo-reduction}, one only needs to solve
$$
\min_A\left\{\tr(S_g^{1/2}A)+\frac12\|A\|_{S_{2p}}^2\right\}
$$
on the active range. The minimizer is the Schatten duality map
$$
A_*=-\|S_g^{1/2}\|_{S_q}^{2-q}S_g^{(q-1)/2}.
$$
Lifting by $h(x)=A_*^\top S_g^{\dagger/2}g(x)$ gives the formula.
\end{proof}

The displayed Schatten selector is Borel measurable, including at singular
covariances, by its spectral-calculus definition. Measurability does not imply
continuity: for example, the canonical matrix selector sends
$G_\varepsilon=\diag(1,\varepsilon)$ to
$-(1+\varepsilon)I$ for $\varepsilon>0$, but sends
$G_0=\diag(1,0)$ to $-\diag(1,0)$. Thus rank changes at the Muon endpoint
require a separate analysis for ODE or PDE well-posedness.
The full argmin is also not necessarily a singleton: for
$G=\diag(1,0)$ and the operator tangent norm, every
$V=-\diag(1,a)$ with $|a|\le1$ attains the minimum $-1/2$.
The displayed canonical selector chooses $a=0$.

The support representation of $\gamma$ gives another useful construction. It
is less intrinsic than the finite-dimensional LMO above, but when it applies it
produces the explicit formula $v= -Q^{-1}g$ and identifies an active quadratic
cost. Its limitation is the required attainment at an invertible matrix,
which may fail for rank-deficient force covariances, including at $p=\infty$.
This is not an inherent rank-one restriction at the operator endpoint: for
$S_g\succ0$ and $\cK=\{Q\succeq0:\tr Q\le1\}$, the minimizer is
$Q^*=S_g^{1/2}/\tr(S_g^{1/2})\succ0$. The rank-one certificates that can
appear in the static transport problem play a different role.

\begin{proposition}[Inverse-trace dual certificate]
\label{prop:dual-inverse-trace-selector}
Assume that $\cK\subset\mathbb S_+^d$ is a compact convex representing set for
$\gamma$ and that $\cK\cap\Sppd\ne\emptyset$.  Then
$$
\inf_v
\left\{
\int g\cdot v\,d\mu
+\frac12\gamma\!\left(\int vv^\top\,d\mu\right)
\right\}
=
-\frac12
\inf_{Q\in\cK\cap\Sppd}
\tr\!\bigl(Q^{-1}S_g\bigr),
\qquad
S_g=\mathcal S_\mu(g).
$$
If the matrix infimum is attained at
$$
Q_\mu^*\in
\operatorname*{argmin}_{Q\in\cK\cap\Sppd}
\tr\!\bigl(Q^{-1}S_g\bigr),
$$
then $v_*(x)=-(Q_\mu^*)^{-1}g(x)$ is a Measure-LMO.
\end{proposition}

\begin{proof}
For $Q\in\cK$, define
$$
\Phi_Q(v)\coloneqq \int g\cdot v\,d\mu+\frac12\int v^\top Qv\,d\mu.
$$
Since $\gamma(R)=\max_{Q\in\cK}\tr(QR)$ for $R\succeq0$, the Measure-LMO
objective is $\inf_v\max_{Q\in\cK}\Phi_Q(v)$.  The convex-concave minimax
theorem gives
$$
\inf_v\max_{Q\in\cK}\Phi_Q(v)
=
\max_{Q\in\cK}\inf_v\Phi_Q(v).
$$
Indeed, the matrix set is compact and convex, while any
$Q_0\in\cK\cap\Sppd$ provides coercivity in $v$.  For fixed positive definite
$Q$, the inner minimization gives
$$
v_Q(x)=-Q^{-1}g(x),
\qquad
\inf_v\Phi_Q(v)=-\frac12\tr\!\bigl(Q^{-1}\mathcal{S}_\mu(g)\bigr).
$$
Convexity of $\cK$ allows singular matrices to be approximated by
$(1-\varepsilon)Q+\varepsilon Q_0\succ0$, which proves the displayed
inverse-trace identity.  Now let $Q_\mu^*$ attain the matrix infimum and set
$v_*=-(Q_\mu^*)^{-1}g$. The first-order optimality condition
for the convex function $Q\mapsto\tr(Q^{-1}S_g)$ over $\cK$ gives, for every
$Q\in\cK$,
$$
\tr\!\left(Q\,(Q_\mu^*)^{-1}S_g(Q_\mu^*)^{-1}\right)
\le
\tr\!\left(Q_\mu^*(Q_\mu^*)^{-1}S_g(Q_\mu^*)^{-1}\right).
$$
Since
$$
\int v_*v_*^\top\,d\mu=(Q_\mu^*)^{-1}S_g(Q_\mu^*)^{-1},
$$
this says precisely that $Q_\mu^*$ is active in the support representation of
$\gamma$ at the covariance of $v_*$. Therefore
$$
\int g\cdot v_*\,d\mu+\frac12\gamma\!\left(\int v_*v_*^\top\,d\mu\right)
=
-\frac12\tr\!\bigl((Q_\mu^*)^{-1}\mathcal{S}_\mu(g)\bigr),
$$
which equals the common minimax value. Thus, $v_*$ minimizes the Measure-LMO
objective.
\end{proof}

\paragraph{Proof of Proposition~\ref{cor:particle}.}
\begin{proof}
At distinct atoms, prescribe $w(x_i)=H_i$ for an arbitrary matrix $H$.
The transport chain rule gives
\[
 D\mF(X)[H]
 =n\left.\frac{d}{dh}\MF((\Id+hw)_\#\mu_X)\right|_{h=0}
 =\sum_i g_{\mu_X}(x_i)\cdot H_i .
\]
Thus $\nabla\mF(X)=G$, where $G_i=g_{\mu_X}(x_i)^\top$.
At coincident atoms the corresponding rows of $\nabla\mF$ are equal
by permutation symmetry; moving the whole coincident group in the
chain rule identifies each row with the same force.
The transport expansion in Theorem~\ref{prop:formalGF} is a sufficient
condition for this chain rule.

For distinct atoms, testing the weak continuity equation near each atom
identifies its velocity almost everywhere. The same identification applies
to coincident groups assigned a common velocity.
The identity $\MN_{\mu_X}(v)^2=n^{-1}\mN(V)^2$ and pairing
$n^{-1}\tr(G^\top V)$ reduce Definition~\ref{def:duality} to
$$
\operatorname*{argmin}_V\left\{\tr(G^\top V)+\frac12\mN(V)^2\right\}=\mJ_\gamma(G).
$$
With matching selections, stacking the velocities therefore gives
$\dot X_t=\mJ_\gamma(\nabla\mF(X_t))$, exactly
ODE~\eqref{eq:ode-basic}. Conversely, this ODE yields the weak continuity
equation by differentiating test functions along the particles.
For the unscaled
objective $\widetilde{\mF}(X)=\MF(\mu_X)$, one instead has
$\nabla\widetilde{\mF}(X)=n^{-1}G_t$ and
$\dot X_t=n^{-1}\mJ_\gamma(G_t)$ by one-homogeneity. This is the same weak
measure equation after rescaling time by $n$.
\end{proof}

\subsection{Comparison between Classical and Spectral Wasserstein Distances}
\label{app:static-details}

\paragraph{Proof of Proposition~\ref{prop:bounds}.}
\begin{proof}
Since $\gamma$ is a norm on the finite-dimensional space $\mathbb S^d$, its
restriction to the compact set $\{S\succeq 0:\tr(S)=1\}$ is continuous and
strictly positive. Hence there exist constants
$0<c_\gamma\le C_\gamma<\infty$ such that
$c_\gamma\tr(S)\le \gamma(S)\le C_\gamma\tr(S)$ for every $S\succeq0$.

Let $\pi\in\Pi(\mu,\nu)$ be any coupling, and define its displacement
covariance by
\[
    \Sigma_\pi
    \coloneqq
    \int_{\R^d\times\R^d} (x-y)(x-y)^\top\,d\pi(x,y).
\]
Then $\Sigma_\pi\succeq0$, and therefore
\[
    c_\gamma\tr(\Sigma_\pi)
    \le
    \gamma(\Sigma_\pi)
    \le
    C_\gamma\tr(\Sigma_\pi).
\]
Moreover,
\[
    \tr(\Sigma_\pi)
    =
    \int_{\R^d\times\R^d} \tr\bigl((x-y)(x-y)^\top\bigr)\,d\pi(x,y)
    =
    \int_{\R^d\times\R^d} |x-y|^2\,d\pi(x,y).
\]
Thus, for every coupling $\pi\in\Pi(\mu,\nu)$,
\[
    c_\gamma
    \int |x-y|^2\,d\pi(x,y)
    \le
    \gamma(\Sigma_\pi)
    \le
    C_\gamma
    \int |x-y|^2\,d\pi(x,y).
\]
Taking infima over $\pi\in\Pi(\mu,\nu)$ gives
\[
    c_\gamma \Wtwo(\mu,\nu)^2
    \le
    \Wg(\mu,\nu)^2
    \le
    C_\gamma \Wtwo(\mu,\nu)^2,
\]
where we used the definitions of $\Wtwo$ and $\Wg$. Taking square roots
yields
\[
    \sqrt{c_\gamma}\,\Wtwo(\mu,\nu)
    \le
    \Wg(\mu,\nu)
    \le
    \sqrt{C_\gamma}\,\Wtwo(\mu,\nu).
\]
\end{proof}

\paragraph{Schatten constants and sharpness.}

For Schatten norms, $d^{1/p-1}\,\Wtwo(\mu,\nu)^2\le \mathsf W_{\gamma_p}(\mu,\nu)^2\le \Wtwo(\mu,\nu)^2$, equivalently $c_{\gamma_p}=d^{1/p-1}$ and $C_{\gamma_p}=1$. In particular, $(c_{\gamma_2},C_{\gamma_2})=(d^{-1/2},1)$ and $(c_{\gamma_\infty},C_{\gamma_\infty})=(d^{-1},1)$. These constants follow from $\norm{\lambda}_{\ell_p}\le \norm{\lambda}_{\ell_1}\le d^{1-1/p}\norm{\lambda}_{\ell_p}$.

The lower Schatten bound is sharp. For $\gamma=\gamma_p$, equality in the lower bound is attained by isotropic Gaussian pairs $\mu=\mathcal N(0,\alpha I_d)$ and $\nu=\mathcal N(0,\beta I_d)$, $\alpha,\beta>0$. By the Gaussian computation of Appendix~\ref{app:gaussians}, the displacement covariance is $(\sqrt\alpha-\sqrt\beta)^2I_d$, hence $\mathsf W_{\gamma_p}(\mu,\nu)^2=d^{1/p-1}\Wtwo(\mu,\nu)^2$.

\paragraph{Proof of Theorem~\ref{thm:minmax}.}
\begin{proof}
Let $\pi\in\Pi(\mu,\nu)$ and set
\[
    \Sigma_\pi
    \coloneqq
    \int_{\R^d\times\R^d}
    (y-x)(y-x)^\top\,d\pi(x,y).
\]
Since $\mu,\nu\in\cP_2(\R^d)$, the matrix $\Sigma_\pi$ is well-defined and
belongs to $\mathbb S_+^d$. By the support representation of $\gamma$,
\[
    \gamma(\Sigma_\pi)
    =
    \max_{Q\in\cK}\tr(Q\Sigma_\pi).
\]
Moreover, for every $Q\in\cK$,
\[
    \tr(Q\Sigma_\pi)
    =
    \int_{\R^d\times\R^d}
    (y-x)^\top Q(y-x)\,d\pi(x,y).
\]
Therefore
\[
    \Wg(\mu,\nu)^2
    =
    \inf_{\pi\in\Pi(\mu,\nu)}
    \max_{Q\in\cK}
    \int
    (y-x)^\top Q(y-x)\,d\pi(x,y).
\]

It remains to justify the exchange of $\inf_\pi$ and $\max_Q$. Define
\[
    F(\pi,Q)
    \coloneqq
    \int_{\R^d\times\R^d}
    (y-x)^\top Q(y-x)\,d\pi(x,y).
\]
The coupling set $\Pi(\mu,\nu)$ is convex and weakly compact, since its
marginals are fixed. The set $\cK$ is convex and compact by assumption. For
fixed $\pi$, the map $Q\mapsto F(\pi,Q)$ is affine, hence both concave and
convex, and is continuous on $\cK$. For fixed $Q\succeq0$, the cost
$(x,y)\mapsto (y-x)^\top Q(y-x)$ is nonnegative and lower semicontinuous.
Hence $\pi\mapsto F(\pi,Q)$ is lower semicontinuous for weak convergence of
probability measures, by the Portmanteau theorem. It is also affine in
$\pi$.

Thus the lower-semicontinuous form of Sion's minimax theorem~\cite{sion1958general}  applies and
gives
\[
    \inf_{\pi\in\Pi(\mu,\nu)}
    \max_{Q\in\cK} F(\pi,Q)
    =
    \max_{Q\in\cK}
    \inf_{\pi\in\Pi(\mu,\nu)} F(\pi,Q).
\]
The only point requiring care is that the quadratic cost is not bounded on
$\R^d\times\R^d$. This is precisely why one uses the lower-semicontinuous
version of the minimax theorem rather than the bounded continuous one. 
Consequently,
\[
    \Wg(\mu,\nu)^2
    =
    \max_{Q\in\cK}
    \inf_{\pi\in\Pi(\mu,\nu)}
    \int
    (y-x)^\top Q(y-x)\,d\pi(x,y).
\]
Finally, by definition of the quadratic optimal transport cost associated
with $Q$,
\[
    \WQ{Q}(\mu,\nu)^2
    =
    \inf_{\pi\in\Pi(\mu,\nu)}
    \int
    (y-x)^\top Q(y-x)\,d\pi(x,y).
\]
Therefore
\[
    \Wg(\mu,\nu)^2
    =
    \max_{Q\in\cK}\WQ{Q}(\mu,\nu)^2,
\]
which proves the value formula. Finally, if $(Q^\star,\pi^\star)$ is a saddle
point, then $Q^\star$ maximizes
$Q\mapsto\tr(Q\Sigma_{\pi^\star})$ over $\cK$. Since $\gamma$ is the support
function of $\cK$ on $\mathbb S_+^d$, this is equivalent to the cone
subgradient relation
$Q^\star\in\partial\gamma(\Sigma_{\pi^\star})$, proving the stated
optimality relation.
\end{proof}

\begin{proof}[Proof of Corollary~\ref{cor:kantorovich-potentials}]
For each fixed $Q\succeq0$, Kantorovich duality for the continuous,
nonnegative quadratic cost identifies the inner supremum with
$\WQ{Q}(\mu,\nu)^2$. This remains valid for singular $Q$: equivalently,
one may transport the pushforwards by $Q^{1/2}$. Apply
Theorem~\ref{thm:minmax}. The outer maximum is attained by compactness of
$\cK$ and upper semicontinuity of the infimum of linear functions of $Q$.
\end{proof}

\subsection{Metric property from the robust-cost formulation}
\label{subsec:metric-static}

The max--min representation of Theorem~\ref{thm:minmax} gives a direct proof
of the triangle inequality when $\gamma$ is monotone. This avoids using the
Benamou--Brenier formulation for this specific purpose.

\begin{proposition}[Static proof of the metric property]
\label{prop:metric-static}
Assume that $\gamma$ is monotone on $\mathbb S_+^d$. Then $\Wg$ is a distance
on $\cP_2(\R^d)$ and induces the same topology as $\Wtwo$.
\end{proposition}

\begin{proof}
By Proposition~\ref{prop:psdpolar}, monotonicity allows us to choose a
representing set $\cK\subset \mathbb S_+^d$ such that
\[
    \gamma(S)=\max_{Q\in\cK}\tr(QS),
    \qquad S\succeq0.
\]
Hence Theorem~\ref{thm:minmax} gives
\[
    \Wg(\mu,\nu)
    =
    \max_{Q\in\cK}\WQ{Q}(\mu,\nu).
\]
For a fixed $Q\in\cK$, write
\[
    d_Q(x,y)\coloneqq \norm{Q^{1/2}(x-y)}_2 .
\]
Since $Q\succeq0$, $d_Q$ is a seminorm distance on $\R^d$, and therefore
$\WQ{Q}$ is the corresponding quadratic Wasserstein pseudodistance. In
particular, for all $\mu_0,\mu_1,\mu_2\in\cP_2(\R^d)$,
\[
    \WQ{Q}(\mu_0,\mu_2)
    \le
    \WQ{Q}(\mu_0,\mu_1)+\WQ{Q}(\mu_1,\mu_2).
\]
Taking the maximum over $Q\in\cK$ yields
\[
\begin{aligned}
    \Wg(\mu_0,\mu_2)
    &=
    \max_{Q\in\cK}\WQ{Q}(\mu_0,\mu_2) \\
    &\le
    \max_{Q\in\cK}
    \left\{
        \WQ{Q}(\mu_0,\mu_1)+\WQ{Q}(\mu_1,\mu_2)
    \right\} \\
    &\le
    \max_{Q\in\cK}\WQ{Q}(\mu_0,\mu_1)
    +
    \max_{Q\in\cK}\WQ{Q}(\mu_1,\mu_2) \\
    &=
    \Wg(\mu_0,\mu_1)+\Wg(\mu_1,\mu_2).
\end{aligned}
\]
Symmetry is immediate from the definition. Moreover $\Wg(\mu,\mu)=0$, and
separation follows from Proposition~\ref{prop:bounds}: if $\Wg(\mu,\nu)=0$,
then $\Wtwo(\mu,\nu)=0$, hence $\mu=\nu$. The same comparison bounds also show
that $\Wg$ and $\Wtwo$ induce the same topology on $\cP_2(\R^d)$.
\end{proof}

\begin{remark}[Relation with subspace robust Wasserstein distances]
\label{rem:paty-cuturi-metric-proof}
The proof is the same robust-cost argument used by
\citet{PatyCuturi2019} for Ky Fan gauges. The only point needed in their
argument is that the maximizing cost matrices are positive semidefinite. By
Proposition~\ref{prop:psdpolar}, this positivity holds for any monotone gauge
$\gamma$, not only for Ky Fan gauges. Thus the Paty--Cuturi triangle-inequality
proof extends verbatim to the present monotone-gauge setting.
\end{remark}

\subsection{Large-mean asymptotics for Schatten geometries}
\label{subsec:large-mean-schatten}

We now describe the behavior of $\Wg$ when the displacement between the means
is much larger than the centered fluctuations of the measures. This regime is
useful to interpret the optimizer in the max-min formula: for Schatten
geometries with $p>1$, the optimal certificate becomes increasingly aligned
with the direction of the mean displacement, and the residual transport
problem is effectively measured only along this direction.

\begin{proposition}[Large-mean asymptotics for Schatten geometries]
\label{prop:large-mean-schatten}
Let $\gamma=\gamma_p$ be the Schatten-$p$ norm, with $p\in[1,\infty]$. Fix
centered measures $\bar\mu,\bar\nu\in\cP_2(\R^d)$ and a unit vector
$u\in\R^d$. For each $R>0$, let $\mu_R,\nu_R$ have centered laws
$\bar\mu,\bar\nu$ and means whose difference is $Ru$. Then, as
$R\to\infty$, the following asymptotics hold.
For $p=1$, 
\[
    \Wg(\mu_R,\nu_R)^2
    =
    R^2+\Wtwo(\bar\mu,\bar\nu)^2 .
\]
For $1<p\le\infty$,
\[
    \Wg(\mu_R,\nu_R)^2
    =
    R^2+\WQ{uu^\top}(\bar\mu,\bar\nu)^2+o(1).
\]
\end{proposition}

\begin{proof}
We first separate the mean displacement from the centered transport. Fix
$R>0$, abbreviate $\mu=\mu_R$, $\nu=\nu_R$, set
$m_\mu=\int x\,d\mu(x)$, $m_\nu=\int y\,d\nu(y)$, and let
$\Delta=m_\nu-m_\mu=Ru$. Fix
$Q\succeq0$. If $\pi\in\Pi(\mu,\nu)$ and if $\bar\pi$ denotes the corresponding
coupling of $\bar\mu$ and $\bar\nu$, then writing
$x=m_\mu+\bar x$ and $y=m_\nu+\bar y$ gives
\[
    y-x=\Delta+(\bar y-\bar x).
\]
Therefore
\[
\begin{aligned}
    \int (y-x)^\top Q(y-x)\,d\pi(x,y)
    &=
    \Delta^\top Q\Delta
    +
    2\Delta^\top Q
    \int(\bar y-\bar x)\,d\bar\pi(\bar x,\bar y) \\
    &\qquad
    +
    \int(\bar y-\bar x)^\top Q(\bar y-\bar x)\,
    d\bar\pi(\bar x,\bar y).
\end{aligned}
\]
The middle term vanishes because $\bar\mu$ and $\bar\nu$ have zero mean. Taking
the infimum over couplings yields the exact decomposition
\[
    \WQ{Q}(\mu,\nu)^2
    =
    \Delta^\top Q\Delta+\WQ{Q}(\bar\mu,\bar\nu)^2.
\]
By Theorem~\ref{thm:minmax},
\[
    \Wg(\mu,\nu)^2
    =
    \max_{Q\in\mathcal K_{\gamma_p}}
    \left\{
        \Delta^\top Q\Delta+\WQ{Q}(\bar\mu,\bar\nu)^2
    \right\}.
\]

For $p=1$, we use the reduced representing set
$\mathcal K_{\gamma_1}=\{\Id\}$. Then
\[
    \Wg(\mu,\nu)^2
    =
    \Delta^\top\Delta+\WQ{\Id}(\bar\mu,\bar\nu)^2
    =
    R^2+\Wtwo(\bar\mu,\bar\nu)^2,
\]
as claimed.

Assume now that $1<p\le\infty$, and let $p'$ be the conjugate exponent, with
$p'=1$ when $p=\infty$. We use the positive dual Schatten representing set
\[
    \mathcal K_{\gamma_p}
    =
    \{Q\succeq0:\norm{Q}_{S_{p'}}\le1\}.
\]
Since $\Delta=Ru$, the max-min formula becomes
\[
    \Wg(\mu,\nu)^2
    =
    \max_{Q\in\mathcal K_{\gamma_p}}
    \left\{
        R^2 u^\top Q u+\WQ{Q}(\bar\mu,\bar\nu)^2
    \right\}.
\]
By Schatten duality, applied to the rank-one matrix $uu^\top$,
\[
    \max_{Q\in\mathcal K_{\gamma_p}} u^\top Q u
    =
    \max_{\substack{Q\succeq0\\ \norm{Q}_{S_{p'}}\le1}}
    \tr(Q\,uu^\top)
    =
    \norm{uu^\top}_{S_p}
    =
    1.
\]
For $1<p<\infty$, the equality case in H\"older's inequality is unique and gives
$Q=uu^\top$. For $p=\infty$, the constraint is $\tr(Q)\le1$, and
$u^\top Q u\le\tr(Q)\le1$; equality again forces $Q=uu^\top$. Thus the exposed
face of $\mathcal K_{\gamma_p}$ in the direction $uu^\top$ is the singleton
$\{uu^\top\}$ for every $1<p\le\infty$.

Let $Q_R$ be an optimal certificate for the previous maximization problem.
Since $\mathcal K_{\gamma_p}$ is compact, any sequence $R_n\to\infty$ admits a
subsequence such that $Q_{R_n}\to Q_\infty\in\mathcal K_{\gamma_p}$. Dividing
the optimality relation by $R_n^2$ and passing to the limit shows that
$Q_\infty$ maximizes $Q\mapsto u^\top Q u$ over $\mathcal K_{\gamma_p}$. By
the uniqueness of the exposed maximizer, $Q_\infty=uu^\top$. Hence every
maximizing certificate satisfies
\[
    Q_R\longrightarrow uu^\top
    \qquad\text{as }R\to\infty.
\]
Set $b(Q)=\WQ{Q}(\bar\mu,\bar\nu)^2$. This map is Lipschitz on
$\mathbb S_+^d$: indeed, comparing the two infima and using
$|(\bar y-\bar x)^\top(Q-Q')(\bar y-\bar x)|
\le\|Q-Q'\|_{\rm op}|\bar y-\bar x|^2$ gives
\[
 |b(Q)-b(Q')|
 \le
 2\|Q-Q'\|_{\rm op}
 \left(\int|\bar x|^2\,d\bar\mu
       +\int|\bar y|^2\,d\bar\nu\right).
\]
Consequently,
\[
    b(Q_R)
    \longrightarrow
    b(uu^\top).
\]
Convergence of the certificates alone would not yet control the
$R^2$-weighted term. Optimality of $Q_R$, compared with the admissible choice
$uu^\top$, supplies the missing estimate:
\[
 0\le R^2(1-u^\top Q_Ru)
 \le b(Q_R)-b(uu^\top)=o(1).
\]
Substituting both estimates into the max--min formula gives
\[
    \Wg(\mu,\nu)^2
    =
    R^2+b(uu^\top)+o(1).
\]
Taking square roots yields
\[
    \Wg(\mu,\nu)
    =
    R+\frac{1}{2R}\WQ{uu^\top}(\bar\mu,\bar\nu)^2+o(R^{-1}),
\]
which completes the proof.
\end{proof}

The proposition shows that all Schatten geometries have the same leading
large-mean behavior, namely $\Wg(\mu_R,\nu_R)\sim R$. The difference
appears at the next order and in the optimizer selected by the max-min
formula. For the reduced trace geometry $p=1$, the certificate is fixed and
isotropic, $Q=\Id$, so the centered correction is the usual Wasserstein cost
$\Wtwo(\bar\mu,\bar\nu)^2$. By contrast, for every $p>1$, the optimal
certificate aligns with the normalized mean displacement
$u$, and the centered correction becomes
$\WQ{uu^\top}(\bar\mu,\bar\nu)^2$. This is exactly the one-dimensional
$\Wtwo^2$ cost between the projected marginals
$(u^\top\!\cdot)_\#\bar\mu$ and $(u^\top\!\cdot)_\#\bar\nu$. Thus the first
centered correction depends only on fluctuations along $u$; transverse
fluctuations do not enter at this order.

\section{Connection with Muon, Orthogonalization, and Momentum}
\label{sec:muon-extensions}
This appendix complements Section~\ref{sec:flows} by linking the matrix ODE for $p=\infty$ to practical Muon variants used at scale, and by showing how Newton--Schulz orthogonalization and momentum fit in the same normalized-flow language.

\paragraph{Taking into account Newton--Schulz orthogonalization.}
Setting $p=\infty$ in Proposition~\ref{prop:schatten} yields the LMO \eqref{eq:muon-selector}, so the spectral matrix flow
$$
\dot X_t=\mJ_\infty(\nabla \mF(X_t))
$$
is exactly the vanishing-step-size, vanishing-momentum, full-batch limit of idealized Muon in the normalization convention used throughout this paper.

An important practical feature of Muon is that one does not compute the exact polar factor $UW^\top$ by an SVD at every step. To make the method scalable and GPU-friendly, practical implementations replace $UW^\top = G(G^\top G)^{\dagger/2}$ by a few Newton--Schulz iterations \citep{KellerJordan2024,LiuSuYaoEtAl2025}. After rescaling $G$ by a scalar $\alpha(G)>0$ chosen so that the spectrum of $\alpha(G)^{-2}G^\top G$ lies in $[0,1]$, the classical Newton--Schulz iteration reads
$$
Z_{k+1}
=
\frac12 Z_k(3I-Z_k^\top Z_k),
\qquad
Z_0=\alpha(G)^{-1}G.
$$
For the discussion below, let us freeze the scale and regard $\alpha>0$ as fixed a priori. Then after one iteration one obtains
$$
\mathcal U_{\alpha,1}(G)
\coloneqq
-Z_1
=
-\frac{1}{2\alpha}
\,G\left(3I-\alpha^{-2}G^\top G\right),
$$
and after $k$ iterations one obtains more generally a polynomial update of the
form
$$
\mathcal U_{\alpha,k}(G)
=
-G\,p_{\alpha,k}\!\left(G^\top G\right),
$$
where $p_{\alpha,k}$ is a matrix polynomial approximating $s\mapsto s^{-1/2}$ on the spectral interval of interest. In this fixed-$\alpha$ setting, the polynomial orthogonalization part does admit a scalar generating function: if
$$
p_{\alpha,k}(s)=\sum_{\ell=0}^m a_\ell s^\ell,
$$
then
$$
\Psi_{\alpha,k}(S)\coloneqq \sum_{\ell=0}^m \frac{a_\ell}{\ell+1}\tr(S^{\ell+1})
$$
satisfies
$$
\nabla \Psi_{\alpha,k}(S)=p_{\alpha,k}(S)
\qquad\text{and}\qquad
\mathcal U_{\alpha,k}(G)
=
-\nabla_G\!\left(\frac12 \Psi_{\alpha,k}(G^\top G)\right).
$$
For one Newton--Schulz step this gives the explicit formula
$$
\Psi_{\alpha,1}(S)
=
\frac{3}{2\alpha}\tr(S)-\frac{1}{4\alpha^3}\tr(S^2).
$$
This function is not a norm, and it is not globally monotone on $\mathbb S_+^d$, because
$$
\nabla \Psi_{\alpha,1}(S)=\frac{3}{2\alpha}I-\frac{1}{2\alpha^3}S
$$
ceases to be positive semidefinite once $S$ has eigenvalues larger than $3\alpha^2$. It is nevertheless monotone on the truncated spectral region
$$
0\preceq S\preceq 3\alpha^2 I,
$$
and in particular on the standard Newton--Schulz regime $0\preceq S\preceq \alpha^2 I$.

Thus a fixed-$\alpha$ Newton--Schulz update is generated by the spectral
potential $\Psi_{\alpha,k}$. It must be interpreted on the \emph{force}
side, not substituted for the primal velocity gauge $\gamma$.
Although this polynomial is generally neither convex in $S$ nor homogeneous,
there is a precise convex-dissipation interpretation on the stable region.

\begin{proposition}[Fixed-scale Newton--Schulz dissipation]
\label{prop:ns-dissipation}
For the classical iteration above with fixed $\alpha$ and $k\ge1$, set
$D_{\alpha,k}^*(G)=\frac12\Psi_{\alpha,k}(G^\top G)$.
On $\{G:\|G\|_{\rm op}\le\alpha\}$ this is convex, and
$\Psi_{\alpha,k}$ is Loewner-monotone on $0\preceq S\preceq\alpha^2I$.
There is a global convex extension $\widehat D_{\alpha,k}^*$ agreeing with
$D_{\alpha,k}^*$ on this region. With
$D_{\alpha,k}=(\widehat D_{\alpha,k}^*)^*$, for
$\|G\|_{\rm op}\le\alpha$ one has
\[
 \mathcal U_{\alpha,k}(G)=-\nabla\widehat D_{\alpha,k}^*(G)
 \in\operatorname*{argmin}_V
 \{\langle G,V\rangle_F+D_{\alpha,k}(V)\}.
\]
\end{proposition}
\begin{proof}
Let $h_0(s)=s/\alpha$ and
$h_{j+1}(s)=(3h_j(s)-h_j(s)^3)/2$ for $0\le s\le\alpha$.
Induction gives $0\le h_j\le1$ and $h_j'\ge0$ there, with
$h_j(\alpha)=1$. The singular-value response is $h_k$, so
$D_{\alpha,k}^*(G)=\sum_i\int_0^{\sigma_i(G)}h_k(s)\,ds$.
Extend $h_k$ by $1$ for $s\ge\alpha$. Its primitive is convex and
nondecreasing; denote the resulting even convex spectral function by
$\widehat D_{\alpha,k}^*$.
Its gradient agrees with $-\mathcal U_{\alpha,k}$ in the stable region.
Moreover, $p_{\alpha,k}(s)=h_k(\sqrt{s})/\sqrt{s}\ge0$ there
(with the continuous value at $s=0$),
which proves Loewner monotonicity by integration along ordered matrix
segments. For $\|G\|_{\rm op}\le\alpha$, Fenchel duality gives
$\mathcal U_{\alpha,k}(G)\in\arg\min_V
\{\langle G,V\rangle_F+D_{\alpha,k}(V)\}$.
\end{proof}

For one step, the singular-value derivative of the dual dissipation is
$h_1(s)=3s/(2\alpha)-s^3/(2\alpha^3)$, with
$h_1'(s)=3(1-s^2/\alpha^2)/(2\alpha)\ge0$ on the stable region.
This illustrates why convexity of the dissipation, rather than convexity of
$\Psi$ as a function of the covariance, is the relevant condition.
The same spectral construction works for fields $g\in L^2(\mu;\R^d)$,
using $S_g=\int gg^\top d\mu$ and the Hilbert-space operator $T_g$ from
Proposition~\ref{prop:measure-lmo-reduction}.
It gives a convex dissipation and an energy-decreasing generalized gradient
system while $S_{g_{\mu_t}}\preceq\alpha^2I$.
Here $\alpha$ bounds the square root of the \emph{normalized} force covariance.
For $\mu=\mu_X$ and $M_i=g_\mu(x_i)^\top$, the stacked field velocity is
\begin{equation}\label{eq:ns-empirical-scaling}
 -M p_{\alpha,k}(M^\top M/n)
 =\sqrt n\,\mathcal U_{\sqrt n\alpha,k}(M),
 \qquad M^\top M/n\preceq\alpha^2I.
\end{equation}
Indeed, rescaling $M/\sqrt n$ by $\alpha$ and $M$ by $\sqrt n\alpha$
gives the same initial Newton--Schulz iterate. The field convention
therefore differs from applying $\mathcal U_{\alpha,k}$ directly to $M$;
the homogeneous clock conversion of Proposition~\ref{cor:particle}
cannot be reused at fixed $\alpha$.

What does not follow is a gradient flow for a fixed $\Wg$ metric:
finite-step fixed-scale maps are not one-homogeneous in the force, whereas
the duality map of any squared norm is. A dynamic transport \emph{action}
can be defined using the conjugate dissipation, but the squared-distance BB
theorem in this paper does not apply to it. Monotonicity alone does not fix
this homogeneity obstruction. A priori control keeping the evolving force
covariance in the stable region is an additional requirement. Adaptive
rescaling and the nonclassical polynomial coefficients used in practical
implementations require their own analysis.

\paragraph{Taking into account momentum.}
A natural momentum extension, in the spirit of Muon with exponential moving average, is obtained by introducing an auxiliary matrix variable $A_t$ and considering
\begin{equation}
\label{eq:matrix-momentum}
\tau \dot A_t + A_t = \nabla \mF(X_t),
\qquad
\dot X_t = \mJ_\gamma(A_t),
\end{equation}
where $\tau\ge 0$ is the memory time scale. The variable $A_t$ plays the role of a continuous-time exponentially averaged gradient. When $\tau=0$, one recovers exactly the first-order spectral flow studied above.

For finite Schatten geometries ($1\le p<\infty$), this system is equivalent
to the second-order inertial equation
$$
\tau \frac{d}{dt}\bigl(\mI_p(\dot X_t)\bigr) + \mI_p(\dot X_t) + \nabla \mF(X_t)=0,
$$
where
$$
\mI_p(V)
\coloneqq
\norm{V}_{S_r}^{2-r}\,U\diag(\sigma_i^{r-1})W^\top,
\qquad r=2p,
$$
for $V=U\diag(\sigma_i)W^\top$. In particular, for $p=1$ one has $\mI_1(V)=V$ and one recovers the heavy-ball equation
$$
\tau \ddot X_t + \dot X_t + \nabla \mF(X_t)=0.
$$
We set $\mI_p(0)=0$.
At the endpoint $p=\infty$, \eqref{eq:matrix-momentum} reads
$$
\tau \dot A_t + A_t = \nabla \mF(X_t),
\qquad
\dot X_t = -\norm{A_t}_{S_1}\,U_tW_t^\top,
$$
which is the exact-polar continuous-time Muon dynamics with momentum.

The corresponding mean-field description naturally lives on phase space. Let
$$
\eta_t\in\cP_2(\R^d\times\R^d)
$$
be the law of pairs $(x,a)$, where $x$ is the particle position and $a$ is the exponentially averaged force. Its spatial marginal is
$$
\mu_t=(\pi_x)_\#\eta_t.
$$
For a functional $\MF$ on measures, write its Wasserstein force as
$$
g_\mu(x)=\nabla_x\frac{\delta\MF}{\delta\mu}(x).
$$
Given a phase-space law $\eta$, define its normalized spatial velocity by the
phase-space Measure-LMO
\[
 u_\eta\in\operatorname*{argmin}_{u\in L^2(\eta;\R^d)}
 \left\{
   \int a\cdot u(x,a)\,d\eta(x,a)
   +\frac12\gamma\!\left(\int uu^\top\,d\eta\right)
 \right\}.
\]
The operator proof of Proposition~\ref{prop:measure-lmo-reduction} applies
verbatim with phase space as the underlying probability space and computes
$u_\eta$ from the momentum covariance
$S(\eta)=\int aa^\top\,d\eta$. In particular, this definition remains valid
for singular selectors such as $p=\infty$. The phase-space characteristics
are
\begin{equation}
\label{eq:phase-momentum-characteristics}
\dot x_t=u_{\eta_t}(x_t,a_t),
\qquad
\tau\dot a_t=g_{\mu_t}(x_t)-a_t .
\end{equation}
Consequently, the law $\eta_t$ solves the Liouville-type equation
\begin{equation}
\label{eq:phase-momentum-liouville}
\partial_t\eta_t
+\operatorname{div}_x(\eta_t u_t)
+\frac1\tau\operatorname{div}_a\!\bigl(\eta_t(g_{\mu_t}(x)-a)\bigr)=0,
\end{equation}
where
$$
u_t(x,a)=u_{\eta_t}(x,a),
\qquad
\mu_t=(\pi_x)_\#\eta_t.
$$
When the inverse-trace problem in
Proposition~\ref{prop:dual-inverse-trace-selector} is attained at an
invertible $Q_\eta^*$, this constitutive relation reduces to the closed form
$u_\eta(x,a)=-(Q_\eta^*)^{-1}a$.
When $\tau=0$, the averaged force instantaneously relaxes to $a=g_{\mu_t}(x)$, and the spatial marginal formally recovers the first-order normalized PDE \eqref{eq:measure-spectral-flow}. For $\tau>0$, the dynamics is no longer a metric gradient flow for $\Wg$ on the position marginal alone; it is a damped kinetic system in which the spectral geometry enters through the constitutive relation between the phase-space force covariance and the spatial velocity.

\paragraph{Stochastic minibatches.}
An unbiased force estimate need not give an unbiased spectral velocity:
$\mathbb E[\MJ_\gamma^\mu(g_\mu+\xi)]$ generally differs from
$\MJ_\gamma^\mu(g_\mu)$ even if $\mathbb E\xi=0$. Bias after nonlinear
normalization can prevent convergence; sign-based counterexamples are given
by \citet{KarimireddyEtAl2019}. These examples do not establish failure for
every stochastic Muon variant, but rule out transferring unbiased-gradient
arguments without further analysis.
Under a fast-resampling closure with fixed-variance random force fields,
the candidate transport equation is
\[
 \partial_t\mu_t+\divv\!\left(
  \mu_t\,\mathbb E_\xi[\MJ_\gamma^{\mu_t}(g_{\mu_t}+\xi)]\right)=0.
\]
The covariance inside each LMO is that of the entire random force field;
shared minibatch noise is not independent across particles. There is no
automatic diffusion term in this scaling. Its gradient-flow structure is
open in general.

A different toy model adds independent Brownian noise \emph{after} selecting
the deterministic velocity. Its equation is
$\partial_t\mu+\divv(\mu\MJ_\gamma^\mu(g_\mu))
=\varepsilon\Delta\mu$.
For $\gamma=\tr$ this is the $\Wtwo$ flow of
$\MF+\varepsilon\operatorname{Ent}$, where
$\operatorname{Ent}(\mu)=\int\rho\log\rho\,dx$ for $d\mu=\rho\,dx$.
For a general spectral gauge, however,
$\MJ_\gamma^\mu(g_\mu+\varepsilon\nabla\log\rho)$ is not in general
$\MJ_\gamma^\mu(g_\mu)-\varepsilon\nabla\log\rho$.
Thus this same entropy-penalization interpretation does not carry over.

\paragraph{Block-wise normalization.}
For parameters $x=(u,v)$, separate normalization of the outer and inner
weight blocks corresponds to a block gauge. The choice depends on whether
each block uses its own squared-norm LMO or an unscaled polar update. For
$S=\left(\begin{smallmatrix}S_{11}&S_{12}\\ S_{21}&S_{22}\end{smallmatrix}\right)\succeq0$, define
$$
\gamma_{\rm blk}(S)\coloneqq \gamma_1(S_{11})+\gamma_2(S_{22}),
$$
where $\gamma_1$ and $\gamma_2$ are monotone gauges on the two positive
semidefinite block cones. This gauge is convex, positively homogeneous, and
positive on every nonzero PSD matrix. It is monotone because
$0\preceq S\preceq T$ implies $S_{ii}\preceq T_{ii}$ for each principal
block. Proposition~\ref{prop:gauge-norm-extension} therefore supplies a norm
extension to the whole symmetric space.
Furthermore, writing $v=(v_1,v_2)$ gives
$\MN_\mu(v)^2=\gamma_1(\int v_1v_1^\top d\mu)
+\gamma_2(\int v_2v_2^\top d\mu)$; the LMO objective separates into two
independent minimizations. For operator block gauges and a gradient matrix
partitioned as $G=(G_1,\ldots,G_B)$, the resulting update in block $b$ is
$-\|G_b\|_*P_b$, where $P_b$ is its compact polar factor. These individual
speed factors cannot generally be absorbed into a single time change.

To model raw blockwise polar updates with one common clock, instead use
\[
 \gamma_{\rm max}(S)=\max_{1\le b\le B}\lambda_{\max}(S_{bb}).
\]
This is again a positive monotone convex gauge. Its tangent norm is
$\mN(V)=\max_b\|V_b\|_{\rm op}$, with dual
$\mN^*(G)=\sum_b\|G_b\|_*$. The unit-ball formula
\eqref{eq:lmo-unit-ball} therefore gives the canonical selection
\[
 \mJ_{\rm max}(G)_b=-\left(\sum_a\|G_a\|_*\right)P_b,
\]
where $P_b=0$ for a zero-gradient block. Thus it follows the raw blockwise
polar trajectories up to a common time change away from stationarity.
Fixed positive block learning-rate factors $\eta_b$ are obtained by replacing
the maximum gauge by $\max_b\lambda_{\max}(S_{bb})/\eta_b^2$.
The static and dynamic theory and conditional gradient-flow theorem apply
to both constructions, for any finite number of blocks. Neither construction
replaces the joint measure by independent marginals: the static problem still
uses a common coupling, and the chosen block layout must match the modeled
parameterization.

\subsection{Time Scales and Discretization}
\label{sec:smoothness-discretization}
The choice of gauge fixes a time scale. Replacing $\gamma$ by $a\gamma$
multiplies $\Wg$ by $\sqrt a$ and divides the selected velocity by $a$.
Thus an arbitrary acceleration of a continuous-time curve is not itself an
optimization gain. Practical iteration counts additionally depend on
smoothness in the chosen geometry and on the discretization.

For a smooth matrix objective on a convex region $\Omega$, a precise
Lipschitz-gradient constant is
\begin{equation}\label{eq:spectral-smoothness}
 L_\gamma\coloneqq
 \sup_{X\ne Y\in\Omega}
 \frac{\mN^*(\nabla\mF(X)-\nabla\mF(Y))}{\mN(X-Y)}.
\end{equation}
If it is finite, integration along a segment gives the descent lemma
$\mF(X+V)\le\mF(X)+\langle\nabla\mF(X),V\rangle_F
+\frac12L_\gamma\mN(V)^2$.
For $V=\tau\mJ_\gamma(\nabla\mF(X))$, this yields
\[
 \mF(X+V)\le\mF(X)
 -\tau\left(1-\frac{L_\gamma\tau}{2}\right)
       \mN^*(\nabla\mF(X))^2,
\]
provided the segment remains in $\Omega$. Hence $0<\tau<2/L_\gamma$
is a sufficient descent condition, not a universal numerical accuracy or
optimal stability bound. A one-sided quadratic upper model is enough for
this conclusion; it need not imply the two-sided gradient bound
\eqref{eq:spectral-smoothness}.

For the unscaled trace and operator gauges on $\R^{n\times d}$, let
$r_n=\min(n,d)$. Since
$\|V\|_{\rm op}\le\|V\|_F\le\sqrt{r_n}\|V\|_{\rm op}$ and
$\|G\|_F\le\|G\|_*\le\sqrt{r_n}\|G\|_F$, one obtains on the same region
\[
 L_1\le L_\infty\le r_nL_1.
\]
The operator smoothness constant cannot be smaller under this common
normalization. This worst-case comparison does not order convergence in
iterations: the operator update also has the larger dissipation
$\|G\|_*^2$, and the directions and states visited differ.
For fixed particle dimension $d$, the factor remains bounded by $d$ as
$n\to\infty$; there is no dimension-independent bound when $d$ also grows.
Existence of a finite smoothness constant, for either geometry, is a
separate property of the objective. Neither ReLU losses nor arbitrary
smooth neural-network losses have a global such bound automatically.

A transport analogue, avoiding comparison of covectors at different
measures, is the upper estimate
\[
 \MF(\nu)\le\MF(\mu)+\int g_\mu(x)\cdot(y-x)\,d\pi
 +\frac{L_\gamma}{2}\gamma\!\left(\int(y-x)(y-x)^\top\,d\pi\right)
\]
for the couplings and states under consideration. Applied to
$\nu=(\Id+\tau v)_\sharp\mu$ with
$v\in\MJ_\gamma^\mu(g_\mu)$, it gives the same descent bound with
$\MN_\mu^*(g_\mu)^2$ in place of the matrix dual norm. Establishing this
estimate along a neural-network flow, and controlling discretization or
particle approximation errors, are beyond the present analysis. In
particular, a Lipschitz bound on the energy differential must not be
confused with a Lipschitz bound on the Muon velocity map, which can be
discontinuous at rank changes.

\section{Monge Maps and Conditional Brenier Theorem}
\label{app:monge}
This appendix collects the map-based statements associated with the static formulation of Section~\ref{sec:static}.

The next definition records the Monge restriction, which is useful for comparison but is not the correct notion in general.
\begin{definition}[Monge restriction]
\label{def:monge}
Define
$$
\WgM(\mu,\nu)^2
\coloneqq 
\inf_{T_\#\mu=\nu}
\gamma\!\left(
\int_{\R^d}(T(x)-x)(T(x)-x)^\top\,d\mu(x)
\right),
$$
with value $+\infty$ if no transport map exists.
\end{definition}

The next proposition simply records that the map-based problem is a restriction of the coupling-based one.
\begin{proposition}[Monge is a restriction]
\label{prop:monge}
For every $\mu,\nu\in\cP_2(\R^d)$,
$$
\Wg(\mu,\nu)\le \WgM(\mu,\nu).
$$
\end{proposition}

\begin{proof}
Every transport map $T$ induces the coupling $(\Id,T)_\#\mu$, and the static cost evaluated on this coupling is exactly the Monge cost associated with $T$.
\end{proof}

The following remark shows on a two-Dirac example that the relaxation can be genuinely strict.
\begin{remark}[Strictness of the Monge restriction]
Consider
$$
\mu=\frac12(\delta_{(-1,0)}+\delta_{(1,0)}),
\qquad
\nu=\frac12(\delta_{(0,-1)}+\delta_{(0,1)}).
$$
Any transport map from $\mu$ to $\nu$ is a bijection between the two atoms. The two possible displacement covariance matrices are
$$
\begin{pmatrix}
1 & 1\\
1 & 1
\end{pmatrix}
\qquad\text{or}\qquad
\begin{pmatrix}
1 & -1\\
-1 & 1
\end{pmatrix},
$$
so for the operator norm and Frobenius norm one gets
$$
\WgM(\mu,\nu)^2=2.
$$
By contrast, the split coupling assigning mass $1/4$ to each source-target pair has displacement covariance equal to the identity matrix, hence
$$
\Wg(\mu,\nu)^2=
1
\quad\text{for }\gamma=\norm{\cdot}_{S_\infty},
\qquad
\Wg(\mu,\nu)^2=
\sqrt{2}
\quad\text{for }\gamma=\norm{\cdot}_{S_2}.
$$
Therefore the inequality in Proposition~\ref{prop:monge} is strict already for two-point measures.
\end{remark}

The next remark explains that, for the restriction of an arbitrary matrix
norm to $\mathbb S_+^d$, the static cost need not be a bona fide metric.
\begin{remark}[Triangle inequality may fail for non-monotone norms]
\label{rem:notmetric}
For Dirac masses, the unique coupling gives
$$
\Wg(\delta_x,\delta_y)^2
=
\gamma\!\left((y-x)(y-x)^\top\right).
$$
Hence if $\Wg$ were always a distance, then the pointwise cost
$$
d_\gamma(x,y)
\coloneqq 
\sqrt{\gamma\!\left((y-x)(y-x)^\top\right)}
$$
would in particular have to define a metric on $\R^d$.

This fails for general norms on $\mathbb S_+^d$. Fix $d\ge 2$ and a parameter $M>2$, and define
$$
\gamma_{\mathrm{ns}}(S)
\coloneqq 
\tr(S)+M\sum_{1\le i<j\le d}\abs{S_{ij}},
\qquad S\succeq 0.
$$
This is the restriction to $\mathbb S_+^d$ of the weighted entrywise norm
$A\mapsto\sum_i|A_{ii}|+M\sum_{i<j}|A_{ij}|$ on $\mathbb S^d$. It is neither
spectral nor monotone. Indeed, in dimension two,
\[
 S=\begin{pmatrix}1&1\\1&1\end{pmatrix}
 \preceq
 T=2I,
 \qquad
 \gamma_{\rm ns}(S)=2+M>4=\gamma_{\rm ns}(T).
\]
For the displacement
$$
\Delta=y-x=(\Delta_1,\dots,\Delta_d)
$$
one has
$$
d_{\gamma_{\mathrm{ns}}}(x,y)^2
=
\norm{\Delta}_2^2
+
M\sum_{1\le i<j\le d}\abs{\Delta_i\Delta_j}.
$$
Taking
$$
x_0=0,
\qquad
x_1=e_1,
\qquad
x_2=e_1+e_2
$$
gives
$$
d_{\gamma_{\mathrm{ns}}}(x_0,x_1)=1,
\qquad
d_{\gamma_{\mathrm{ns}}}(x_1,x_2)=1,
\qquad
d_{\gamma_{\mathrm{ns}}}(x_0,x_2)=\sqrt{2+M}>2,
$$
so the triangle inequality fails.

By contrast, if $\gamma$ is monotone on $\mathbb S_+^d$, then Proposition~\ref{prop:psdpolar} allows us to choose $\cK\subset \mathbb S_+^d$, and therefore
$$
d_\gamma(x,y)
=
\sup_{Q\in\cK}\norm{Q^{1/2}(x-y)}_2,
$$
which is a supremum of seminorms and, by positivity of $\gamma$, a norm of
$x-y$. For general measures, however, this pointwise argument does not control
the cross terms created by gluing couplings. The full metric property of $\Wg$
follows either from the robust static argument in
Proposition~\ref{prop:metric-static} or from the dynamic formulation in
Corollary~\ref{cor:metric}.
\end{remark}

The following proposition explains when an optimal coupling is induced by a Monge map of Brenier type.
\begin{proposition}[Conditional Brenier theorem]
\label{cor:brenier}
Assume $\mu$ is absolutely continuous. If a dual maximizing matrix
$Q_*\in\cK$ in Theorem~\ref{thm:minmax} is positive definite, then every
optimal coupling for $\Wg(\mu,\nu)$ is induced by a map. More precisely,
there exists a convex function $u$ such that
$$
T_*(x)=Q_*^{-1/2}\nabla u(Q_*^{1/2}x)
$$
is optimal.
\end{proposition}

Without monotonicity, a polar certificate need not be positive semidefinite
and may even be indefinite. Under the standing monotonicity assumption,
however,
Proposition~\ref{prop:psdpolar} allows one to choose the representing set
$\cK$ inside $\mathbb S_+^d$; the active matrix is then automatically positive
semidefinite, and the remaining nondegeneracy assumption is simply that it be
invertible. For the trace gauge $p=1$, one may use the reduced set
$\cK=\{I\}$, so the condition is automatic. For $1<p<\infty$, the canonical
active matrix is
$\Sigma_{\pi^\star}^{p-1}/\|\Sigma_{\pi^\star}\|_{S_p}^{p-1}$ and is therefore
positive definite whenever the optimal displacement covariance is. The
endpoint $p=\infty$ is sharply different: in that case
$\gamma(\Sigma)=\lambda_{\max}(\Sigma)$ and $Q_*$ is supported on the top
eigenspace of $\Sigma_{\pi^\star}$. When the leading eigenvalue is simple,
$Q_*$ is rank one. The associated degenerate quadratic cost does not control
directions orthogonal to the active eigenspace, so Brenier's uniqueness
argument no longer applies. Monge optimizers may still exist, but the dual
certificate alone neither selects one nor implies that every optimizer is
map-induced.

\begin{proof}
Fix such a positive definite dual maximizer $Q_*$ and let $\pi_*$ be any
primal optimizer. With
$F(\pi,Q)=\int(y-x)^\top Q(y-x)\,d\pi$, minimax equality gives
\[
 F(\pi_*,Q_*)
 \le \max_{Q\in\cK}F(\pi_*,Q)
 =\Wg(\mu,\nu)^2
 =\min_{\pi\in\Pi(\mu,\nu)}F(\pi,Q_*)
 \le F(\pi_*,Q_*).
\]
Thus $\pi_*$ is optimal for the quadratic cost induced by $Q_*$. To identify
this optimizer, let
$$
\widetilde\pi_*=(Q_*^{1/2},Q_*^{1/2})_\#\pi_*,
$$
and define
$$
\widetilde\mu=(Q_*^{1/2})_\#\mu,
\qquad
\widetilde\nu=(Q_*^{1/2})_\#\nu.
$$
Then
$$
\int (y-x)^\top Q_*(y-x)\,d\pi_*(x,y)
=
\int \abs{\tilde y-\tilde x}^2\,d\widetilde\pi_*(\tilde x,\tilde y).
$$
Because $\mu$ is absolutely continuous and $Q_*^{1/2}$ is invertible,
$\widetilde\mu$ is also absolutely continuous. Brenier's theorem therefore
yields a convex potential $u$ such that the unique optimal coupling between
$\widetilde\mu$ and $\widetilde\nu$ is induced by the map
$$
\widetilde T(\tilde x)=\nabla u(\tilde x).
$$
Pulling this map back to the original variables gives
$$
T_*(x)=Q_*^{-1/2}\widetilde T(Q_*^{1/2}x)=Q_*^{-1/2}\nabla u(Q_*^{1/2}x).
$$
Since $\pi_*$ is $Q_*$-optimal and the latter coupling is unique, $\pi_*$ must
be induced by $T_*$. The argument applies to every primal optimizer.
\end{proof}

\section{Gaussian Reductions}
\label{app:gaussians}
This appendix gathers the Gaussian reductions used in the static geometry and in the linear-network flow.
Gaussian marginals compress the transport problem to a finite-dimensional optimization over admissible covariance blocks.
\medskip
The next proposition shows that, for Gaussian marginals, the infinite-dimensional transport problem collapses to an optimization over the cross-covariance matrix.
\begin{proposition}[Gaussian reduction]
\label{app-thm:gaussian-full}
Let $\mu=\mathcal N(m_0,\Sigma_0)$ and $\nu=\mathcal N(m_1,\Sigma_1)$. Then
$$
\Wg(\mu,\nu)^2
=
\inf_{K}
\gamma\!\left(
(m_1-m_0)(m_1-m_0)^\top
+
\Sigma_0+\Sigma_1-K-K^\top
\right),
$$
where the infimum runs over matrices $K$ such that
$$
\begin{pmatrix}
\Sigma_0 & K \\
K^\top & \Sigma_1
\end{pmatrix}
\succeq 0.
$$
In particular, for centered Gaussians the covariance cost
$$
\BWg(\Sigma_0,\Sigma_1)^2
\coloneqq 
\inf_{K:\,
\left[\begin{smallmatrix}
\Sigma_0 & K\\
K^\top & \Sigma_1
\end{smallmatrix}\right]\succeq 0}
\gamma(\Sigma_0+\Sigma_1-K-K^\top)
$$
is well defined on the cone of covariance matrices. In the monotone regime of Section~\ref{sec:dynamic}, it is the restriction of the metric $\Wg$ to centered Gaussian laws and therefore defines a metric on covariance matrices.
\end{proposition}

\begin{proof}
Let $(X,Y)$ be any coupling of the two Gaussian marginals and set
$K=\mathbb E[(X-m_0)(Y-m_1)^\top]$. Its block covariance is
$$
\begin{pmatrix}
\Sigma_0 & K\\
K^\top & \Sigma_1
\end{pmatrix}.
$$
It must be positive semidefinite. Conversely, every feasible block matrix is
the covariance of a jointly Gaussian vector with marginals $\mu$ and $\nu$.
For any coupling with cross-covariance $K$, the displacement covariance is
$$
(m_1-m_0)(m_1-m_0)^\top+\Sigma_0+\Sigma_1-K-K^\top.
$$
The cost therefore depends only on $K$, and every feasible $K$ is realized by
a Gaussian coupling. Optimizing over all couplings is thus equivalent to the
stated finite-dimensional problem.
\end{proof}

The next corollary shows that commuting covariances lead to a closed form for the Schatten family, exactly as in the classical Bures case.
\begin{corollary}[Commuting covariances]
\label{app-cor:commuting}
If $\Sigma_0$ and $\Sigma_1$ commute, then for Schatten norms $\gamma_p$,
$1\le p\le\infty$, one has
$$
\mathsf B_{\gamma_p}(\Sigma_0,\Sigma_1)^2
=
\gamma_p\!\left(
(\Sigma_0^{1/2}-\Sigma_1^{1/2})^2
\right).
$$
Equivalently, if $a_i$ and $b_i$ are the diagonal entries of $\Sigma_0$ and
$\Sigma_1$ in a common eigenbasis, then, for $p<\infty$,
$$
\mathsf B_{\gamma_p}(\Sigma_0,\Sigma_1)^2
=
\left(
\sum_{i=1}^d \abs{\sqrt{a_i}-\sqrt{b_i}}^{2p}
\right)^{1/p}.
$$
For $p=\infty$, the right-hand side is
$\max_i|\sqrt{a_i}-\sqrt{b_i}|^2$.
When $p=1$, this is the usual Bures--Wasserstein formula.
\end{corollary}

\begin{proof}
Simultaneously diagonalize $\Sigma_0=\diag(a_i)$ and
$\Sigma_1=\diag(b_i)$. For any feasible $K$, each $2\times2$ principal
minor involving the coordinates $X_i$ and $Y_i$ gives
$|K_{ii}|\le\sqrt{a_ib_i}$. Hence the displacement covariance
$D=\Sigma_0+\Sigma_1-K-K^\top\succeq0$ satisfies
\[
 D_{ii}=a_i+b_i-2K_{ii}
 \ge(\sqrt{a_i}-\sqrt{b_i})^2.
\]
Pinching onto the common eigenbasis is contractive for every unitarily
invariant norm, so
\[
 \|D\|_{S_p}
 \ge\|\diag(D_{ii})\|_{S_p}
 \ge\left\|\diag\bigl((\sqrt{a_i}-\sqrt{b_i})^2\bigr)\right\|_{S_p}.
\]
The diagonal choice $K_0=\diag(\sqrt{a_ib_i})$ is feasible and attains both
inequalities, proving the claim.
\end{proof}

The following remark clarifies the scope of Corollary~\ref{app-cor:commuting}.
\begin{remark}
The commuting formula is stated for Schatten norms because their value
depends only on eigenvalues. A general gauge may retain basis-dependent
information. The restriction $p\ge1$ is also essential: at $p=1/2$, take
$\Sigma_0=\diag(4,1)$ and $\Sigma_1=\diag(1,4)$. In the basis
$(e_1+e_2)/\sqrt2,(e_1-e_2)/\sqrt2$, the map
$Y_1=X_1$, $Y_2=X_2-\frac65X_1$ has the required target covariance and
displacement covariance $\diag(0,18/5)$. Its cost is $18/5$, strictly below
the value $\gamma_{1/2}(I)=4$ predicted by the commuting expression.
Thus the root-norm extension below does not extend this particular formula.
\end{remark}

%%%%%%%%
\section{Dynamic Formulation Details}
\label{app:dynamic-details}
This appendix gathers the convex momentum formulation and the proofs deferred from Section~\ref{sec:dynamic}.

\subsection{Convex momentum reformulation}

Passing to momenta linearizes the constraint. If $\mu=\rho\lambda$ and $m=w\lambda$ for some reference measure $\lambda$, define
$$
\cA(\mu,m)
\coloneqq 
\sup_{Q\in\cK}
\int_{\R^d}\frac{w(x)^\top Q\,w(x)}{\rho(x)}\,d\lambda(x),
$$
with the usual perspective convention on $\{\rho=0\}$.

The following proposition shows that the momentum formulation is convex and exactly matches the velocity action.

\begin{proposition}[Convex momentum action]
\label{prop:momentum}
The action $\cA(\mu,m)$ is intrinsic, convex, and weak-* lower semicontinuous. If $m=\mu v$, then
$$
\cA(\mu,m)
=
\gamma\!\left(
\int v(x)v(x)^\top\,d\mu(x)
\right).
$$
Hence
$$
\Wgbb(\mu_0,\mu_1)^2
=
\inf_{(\mu_t,m_t)}
\int_0^1 \cA(\mu_t,m_t)\,dt
$$
under the linear constraint
$$
\partial_t\mu_t+\divv(m_t)=0.
$$
\end{proposition}

\begin{proof}
Let $\lambda$ be any positive measure dominating both $\mu$ and $m$, and write
$\mu=\rho\lambda$ and $m=w\lambda$, with $\rho\ge0$ and
$w:\R^d\to\R^d$. For $Q\succeq0$, define the perspective quadratic integrand
\[
    \phi_Q(\rho,w)
    \coloneqq
    \begin{cases}
        \dfrac{w^\top Qw}{\rho}, & \rho>0, \\[0.8em]
        0, & \rho=0 \text{ and } Q^{1/2}w=0, \\[0.3em]
        +\infty, & \rho=0 \text{ and } Q^{1/2}w\neq0.
    \end{cases}
\]
Then $\phi_Q$ is convex and lower semicontinuous on
$\R_+\times\R^d$. Indeed, for $\rho>0$ it is the perspective of the convex
quadratic form $w\mapsto w^\top Qw$, and the above convention is its closed
extension at $\rho=0$. Notice that kernel directions have zero recession cost
when $Q$ is singular.

The momentum action may be written as
\[
    \cA(\mu,m)
    =
    \sup_{Q\in\cK}
    \int \phi_Q(\rho(x),w(x))\,d\lambda(x).
\]
This formula is intrinsic, i.e., independent of the dominating measure
$\lambda$. To see this, choose another dominating measure $\lambda'$ and write
$\mu=\rho'\lambda'$ and $m=w'\lambda'$. Passing to a common dominating measure,
for instance $\eta=\lambda+\lambda'$, and using the one-homogeneity of
$\phi_Q$, namely $\phi_Q(a\rho,aw)=a\phi_Q(\rho,w)$ for $a\ge0$, gives the same
value of $\int\phi_Q\,d\lambda$ and $\int\phi_Q\,d\lambda'$. Thus
$\cA(\mu,m)$ depends only on the pair $(\mu,m)$.

We next prove the absolutely continuous representation. If $m=\mu v$, then
one may take $\lambda=\mu$, so that $\rho=1$ and $w=v$. Hence, for every
$Q\in\cK$,
\[
    \int \phi_Q(1,v(x))\,d\mu(x)
    =
    \int v(x)^\top Qv(x)\,d\mu(x)
    =
    \tr\!\left(
        Q\int v(x)v(x)^\top\,d\mu(x)
    \right).
\]
Taking the supremum over $Q\in\cK$ and using the support representation of
$\gamma$ yields
\[
    \cA(\mu,\mu v)
    =
    \sup_{Q\in\cK}
    \tr\!\left(
        Q\int vv^\top\,d\mu
    \right)
    =
    \gamma\!\left(
        \int vv^\top\,d\mu
    \right).
\]
Conversely, strict positivity of $\gamma$ implies that $\cK$ contains a
positive definite matrix. Indeed, for every unit vector $a$ there is a
$Q_a\in\cK$ with $a^\top Q_a a>0$; compactness of the sphere gives finitely
many such matrices whose average is positive definite and belongs to $\cK$ by
convexity. Fix one such matrix $Q_0$. If $\cA(\mu,m)<+\infty$, then
$\int\phi_{Q_0}(\rho,w)\,d\lambda$ is finite. Since $Q_0\succ0$, the singular
case $\rho=0$ and $w\neq0$ cannot occur on a set of positive
$\lambda$-measure. Thus $m$ is absolutely continuous with respect to $\mu$,
and $m=\mu v$ with $v=w/\rho$ on $\{\rho>0\}$.

Convexity follows from the same perspective formulation. For fixed
$Q\succeq0$, the map $(\mu,m)\mapsto\int\phi_Q(\rho,w)\,d\lambda$ is convex,
because $\phi_Q$ is convex and one-homogeneous. Taking the supremum over
$Q\in\cK$ preserves convexity, hence $(\mu,m)\mapsto\cA(\mu,m)$ is convex.

Finally, we prove weak-* lower semicontinuity. For each fixed
$Q\succeq0$, the functional
\[
    (\mu,m)\longmapsto
    \int \phi_Q\!\left(\frac{d\mu}{d\lambda},
                       \frac{dm}{d\lambda}\right)d\lambda
\]
is lower semicontinuous under weak-* convergence of measures, by the standard
lower semicontinuity theorem for convex, lower semicontinuous,
one-homogeneous integral functionals. Since $\cA$ is the supremum over
$Q\in\cK$ of these lower semicontinuous functionals, $\cA$ is also weak-*
lower semicontinuous.

The dynamic formula now follows by rewriting the velocity formulation in
momentum variables. If $(\mu_t,v_t)$ satisfies
$\partial_t\mu_t+\divv(\mu_t v_t)=0$ and $m_t=\mu_t v_t$, then the constraint
becomes the linear continuity equation
$\partial_t\mu_t+\divv(m_t)=0$, and the previous identity gives
\[
    \MN_{\mu_t}(v_t)^2
    =
    \gamma\!\left(\int v_t v_t^\top\,d\mu_t\right)
    =
    \cA(\mu_t,m_t).
\]
Therefore the dynamic action
$\int_0^1\MN_{\mu_t}(v_t)^2\,dt$ is exactly
$\int_0^1\cA(\mu_t,m_t)\,dt$ in momentum variables. Taking the infimum over
all admissible curves gives
\[
    \Wgbb(\mu_0,\mu_1)^2
    =
    \inf_{(\mu_t,m_t)}
    \int_0^1 \cA(\mu_t,m_t)\,dt,
\]
under the linear constraint
$\partial_t\mu_t+\divv(m_t)=0$.
\end{proof}

\subsection{Proof of Theorem~\ref{thm:staticdynamic} (static/dynamic equivalence)}

\begin{proof}
Let $\pi\in\Pi(\mu_0,\mu_1)$ and consider the displacement interpolation
$$
z_t(x,y)=(1-t)x+ty,
\qquad
\mu_t=(z_t)_\#\pi.
$$
Its Eulerian velocity is the conditional mean
$v_t(z)=\mathbb E_\pi[y-x\mid z_t=z]$. The pair $(\mu_t,v_t)$ solves the
continuity equation, and conditional Jensen's inequality gives
$$
\int v_t v_t^\top\,d\mu_t
\preceq
\int (y-x)(y-x)^\top\,d\pi
$$
for a.e. $t$. By monotonicity of $\gamma$, its dynamic action is no larger
than the static cost of $\pi$. Taking the infimum over $\pi$ gives
$$
\Wgbb(\mu_0,\mu_1)^2\le \Wg(\mu_0,\mu_1)^2.
$$

Conversely, take any admissible dynamic plan $(\mu_t,v_t)$. By the
superposition principle, there exists a probability measure $\eta$ on
absolutely continuous paths $\omega$ such that
$$
\mu_t=(e_t)_\#\eta,
\qquad
\dot\omega_t=v_t(\omega_t)
$$
for $\eta$-a.e.\ path and a.e.\ $t$. Let $\pi=(e_0,e_1)_\#\eta$. For every $Q\in\cK$,
$$
\int_0^1 \int v_t^\top Q v_t\,d\mu_t\,dt
=
\int \int_0^1 \dot\omega_t^\top Q \dot\omega_t\,dt\,d\eta(\omega).
$$
This is the key point where monotonicity is used: by Proposition~\ref{prop:psdpolar} we have chosen $\cK\subset \mathbb S_+^d$, so every test matrix satisfies $Q\succeq 0$. Without this reduction one would have to work with possibly indefinite matrices, and the quadratic Jensen inequality below would no longer apply.
Applying Jensen to $t\mapsto Q^{1/2}\dot\omega_t$ gives
$$
\int_0^1 \dot\omega_t^\top Q \dot\omega_t\,dt
\ge
(\omega_1-\omega_0)^\top Q(\omega_1-\omega_0).
$$
Hence
$$
\int_0^1 \gamma\!\left(\int v_t v_t^\top\,d\mu_t\right)\,dt
\ge
\int (y-x)^\top Q(y-x)\,d\pi(x,y)
$$
for every $Q\in\cK$. Taking the supremum over $Q$ yields
$$
\int_0^1 \gamma\!\left(\int v_t v_t^\top\,d\mu_t\right)\,dt
\ge
\gamma\!\left(\int (y-x)(y-x)^\top\,d\pi\right).
$$
Finally, take the infimum over admissible dynamic plans.
\end{proof}

\subsection{Proof of Corollary~\ref{cor:geodesic} (Spectral Geodesics)}

\begin{proof}
The upper bound follows by restricting the same displacement interpolation to
the time interval $[s,t]$. The coupling
$((1-s)x+sy,(1-t)x+ty)_\#\pi_*$ has displacement $(t-s)(y-x)$, so the static
definition gives
$$
\Wg(\mu_s,\mu_t)^2\le (t-s)^2 \gamma\!\left(\int (y-x)(y-x)^\top\,d\pi_*(x,y)\right)
=(t-s)^2\Wg(\mu_0,\mu_1)^2.
$$
Applying the same argument to the segments $[0,s]$ and $[t,1]$ yields
$$
\Wg(\mu_0,\mu_s)\le s\Wg(\mu_0,\mu_1),
\qquad
\Wg(\mu_t,\mu_1)\le (1-t)\Wg(\mu_0,\mu_1).
$$
By the triangle inequality from Corollary~\ref{cor:metric},
$$
\Wg(\mu_0,\mu_1)
\le
\Wg(\mu_0,\mu_s)+\Wg(\mu_s,\mu_t)+\Wg(\mu_t,\mu_1)
\le
\Wg(\mu_s,\mu_t)+(1-(t-s))\Wg(\mu_0,\mu_1).
$$
Rearranging gives
$$
\Wg(\mu_s,\mu_t)\ge (t-s)\Wg(\mu_0,\mu_1),
$$
hence equality holds.
\end{proof}

\subsection{Root-norm gauges beyond convexity}
\label{app:root-norm-gauges}

The main text assumes that $\gamma$ is convex, monotone, and positively
one-homogeneous.  Convexity is essential for the support-function
representation $\gamma(S)=\max_Q\tr(QS)$ and hence for the cost-robust
formulation.  Some parts of the theory, however, only use that the square-root
action is a norm.

\begin{definition}[Root-norm gauge]
\label{def:root-norm-gauge}
A positive one-homogeneous function
$\gamma:\mathbb S_+^d\to\R_+$ is a root-norm gauge if, for every Hilbert space
$\mathcal H$, the map
$$
    T:\R^d\to\mathcal H
    \quad\longmapsto\quad
    \mathfrak N_\gamma(T)\coloneqq \gamma(T^*T)^{1/2}
$$
is a norm on finite-rank operators.
\end{definition}

For a vector field $v\in L^2(\mu;\R^d)$, this gives exactly the tangent norm
used in the paper:
$$
    \MN_\mu(v)=
    \mathfrak N_\gamma(T_v),
    \qquad
    (T_va)(x)=a\cdot v(x),
$$
since $T_v^*T_v=\int vv^\top\,d\mu$.

\begin{proposition}[Characterization]
\label{prop:root-norm-characterization}
A positive one-homogeneous $\gamma$ is a root-norm gauge if and only if
$V\mapsto\gamma(V^\top V)^{1/2}$ is a norm on $\R^{d\times d}$.
Equivalently, $\gamma(S)=\mathfrak N(S^{1/2})^2$ for a left-orthogonally
invariant norm $\mathfrak N$ on square matrices. In the spectral case,
this is equivalent to
$$
    \gamma(S)=\varphi\!\left(\lambda(S)^{1/2}\right)^2
$$
for some symmetric gauge norm $\varphi$ on $\R^d$.  Every root-norm gauge is
monotone on $\mathbb S_+^d$, but it need not be convex.
\end{proposition}

\begin{proof}
Necessity follows by choosing the target $\R^d$. Conversely, suppose the
square-matrix map $\mathfrak N$ is a norm. It is left-orthogonally invariant.
Every square contraction is a convex combination of orthogonal matrices,
as follows by writing its singular values as convex combinations of $\pm1$.
Hence $\mathfrak N(CV)\le\mathfrak N(V)$ for $\|C\|_{\rm op}\le1$.
If $0\preceq S\preceq T$, then
$C=S^{1/2}T^{\dagger/2}$ is a contraction and $S^{1/2}=CT^{1/2}$, proving
$\gamma(S)\le\gamma(T)$.

For operators $T_1,T_2:\R^d\to\mathcal H$, choose a contraction
$R:\mathcal H\to\R^d$ that is isometric on the range of $T_1+T_2$, whose
dimension is at most $d$. Then
\[
 \mathfrak N_\gamma(T_1+T_2)
 =\mathfrak N(RT_1+RT_2)
 \le\mathfrak N(RT_1)+\mathfrak N(RT_2)
 \le\mathfrak N_\gamma(T_1)+\mathfrak N_\gamma(T_2).
\]
The last inequality uses $(RT_i)^*RT_i\preceq T_i^*T_i$ and the monotonicity
just proved. Positivity and homogeneity follow from the square-matrix norm.
This proves the norm property for every target. Polar decomposition gives
the equivalent description by left-invariant norms; the spectral case is
the usual symmetric-gauge description of unitarily invariant norms.
Finally, for any target contraction $B$, the inequality
$T^*B^*BT\preceq T^*T$ proves
$\mathfrak N_\gamma(BT)\le\mathfrak N_\gamma(T)$.
\end{proof}

\begin{corollary}[Schatten threshold]
\label{cor:root-norm-schatten}
For $d\ge2$ and $p\in(0,\infty]$, the function
$\gamma_p(S)=\norm{S}_{S_p}$ is a root-norm gauge if and only if $p\ge1/2$.
It is convex if and only if $p\ge1$. In dimension one, all these functions
equal $\gamma_p(s)=s$, so both properties hold for every $p>0$.
\end{corollary}

\begin{proof}
One has
$$
    \gamma_p(V^\top V)^{1/2}=\norm{V}_{S_{2p}}.
$$
When $d\ge2$, this is a norm exactly when $2p\ge1$. Convexity of the Schatten
$p$-functional on the positive cone holds exactly for $p\ge1$. When $d=1$,
there is only one singular value, so neither threshold applies.
\end{proof}

For nonconvex root-norm gauges, the static Kantorovich formula still makes
sense,
$$
    \Wg(\mu,\nu)^2
    =
    \inf_{\pi\in\Pi(\mu,\nu)}
    \gamma\!\left(\int (y-x)(y-x)^\top\,d\pi(x,y)\right),
$$
but it is no longer a convex optimization problem in the coupling.  Moreover,
there is no exact cost-robust formula of the form
$\gamma(S)=\max_{Q\in\cK}\tr(QS)$, since a maximum of linear forms is convex.
Thus the direct robust-cost proof of the triangle inequality from
Appendix~\ref{subsec:metric-static} is no longer available.

\begin{theorem}[Benamou--Brenier formula for root-norm gauges]
\label{thm:root-norm-bb}
Let $\gamma$ be a root-norm gauge.  Then, for every
$\mu_0,\mu_1\in\cP_2(\R^d)$,
$$
    \Wg(\mu_0,\mu_1)^2
    =
    \inf_{\partial_t\mu_t+\divv(\mu_tv_t)=0}
    \int_0^1
    \gamma\!\left(\int v_t v_t^\top\,d\mu_t\right)\,dt,
$$
where the infimum is over narrowly continuous curves with endpoints
$\mu_0,\mu_1$ and square-integrable velocity fields.
\end{theorem}

\begin{proof}
Let $(\mu_t,v_t)$ be admissible. By the superposition principle, there exists
a probability measure $\eta$ on absolutely continuous paths $\omega$ such
that $\mu_t=(e_t)_\#\eta$ and
$\dot\omega_t=v_t(\omega_t)$ for a.e. $t$. Set
$\pi=(e_0,e_1)_\#\eta$ and $\Delta(\omega)=\omega_1-\omega_0$. Then
$$
    \Delta=\int_0^1 v_t(\omega_t)\,dt.
$$
At the operator level,
$T_\Delta=\int_0^1 T_{v_t\circ e_t}\,dt$ as maps
$\R^d\to L^2(\eta)$.  Hence, by the triangle inequality for
$\mathfrak N_\gamma$ and Cauchy--Schwarz,
$$
\begin{aligned}
    \gamma\!\left(\int \Delta\Delta^\top\,d\eta\right)^{1/2}
    &\le
    \int_0^1
    \gamma\!\left(\int v_t v_t^\top\,d\mu_t\right)^{1/2}dt  \\
    &\le
    \left(
    \int_0^1
    \gamma\!\left(\int v_t v_t^\top\,d\mu_t\right)dt
    \right)^{1/2}.
\end{aligned}
$$
Since $\Delta$ depends only on the endpoints, this gives
$\Wg(\mu_0,\mu_1)^2$ bounded above by the dynamic action.

Conversely, fix $\pi\in\Pi(\mu_0,\mu_1)$ and set
$z_t=(1-t)x+ty$, $\mu_t=(z_t)_\#\pi$.  Let
$v_t(z)=\mathbb E_\pi[y-x\mid z_t=z]$.  Then
$(\mu_t,v_t)$ solves the continuity equation. The conditional-expectation
operator $C_t:L^2(\pi)\to L^2(\mu_t)$ is a contraction and satisfies
$T_{v_t}=C_tT_{y-x}$. Since $\mathfrak N_\gamma$ is contractive under target
contractions,
$$
    \MN_{\mu_t}(v_t)
    \le
    \gamma\!\left(\int (y-x)(y-x)^\top\,d\pi\right)^{1/2}.
$$
Integrating over $t$ gives a dynamic action no larger than the static cost of
$\pi$.  Taking the infimum over $\pi$ proves the reverse inequality.
\end{proof}

\begin{corollary}[Metric property for root-norm gauges]
\label{cor:root-norm-metric}
If $\gamma$ is a root-norm gauge, then $\Wg$ is a distance on
$\cP_2(\R^d)$ and induces the same topology as $\Wtwo$.
\end{corollary}

\begin{proof}
Nonnegativity and symmetry are immediate. Apply finite-dimensional norm
equivalence on $\R^{d\times d}$ to the square root $S^{1/2}$. This gives
constants $0<c\le C<\infty$, independent of the Hilbert target, such that
$$
    c\,\tr(S)\le \gamma(S)\le C\,\tr(S),
    \qquad S\succeq0.
$$
Thus $\sqrt c\,\Wtwo\le\Wg\le\sqrt C\,\Wtwo$, which gives separation and
topological equivalence.

For the triangle inequality, one may either use
Theorem~\ref{thm:root-norm-bb} and concatenate dynamic curves, or argue
directly.  Given couplings $\pi_{01}$ and $\pi_{12}$, glue them into a measure
on triples $(x_0,x_1,x_2)$.  With $\Delta_{ij}=x_j-x_i$, one has
$\Delta_{02}=\Delta_{01}+\Delta_{12}$, and the norm property gives
$$
    \mathfrak N_\gamma(T_{\Delta_{02}})
    \le
    \mathfrak N_\gamma(T_{\Delta_{01}})
    +
    \mathfrak N_\gamma(T_{\Delta_{12}}).
$$
Taking infima over the two input couplings proves
$\Wg(\mu_0,\mu_2)\le \Wg(\mu_0,\mu_1)+\Wg(\mu_1,\mu_2)$.
\end{proof}

\subsection{Proof of Theorem~\ref{prop:formalGF} (Spectral Gradient Flow)}
\label{sec:proof-formal-gamma-flow}

\begin{proof}
Write $N=\MN_\mu$ and $N^*=\MN_\mu^*$ at a fixed state on the curve.
First, norm duality gives, for every $v\in\MJ_\gamma^\mu(g)$,
\begin{equation}\label{eq:lmo-dissipation-identities}
 N(v)=N^*(g),\qquad \int g\cdot v\,d\mu=-N^*(g)^2.
\end{equation}
Indeed, optimizing the radius along a unit-ball minimizing direction reduces
the objective to $-rN^*(g)+r^2/2$.

We next prove the slope identity without assuming that $\mu$ has a density.
Let $\pi$ be an optimal spectral coupling from $\mu$ to $\nu$, and let
$w(x)=\int(y-x)\,d\pi(y\mid x)$ be its conditional mean displacement.
Conditional Jensen and monotonicity give
$N(w)\le\Wg(\mu,\nu)$. Also,
$\int|y-x|^2d\pi\le c_\gamma^{-1}\Wg(\mu,\nu)^2$ by
Proposition~\ref{prop:bounds}. The transport expansion therefore implies
\[
 \MF(\mu)-\MF(\nu)
 \le N^*(g)\Wg(\mu,\nu)+C c_\gamma^{-1}\Wg(\mu,\nu)^2,
\]
and hence $|\partial\MF|_\gamma(\mu)\le N^*(g)$.
Conversely, for $w\in L^2(\mu;\R^d)$ and $h>0$, use the deterministic
coupling to $\nu_h=(\Id+hw)_\sharp\mu$:
\[
 \Wg(\mu,\nu_h)\le hN(w),\qquad
 \MF(\nu_h)-\MF(\mu)=h\int g\cdot w\,d\mu+O(h^2\|w\|_{L^2(\mu)}^2).
\]
For $\int g\cdot w\,d\mu<0$, division and passage to the limit give
$|\partial\MF|_\gamma(\mu)\ge-\int g\cdot w\,d\mu/N(w)$.
Taking the supremum proves $|\partial\MF|_\gamma(\mu)=N^*(g)$; the case
$g=0$ follows from the upper bound.

The assumed spatial Lipschitz and growth bounds provide a characteristic
flow, a uniform second-moment bound, and
\[
 \frac{X_{t+h}-X_t}{h}\longrightarrow v_t(X_t)
 \quad\hbox{in }L^2(\mu_0;\R^d).
\]
Apply \eqref{eq:transport-taylor} to the coupling
$(X_t,X_{t+h})_\sharp\mu_0$. Its squared displacement is $O(h^2)$, so
\eqref{eq:lmo-dissipation-identities} yields
\[
 \frac{d}{dt}\MF(\mu_t)=\int g_t\cdot v_t\,d\mu_t
 =-\MN_{\mu_t}^*(g_t)^2.
\]
The energy is locally Lipschitz: $N^*(g_t)=N(v_t)$ is uniformly bounded by
the growth and moment bounds, as is $\|g_t\|_{L^2(\mu_t)}$ by norm equivalence.

The BB formula, rescaled from $[0,1]$ to $[s,t]$, gives
$\Wg(\mu_s,\mu_t)^2\le(t-s)\int_s^t N_{\mu_r}(v_r)^2\,dr$.
Thus the curve is absolutely continuous and its metric derivative exists
almost everywhere, with $|\mu'|_\gamma(t)\le N_{\mu_t}(v_t)$.
At such a time, the definition of descending slope and the energy derivative
give
$N_{\mu_t}^*(g_t)^2\le
|\partial\MF|_\gamma(\mu_t)|\mu'|_\gamma(t)$.
Together with the slope identity this proves equality of speed, slope, and
dual norm whenever $g_t\ne0$. If $g_t=0$, the speed bound already gives zero.
The stated energy dissipation identity follows.
\end{proof}

\paragraph{A concrete smooth class.}
Let $\varphi_k\in C^2(\R^d)$ have bounded Hessian, and let
$\MF(\mu)=\Phi(m(\mu))$, where
$m(\mu)=(\int\varphi_k\,d\mu)_{k=1}^m$ and $\Phi$ is $C^2$ near the
moment vectors of the curve. The forces are
$g_\mu=\sum_k\partial_k\Phi(m(\mu))\nabla\varphi_k$.
On a uniformly second-moment-bounded curve,
\[
 |m(\nu)-m(\mu)|\le C_1
 \left[\left(\int|y-x|^2d\pi\right)^{1/2}
             +\int|y-x|^2d\pi\right].
\]
Taylor expansion of each feature and then of $\Phi$ proves
\eqref{eq:transport-taylor}, uniformly for small displacements.
Potential energies are included by taking a single feature.
Smooth symmetric interaction energies
$\frac12\iint K(x,y)\,d\mu(x)d\mu(y)$ satisfy the same estimate when
$K$ has bounded Hessian, by applying its Taylor formula to $\pi\otimes\pi$.
These examples concern the differential of the energy; they do not by
themselves establish regularity of the selected spectral velocity.
At empirical measures, their transport expansion gives the chain rule
of Proposition~\ref{cor:particle}; their particle lifts are $C^1$ by
ordinary finite-dimensional differentiation.

\section{Gaussian-Preserving Flows and Linear Networks}
\label{app:gaussian-linear}
This section isolates a tractable regime in which the spectral flow reduces from an infinite-dimensional PDE to finite-dimensional covariance dynamics. The general Gaussian-preservation principle comes first, and the rest of the section specializes it to centered Gaussian training of linear two-layer networks.

\subsection{Gaussian-Preserving Gradient Flows}
\label{app-sec:gaussflow}
This subsection isolates a simple class of spectral gradient flows that can still be analyzed rather explicitly. In general, understanding the long-time behavior of the flow is hard, but when Gaussian measures are preserved the infinite-dimensional evolution reduces to a finite-dimensional ODE on means and covariances.

The next corollary explains why Gaussian preservation transfers from the
classical $\Wtwo$ flow to the spectral flow. The finite-dimensional
Measure-LMO reduction is important here because it also covers singular
endpoints such as $p=\infty$.

\begin{corollary}[From $\Wtwo$ Gaussian preservation to spectral Gaussian preservation]
\label{app-cor:w2-to-spectral-paper}
Fix a class of Gaussian states in the domain of $\MF$ and assume that,
at each such state, the classical Wasserstein gradient
$$
g_\mu(x)=\nabla_x\frac{\delta \MF}{\delta\mu}(x)
$$
is affine in $x$. Then the Measure-LMO admits an affine selection at each
state in this class, and the spectral flow using this selection,
$$
\partial_t\mu_t+\divv\!\bigl(\mu_t \MJ_\gamma^{\mu_t}(g_{\mu_t})\bigr)=0
$$
admits a closed Gaussian ansatz: every solution of the induced mean and
covariance ODE that stays in this class and has locally integrable affine
coefficients defines a Gaussian solution of the PDE. This does not assert
uniqueness among all PDE solutions.
\end{corollary}

\begin{proof}
Fix a state $\mu$ in this class and write
$$
g_\mu(x)=b_\mu+B_\mu x
$$
for the affine $\Wtwo$ gradient provided by the assumption. Set
$S_g=\mathcal S_\mu(g_\mu)$. By
Proposition~\ref{prop:measure-lmo-reduction}, there is a matrix $A_g$ such that
$$
\MJ_\gamma^\mu(g_\mu)(x)
=A_g^\top S_g^{\dagger/2}(b_\mu+B_\mu x).
$$
This is again affine in $x$ (and is zero when $S_g=0$). Therefore, the
spectral continuity equation is of the form
$$
v_t(x)=a_t+A_t x
$$
with coefficients depending on the current Gaussian state. A continuity equation with affine drift preserves the Gaussian class, and the associated mean $m_t$ and covariance $\Sigma_t$ solve the closed ODE system
$$
\dot m_t=a_t+A_t m_t,
\qquad
\dot \Sigma_t=A_t\Sigma_t+\Sigma_t A_t^\top.
$$
Conversely, a solution of these moment equations with locally integrable
$a_t,A_t$ is transported by their affine characteristic flow, so its Gaussian
law solves the PDE. Existence and uniqueness of this state-dependent moment
system, particularly at rank changes, are separate questions.
\end{proof}

Two important examples fit into this mechanism. First, relative entropy
toward a nondegenerate Gaussian target has an affine Wasserstein force at
every nondegenerate Gaussian state; singular Gaussian states have infinite
relative entropy and are excluded. Second, if $\MF$ depends smoothly on
the mean and covariance, its first variation is quadratic and its
Wasserstein force is affine. This includes the linear two-layer model with
quadratic risk below, whose loss is a quadratic function of its uncentered
cross-moment.

\subsection{Centered Gaussian Linear Two-Layer Model}
This subsection develops a concrete Gaussian-preserving example where the
predictor is a linear map and the spectral flow closes on covariance
matrices. Write $m$ for the input and output dimension; the particle
dimension in the general theory is then $d=2m$.
Consider particles
$$
x=(u,v)\in\R^m\times\R^m,
$$
and the linear two-layer feature map
$$
\phi(z,x)=u(v^\top z),
\qquad z\in\R^m.
$$
Here $u$ is the outer weight and $v$ is the inner weight. For a probability
measure $\mu$ on $(u,v)$, the predictor is
$$
H_\mu(z)=\int u(v^\top z)\,d\mu(u,v)=M_\mu z,
\qquad
M_\mu\coloneqq \int uv^\top\,d\mu(u,v).
$$
With a linear teacher $y^\star(z)=A_\star z$ and data law $\rho$, the population least-squares loss is
$$
\MF(\mu)=\frac12\int \abs{H_\mu(z)-A_\star z}^2\,d\rho(z).
$$
Writing
$$
C_\rho\coloneqq \int zz^\top\,d\rho(z),
\qquad
B_\rho\coloneqq \int y^\star(z)z^\top\,d\rho(z)=A_\star C_\rho,
$$
one obtains
$$
\MF(\mu)=\frac12\tr(M_\mu C_\rho M_\mu^\top)-\tr(B_\rho M_\mu^\top)+\mathrm{const}.
$$
For centered Gaussian states
$$
\mu=\mathcal N(0,\Sigma),
\qquad
\Sigma=
\begin{pmatrix}
\Sigma_U & \Sigma_{UV}\\
\Sigma_{VU} & \Sigma_V
\end{pmatrix},
$$
the predictor is simply $M_\mu=\Sigma_{UV}$. Thus the cross-covariance block
carries the learned linear map, while the diagonal covariance blocks influence
the normalized metric direction.

The next proposition gives the closed covariance equation for Schatten spectral flows.
\begin{proposition}[Covariance ODE for centered Gaussian states]
\label{app-prop:gaussian-linear-covariance}
Let $p\in[1,\infty]$, let $\gamma=\gamma_p$ be the Schatten-$p$ norm, and let
$$
q=\frac{2p}{2p-1}
$$
be the dual exponent of $2p$, with $q=1$ when $p=\infty$.
On any interval where a centered Gaussian solution
$\mu_t=\mathcal N(0,\Sigma_t)$ uses the canonical Schatten selector and
$\Sigma_t$ is positive definite, one has
$$
\dot \Sigma_t=A_t\Sigma_t+\Sigma_tA_t^\top,
\qquad
A_t=B_t\Sigma_t^{-1/2},
$$
where
$$
B_t=-\norm{K_t}_{S_q}^{2-q}U_t\diag(\sigma_i(t)^{q-1})W_t^\top,
\qquad
K_t=L_t\Sigma_t^{1/2},
$$
$K_t=U_t\diag(\sigma_i(t))W_t^\top$ is a compact singular value decomposition
over its positive singular values, and
$$
L_t=
\begin{pmatrix}
0 & E_t\\
E_t^\top & 0
\end{pmatrix},
\qquad
E_t=\Sigma_{UV,t}C_\rho-B_\rho.
$$
Equivalently,
$$
\dot \Sigma_t=B_t\Sigma_t^{1/2}+\Sigma_t^{1/2}B_t^\top.
$$
When $K_t=0$, set $B_t=0$. This inverse-covariance formula is only asserted
up to the first loss of positive definiteness.
\end{proposition}

\begin{proof}
For centered Gaussian $\mu=\mathcal N(0,\Sigma)$, the loss depends on $\mu$
only through $M_\mu=\Sigma_{UV}$, so differentiating with respect to this
block gives $E=\Sigma_{UV}C_\rho-B_\rho$. Therefore
$$
g_\mu(u,v)=\nabla_x\frac{\delta\MF}{\delta\mu}(u,v)=
\begin{pmatrix}
Ev\\ E^\top u
\end{pmatrix}
=Lx.
$$
The Wasserstein gradient is affine, hence the Gaussian class is preserved by
Corollary~\ref{app-cor:w2-to-spectral-paper}. For an affine velocity
$v(x)=Ax$, the selected direction solves
$$
\min_A\left\{\tr(L\Sigma A^\top)+\frac12\norm{A\Sigma^{1/2}}_{S_{2p}}^2\right\}.
$$
Setting $B=A\Sigma^{1/2}$ and $K=L\Sigma^{1/2}$ reduces this to the finite-dimensional Schatten duality-map problem
$$
\min_B\left\{\tr(KB^\top)+\frac12\norm{B}_{S_{2p}}^2\right\},
$$
whose minimizer is the displayed formula for $B_t$. Finally, if $X_t$ follows $\dot X_t=A_tX_t$ and $X_t\sim\mathcal N(0,\Sigma_t)$, then $\dot\Sigma_t=A_t\Sigma_t+\Sigma_tA_t^\top$.
\end{proof}

\subsection{Modal Reduction Under Simultaneous Diagonalization}
This subsection shows that, under a compatible data basis and balanced initialization, the covariance ODE decomposes into interpretable feature modes.
Assume that after orthogonal changes of coordinates one has
$$
C_\rho=\Id,
\qquad
B_\rho=\diag(\beta_1,\dots,\beta_m),
$$
and use a balanced diagonal initialization
$$
\Sigma_{U,0}=\diag(\alpha_i),
\qquad
\Sigma_{V,0}=\diag(\alpha_i),
\qquad
\Sigma_{UV,0}=0,
\qquad
\alpha_i>0.
$$
Then the covariance remains diagonal mode by mode. We write
$$
r_i(t)=(\Sigma_{UV,t})_{ii},
\qquad
s_i(t)=(\Sigma_{U,t})_{ii}=(\Sigma_{V,t})_{ii},
$$
so that $r_i$ is the predictor coefficient and $s_i$ measures the corresponding layer magnitude.

\begin{proposition}[Modal ODE]
\label{app-prop:modal-ode-paper}
\label{prop:modal-main}
Let $p\in[1,\infty]$, let $q=(2p)'$ (in particular $q=1$ for $p=\infty$), and define
$\psi_q(a)=a|a|^{q-2}$ for $a\ne0$ and $\psi_q(0)=0$. Under the assumptions
above, set
$$
a_i=s_i+r_i,
\qquad
b_i=s_i-r_i,
\qquad
\mathcal N_t^q=
\sum_{j=1}^m \abs{r_j-\beta_j}^q\left(a_j^{q/2}+b_j^{q/2}\right).
$$
Then each mode solves
$$
\dot r_i
=
-\mathcal N_t^{2-q}\,\psi_q(r_i-\beta_i)
\left(a_i^{q/2}+b_i^{q/2}\right),
$$
$$
\dot s_i
=
-\mathcal N_t^{2-q}\,\psi_q(r_i-\beta_i)
\left(a_i^{q/2}-b_i^{q/2}\right).
$$
Equivalently,
$$
\dot a_i=-2\mathcal N_t^{2-q}\,\psi_q(r_i-\beta_i)a_i^{q/2},
$$
$$
\dot b_i=\phantom{-}2\mathcal N_t^{2-q}\,\psi_q(r_i-\beta_i)b_i^{q/2}.
$$
If $\mathcal N_t=0$, all right-hand sides are understood to be zero.
\end{proposition}

\begin{proof}
When $\Sigma_{UV}$ is diagonal, the error
$E_t=\Sigma_{UV,t}-B_\rho$ is diagonal, hence the affine gradient matrix
$$
L_t=
\begin{pmatrix}
0 & E_t\\ E_t & 0
\end{pmatrix}
$$
preserves each coordinate plane. The balanced diagonal covariance structure is therefore invariant. In the $i$-th plane, the covariance block is
$$
\Sigma_t^{(i)}=\begin{pmatrix}s_i&r_i\\ r_i&s_i\end{pmatrix},
$$
with eigenvalues $a_i=s_i+r_i$ and $b_i=s_i-r_i$, while the corresponding block of $L_t$ has eigenvalues $\pm(r_i-\beta_i)$. The singular values of $K_t=L_t\Sigma_t^{1/2}$ are therefore $\abs{r_i-\beta_i}a_i^{1/2}$ and $\abs{r_i-\beta_i}b_i^{1/2}$, with no additional factor of two. Applying the Schatten LMO of Proposition~\ref{app-prop:gaussian-linear-covariance} blockwise and collecting the common normalization gives the stated equations.
\end{proof}

\subsection{Rank-One and Rank-Two Targets}
This subsection records the two simplest modal systems, which are useful for interpreting how the spectral normalization changes learning profiles.
If only one target coefficient is nonzero,
$$
B_\rho=\diag(\beta,0,\dots,0),
\qquad \beta\neq 0,
$$
then all inactive modes remain frozen and the dynamics reduces to the single
pair $(r,s)$ governed by Proposition~\ref{app-prop:modal-ode-paper}. In the
positive-definite interior, the variables $a=s+r$ and $b=s-r$ are positive
and encode the two layer variances in the active symmetric and antisymmetric
directions. This one-mode system is the basic phase portrait behind the
Gaussian experiments: the predictor coefficient $r_t$ moves toward $\beta$,
while the terminal value of $s_t$ records the implicit covariance bias
selected by $p$.
The admissible covariance cone is
$$
\{(r,s)\in\R^2:\ s\ge|r|\},
$$
with positive-definite interior $s>|r|$, equivalently $a>0$ and $b>0$. In
Figure~\ref{fig:gaussian-leaves}(a), we display interior trajectories from two
initial conditions. The next proposition gives the interior conserved leaves
and explains why they need not describe an entire trajectory.

\begin{proposition}[Interior invariants and finite-time rank loss]
\label{app-prop:rank-one-invariants}
For each mode, on every interval where $a_i,b_i>0$, the trace geometry
$p=1$ preserves $a_i b_i=s_i^2-r_i^2$. For $p>1$, set
$\eta=1-q/2$; then $a_i^\eta+b_i^\eta$ is constant. In particular,
$p=\infty$ preserves $\sqrt{s_i+r_i}+\sqrt{s_i-r_i}$.
From the balanced initialization $a_i=b_i=\alpha_i$, an interior trajectory
with $p>1$ satisfies
$$
 |r_i|<R_{p,i}\coloneqq \alpha_i 2^{p/(p-1)},
 \qquad R_{\infty,i}=2\alpha_i.
$$
If $|\beta_i|>R_{p,i}$, any global nonnegative modal solution must lose
rank in this mode in finite time.
\end{proposition}

\begin{proof}
Write $\Lambda_i=\mathcal N_t^{2-q}\psi_q(r_i-\beta_i)$.
The equations $\dot a_i=-2\Lambda_i a_i^{q/2}$ and
$\dot b_i=2\Lambda_i b_i^{q/2}$ give the invariants by differentiation.
For $p>1$, the balanced invariant is $a_i^\eta+b_i^\eta=2\alpha_i^\eta$,
so $a_i,b_i<(2\alpha_i^\eta)^{1/\eta}$ in the interior, giving the
displayed bound on $r_i=(a_i-b_i)/2$.
The sign of $\dot r_i$ points toward $\beta_i$. If
$|\beta_i|>R_{p,i}$, the residual stays bounded away from zero until rank
loss, and the growing eigenvalue is at least $\alpha_i$.
Consequently $|\Lambda_i|\ge c>0$. The shrinking eigenvalue $z$ then
satisfies $\dot z\le-2c z^{q/2}$. Since $q/2<1$, it reaches zero in
finite time. This argument allows other modes to have already lost rank.
\end{proof}

After rank loss the preceding invariant need no longer be conserved. To specify
the subsequent dynamics, we use the modal equations with every extinct
eigenvalue kept at zero. This is a choice of continuation, not a claim of
uniqueness for the full measure-valued PDE.

\begin{proposition}[Selected continuation on the covariance boundary]
\label{app-prop:gaussian-boundary}
Let $(a_i,b_i)$ be a locally absolutely continuous, nonnegative solution
of the modal equations, keeping zero eigenvalues at zero. Then its centered
Gaussian law solves the continuity equation with the canonical Schatten
selector, including after rank loss, and is a curve of maximal slope for
the linear-network loss. For almost every time, its energy satisfies
$$
 \frac{d}{dt}\MF(\mu_t)=-\mathcal N_t^2,
 \qquad \MF(\mu_t)=\frac12\sum_i(r_i-\beta_i)^2.
$$
This statement is conditional on the specified modal solution; it does not
assert well-posedness for arbitrary initial measures.
\end{proposition}

\begin{proof}
Use the fixed orthonormal symmetric/antisymmetric coordinates in each
mode and realize the law by independent coordinates
$\sqrt{a_i(t)}\xi_i^+$ and $\sqrt{b_i(t)}\xi_i^-$, with standard normal
$\xi_i^\pm$. Before extinction,
$$
 \frac{d}{dt}\sqrt{a_i}=-\Lambda_i a_i^{(q-1)/2},
 \qquad
 \frac{d}{dt}\sqrt{b_i}=\Lambda_i b_i^{(q-1)/2}.
$$
These derivatives are locally bounded, since $q\ge1$ and the modal
solution is locally bounded; after extinction they are zero. Thus this
realization is locally absolutely continuous in $L^2$ and its derivative
is precisely the canonical Schatten velocity on the support. Its action
satisfies $\MN_{\mu_t}(v_t)^2=\mathcal N_t^2$ by
\eqref{eq:lmo-dissipation-identities}. This proves the continuity equation
with locally finite action, even if its affine coefficient is not integrable
through extinction; no inverse of $\Sigma_t$ is needed.
Differentiating the loss gives $\MF'=-\mathcal N_t^2$ by the modal ODE.
The transport--Taylor estimate of
Theorem~\ref{prop:formalGF} holds for this quadratic
moment functional, including at singular Gaussian measures: for each coupling
$\pi\in\Pi(\mu_t,\nu)$ with small squared displacement, expanding the
moments gives a remainder $O(\int|y-x|^2\,d\pi)$ at a fixed state.
This is the displacement cost of the chosen coupling, not necessarily the
optimal $\Wtwo(\mu_t,\nu)^2$.
Its slope is therefore $\mathcal N_t$, by the slope argument in that
theorem, which does not require regularity of the selected velocity.
The dynamic formula gives $|\mu'|_{\Wg}\le\mathcal N_t$; the slope
inequality and the energy identity give the reverse bound almost everywhere.
Integrating yields the metric energy-dissipation identity across rank changes.
\end{proof}

For the first mode in Figure~\ref{fig:gaussian-coupling}(b), $\alpha_1=0.1$,
so $R_{2,1}=0.4$, $R_{4,1}\simeq0.252$, and $R_{\infty,1}=0.2$.
All four plotted targets exceed the latter two bounds in absolute value,
and three exceed the first. These trajectories therefore require the
boundary continuation, not just the positive-definite covariance formula.

Figures~\ref{fig:gaussian-leaves} and~\ref{fig:gaussian-coupling} illustrate
these rank-one and rank-two reductions using the archived trajectories.

\begin{figure}[!ht]
\centering
\begin{tabular}{@{}c@{\hspace{.03\linewidth}}c@{}}
\includegraphics[width=.38\linewidth]{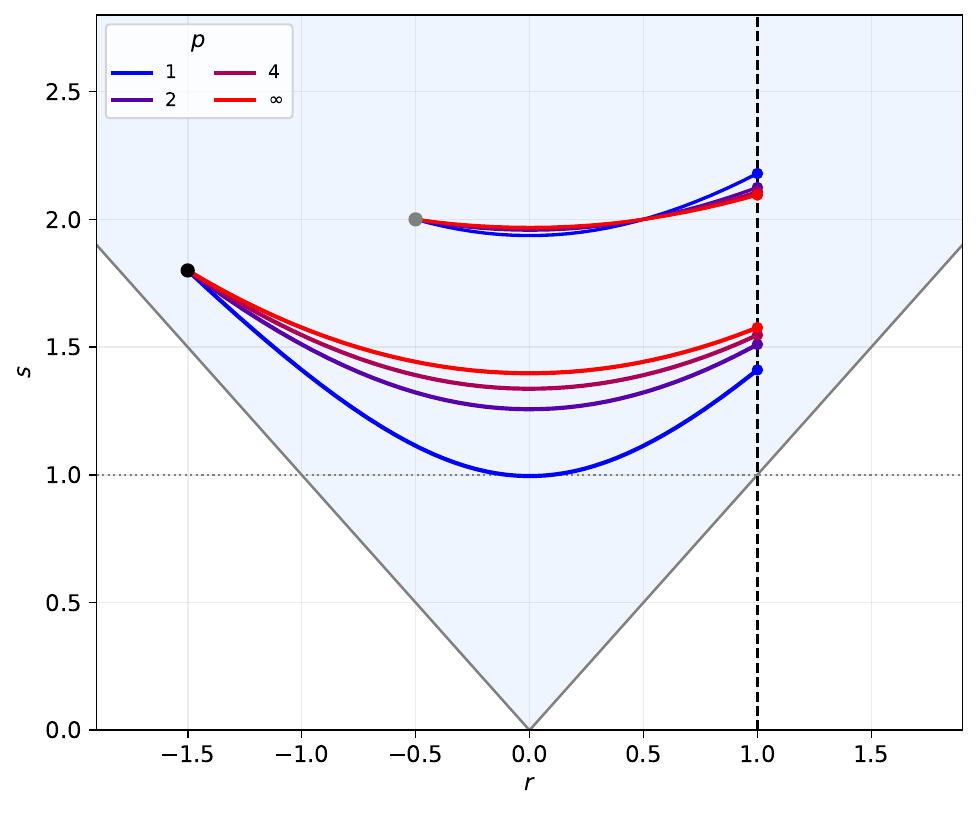}
&
\includegraphics[width=.38\linewidth]{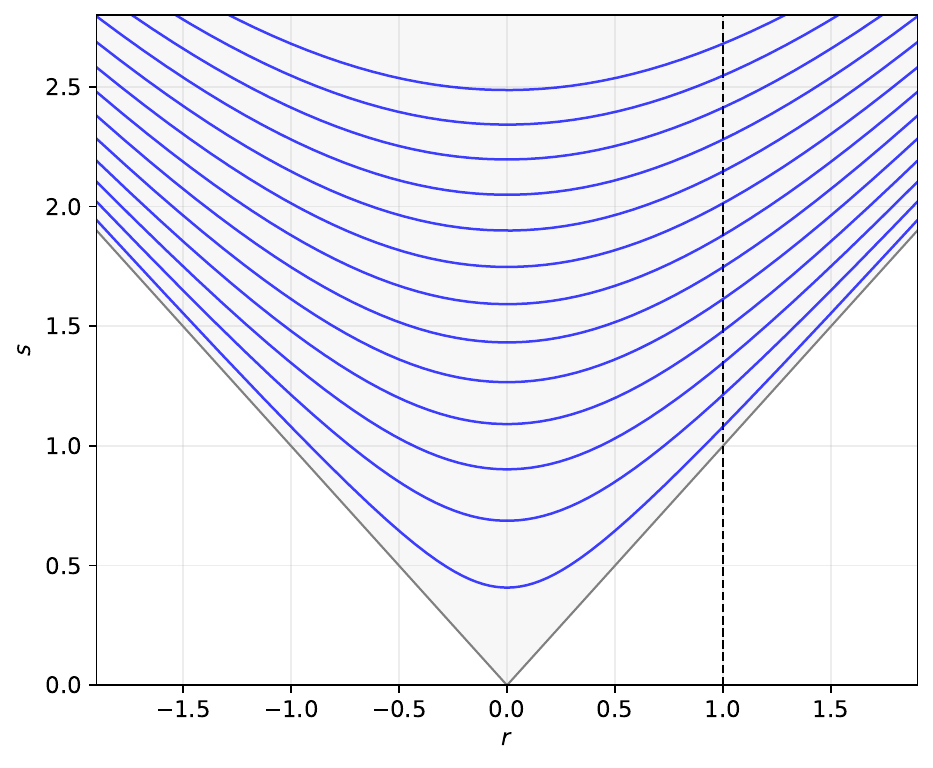}\\
\small (a) One-mode trajectories & \small (b) Conserved leaves, $p=1$
\end{tabular}
\caption{Gaussian one-mode geometry. (a) Trajectories in $(r,s)$ for
$p\in\{1,2,4,\infty\}$. (b) Interior conserved foliation for $p=1$.
After rank loss, trajectories use the continuation of
Proposition~\ref{app-prop:gaussian-boundary}; interior invariants no longer apply.}
\label{fig:gaussian-leaves}
\end{figure}

\begin{figure}[!ht]
\centering
\begin{tabular}{@{}c@{\hspace{.03\linewidth}}c@{}}
\includegraphics[width=.38\linewidth]{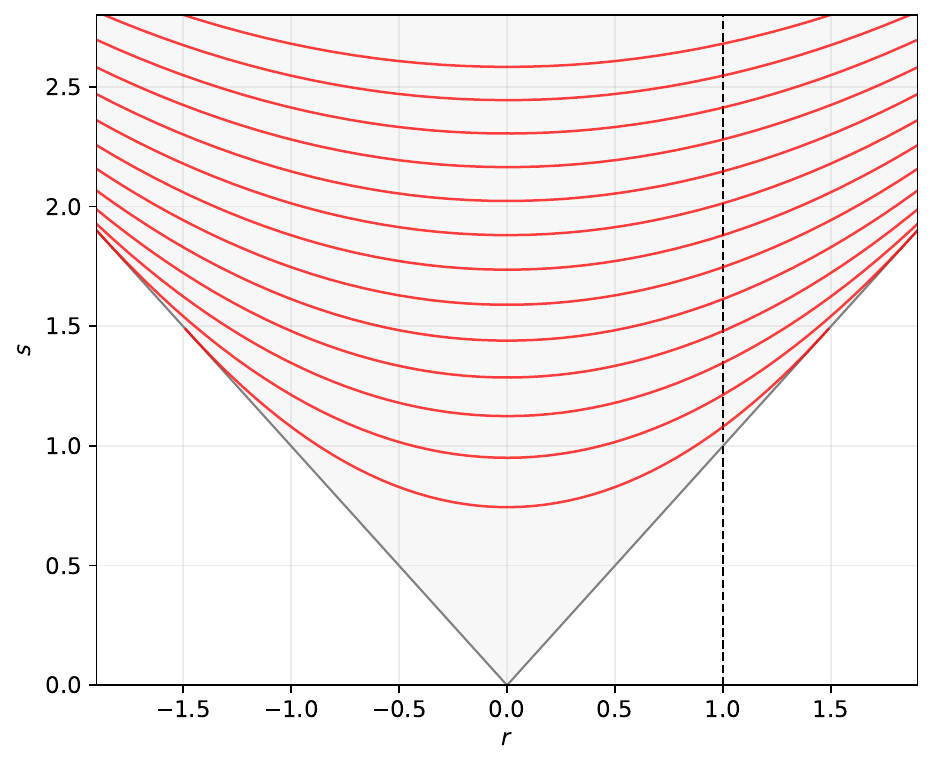}
&
% Match the plotting rectangles despite the unequal PDF margins.
\makebox[.38\linewidth]{\raisebox{.003373\linewidth}{\includegraphics[width=.321049\linewidth]{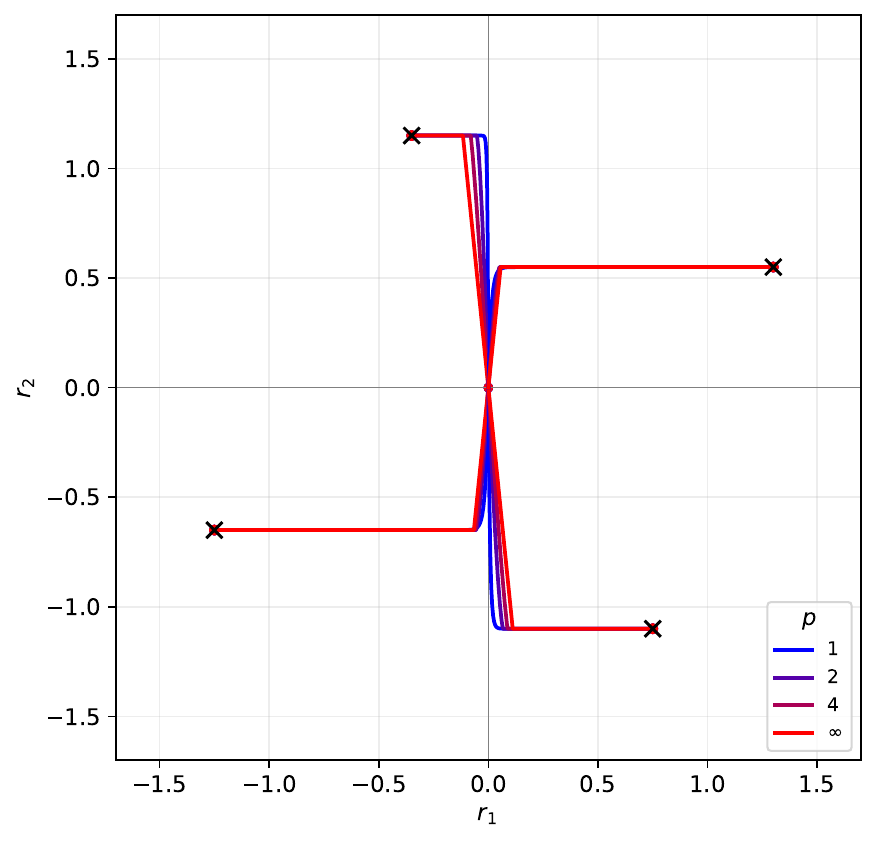}}}\\
\small (a) Conserved leaves, $p=\infty$ & \small (b) Two-mode dynamics
\end{tabular}
\caption{Gaussian operator leaves and coupled modes. (a) Interior conserved
foliation for $p=\infty$. (b) Rank-two trajectories in $(r_1,r_2)$, initially
$r_1=r_2=0$ and $(s_1,s_2)=(0.1,10)$, toward the indicated targets
$(\beta_1,\beta_2)$. The same rank-boundary continuation is used as in
Figure~\ref{fig:gaussian-leaves}.}
\label{fig:gaussian-coupling}
\end{figure}

For a rank-two target
$$
B_\rho=\diag(\beta_1,\beta_2,0,\dots,0),
$$
the first two modes solve the same equations but are coupled through
$$
\mathcal N_t^q=
\sum_{j=1}^2 \abs{r_j-\beta_j}^q\left(a_j^{q/2}+b_j^{q/2}\right).
$$
Thus the trace geometry evolves the two learned directions independently,
whereas every $p>1$ geometry shares a global normalization across them. This
is the simplest Gaussian setting where the matrix nature of the spectral flow
is visible at the predictor level rather than only through a time rescaling of
a single scalar mode. In Figure~\ref{fig:gaussian-coupling}(b), this coupling
appears as distinct planar paths toward different targets
$(\beta_1,\beta_2)$ under the same initialization.

%%%%
\FloatBarrier
\section{Numerical Protocols and Additional Examples}
\label{app:extra-static-dynamic}
\subsection{Time Normalization and Numerical Safeguards}
\label{sec:numerical-protocol}
The experiments are qualitative illustrations of the chosen geometry, not
wall-clock or generalization benchmarks. Their notebooks differentiate the
empirical loss $\MF(\mu_X)$, whereas Proposition~\ref{cor:particle} uses
$\mF(X)=n\MF(\mu_X)$. By homogeneity of the LMO, an unclipped notebook
step of size $\Delta t$ therefore corresponds to a mean-field Euler step
of size $\Delta t/n$.
Table~\ref{tab:numerical-clocks} lists the nominal step sizes and counts.
Each loss plot independently maps its run's nominal horizon to $[0,1]$;
these display coordinates are not common intrinsic flow times.

\begin{table}[ht]
\centering
\begin{tabular}{lrrrrr}
\toprule
Example & $n$ & steps $(1,\infty)$ & $\Delta t_1$ & $\Delta t_\infty$ & step cap\\
\midrule
ReLU MLP & 1920 & $(2000,2000)$ & 75 & 50 & 0.20\\
Attention & 200 & $(960,480)$ & 100 & 180 & 0.45\\
MMD & 200 & $(6000,6000)$ & 5 & 3.2 & none\\
\bottomrule
\end{tabular}
\caption{Nominal notebook clocks. Caps apply to the Frobenius norm of the
entire particle-matrix increment, not separately to particles.}
\label{tab:numerical-clocks}
\end{table}

The MLP and attention implementations multiply an oversized increment by
a scalar to enforce the stated cap, and additionally guard coordinates at
$\pm30$ and $\pm10$, respectively. The cap makes effective step sizes
state-dependent; an active coordinate guard can also change the direction.
The MMD experiment uses the smoothed distance
$\sqrt{\|x-y\|^2+0.01^2}$ and unclipped Euler steps.
The archived plots do not record cumulative effective times or all safeguard
activations, so no corrected continuum-clock comparison is inferred from
them. The slight difference between pre-update loss samples and post-update
snapshot labels is also a discretization convention, not a theoretical time
rescaling.

An initialization-only check of the MLP notebook (seed 7, double precision,
2000 training covariates) makes this distinction concrete. Writing
$G=\nabla_X\MF(\mu_X)$, the ideal mean-field dissipations are
$D_1=n\|G\|_F^2$ and $D_\infty=n\|G\|_*^2$. Their initial ratio is
$D_\infty/D_1\simeq2.55765$, whereas the nominal horizon ratio is $1.5$.
Both initial updates are already clipped, with factors approximately
$0.74469$ and $0.40326$; their first-order predicted loss-decrease ratio
per displayed step is consequently $0.92334$, not $2.55765$.
These initialization checks do not require running training. An initial
dissipation ratio alone cannot calibrate the entire trajectories.
We therefore retain the archived curves without relabeling them by this
ratio; their interpretation is qualitative.

For attention, the five paired trials redraw both teacher and student.
The effective-token statistics use all attention weights, whereas the
displayed weight histograms restrict to $[0,0.1]$ and renormalize there.
The Gaussian modal figures use the ODE of
Proposition~\ref{app-prop:modal-ode-paper}, integrated in double precision by
RK4 with $T=18$ and $\Delta t=2\times10^{-4}$. Negative numerical covariance
eigenvalues are projected to zero; extinct modes are not artificially revived.
We check the vector field and interior conserved leaves independently.
A separate event-resolving integration uses transformed variances to locate
rank-loss and zero-error events, with the continuation of
Proposition~\ref{app-prop:gaussian-boundary}. Across the 24 parameter cases,
it detects 11 rank-loss events. Reference refinement changes their times by
at most $4.9\times10^{-8}$. In local windows around all detected events,
refining the notebook step from $2\times10^{-4}$ to $5\times10^{-5}$ reduces
the maximum tested covariance-eigenvalue error from $3.5\times10^{-4}$ to
$8.6\times10^{-5}$. These are local consistency checks, not a certified
global error bound; in particular, classical fourth-order accuracy is not
assumed at nonsmooth events. This path interpretation is separate from the
empirical neural-network clocks.

The public reproduction runner records software versions, source and input
hashes, elapsed time, and host information for new runs. Complete hardware,
peak-memory, and elapsed-time records for the archived full experiments
are unavailable; recent smoke-test timings are not substitutes for them.

\FloatBarrier
\FloatBarrier
\subsection{Archived Diagnostics}
\label{sec:archived-diagnostics}
These panels show the same archived runs discussed in
Section~\ref{sec:wasserstein-flows}.
The nominal-time and safeguard qualifications of
Section~\ref{sec:numerical-protocol} apply throughout; the curves have not
been recalibrated or replaced by new simulations.

\begin{figure}[H]
\centering
\includegraphics[width=.72\linewidth]{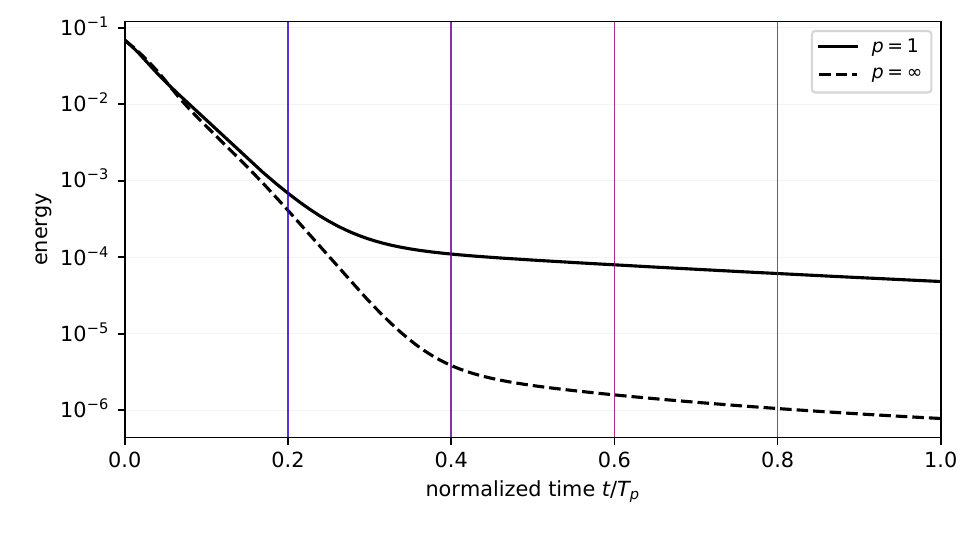}
\caption{ReLU MLP energy on separately normalized nominal horizons:
solid $p=1$, dashed $p=\infty$. Colored vertical lines identify the snapshots
in Figure~\ref{fig:mlp-histograms}; equal horizontal coordinates do not
represent equal intrinsic flow times. Particle trajectories appear in
Figure~\ref{fig:mlp-trajectories}.}
\label{fig:mlp-loss}
\end{figure}

\begin{figure}[H]
\centering
% Split the archived grid in its inter-row whitespace; align the plot baselines.
\raisebox{.024\linewidth}{\includegraphics[width=.46\linewidth,trim=0 260bp 0 0,clip]{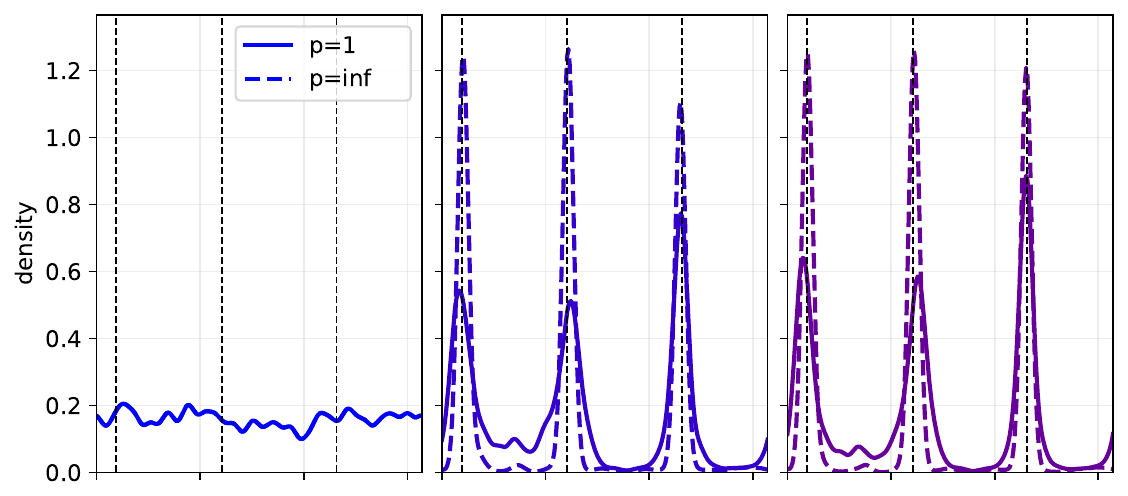}}%
\hspace{.02\linewidth}%
\includegraphics[width=.46\linewidth,trim=0 0 0 238bp,clip]{figures/fig-mlp/mlp_relu_angular_hist_snapshots_camera_ready.pdf}
\caption{Circular histogram snapshots for the ReLU MLP, in chronological
order from left to right, with the same angular axis in all panels. Solid curves correspond to
$p=1$ and dashed curves to $p=\infty$ (Muon); colors match the snapshot
markers in Figure~\ref{fig:mlp-loss}. Dashed vertical lines mark teacher
directions. Only the legacy legend text has been corrected.}
\label{fig:mlp-histograms}
\end{figure}

\begin{figure}[H]
\centering
\begin{tabular}{@{}c@{\hspace{.03\linewidth}}c@{}}
\includegraphics[width=.46\linewidth]{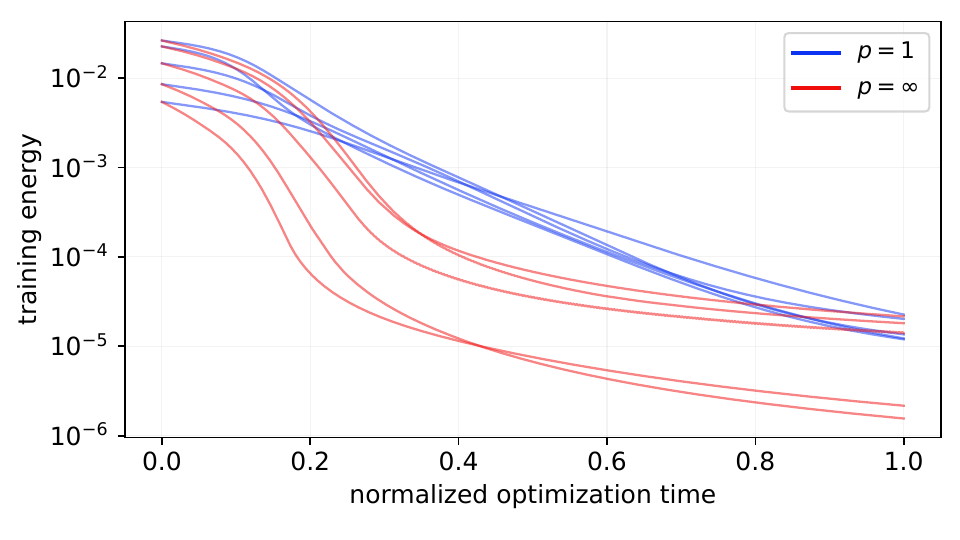}
&
\includegraphics[width=.46\linewidth]{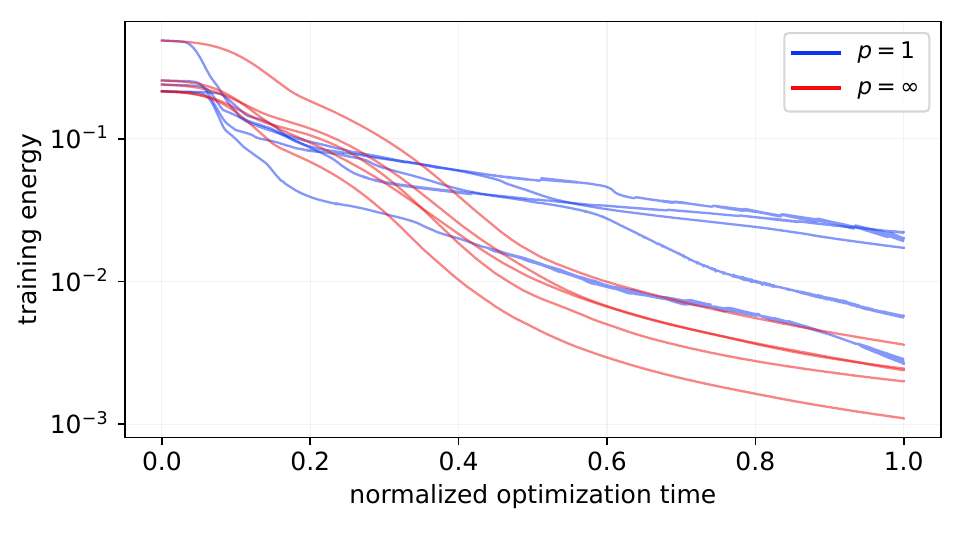}\\
\small (a) Dense attention & \small (b) Sparse attention
\end{tabular}
\caption{Shallow-attention loss curves for five paired trials per regime,
with $p=1$ in blue and $p=\infty$ in red. Each trial redraws both teacher
and student. Effective support is approximately 16 tokens in the dense
regime and 2.1 in the sparse regime. The horizontal axes use separately
normalized nominal horizons, not matched intrinsic flow times.}
\label{fig:attention-trajectories}
\end{figure}

\begin{figure}[H]
\centering
\includegraphics[width=\linewidth]{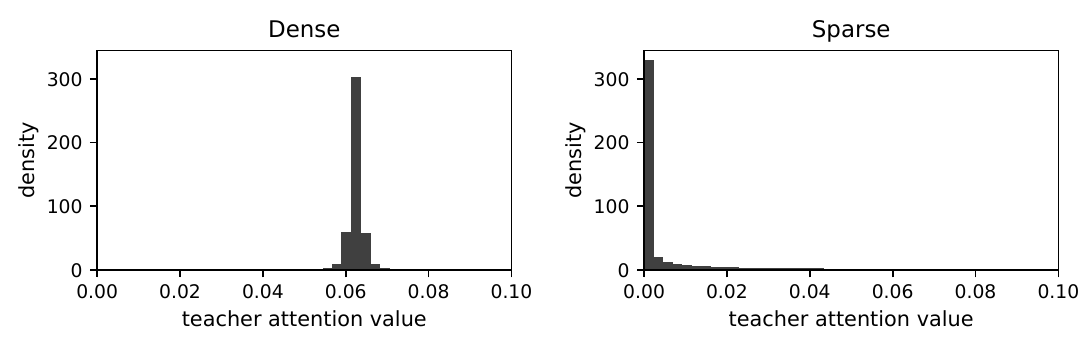}
\caption{Teacher-attention histograms for the dense and sparse regimes.
Weights are truncated to $[0,0.1]$ and renormalized on that interval;
sparse weights larger than $0.1$ are not displayed. Effective-token
statistics in the text instead use all attention weights. Both panels use
the same horizontal scale.}
\label{fig:attention-histograms}
\end{figure}

\begin{figure}[H]
\centering
\includegraphics[width=.44\linewidth]{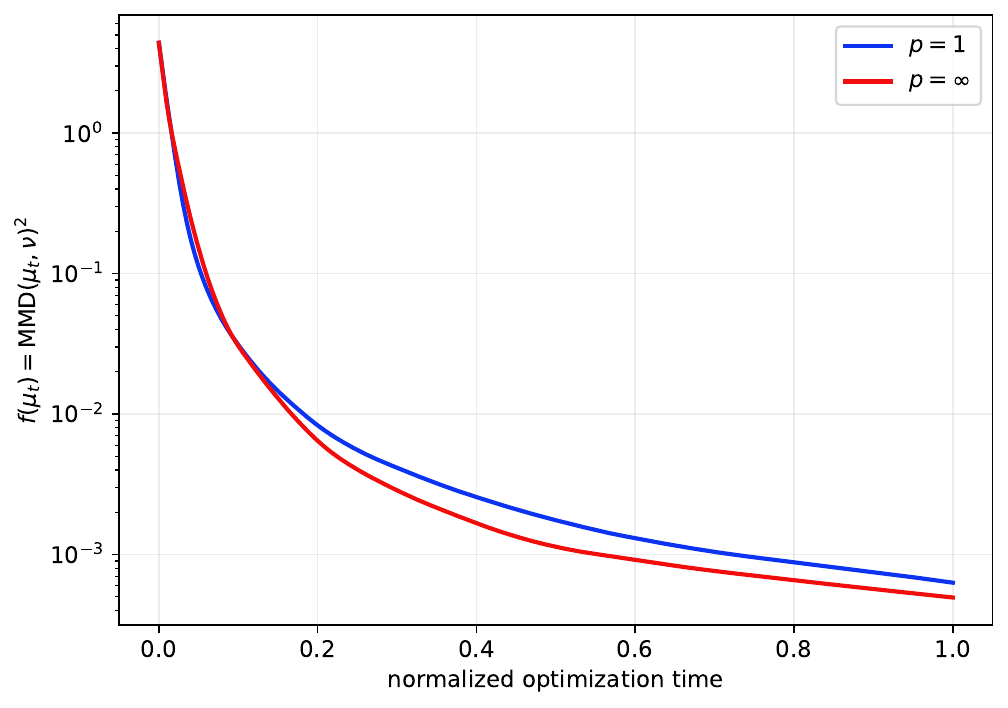}
\caption{Archived MMD objective values for the trajectories in
Figure~\ref{fig:mmd-traj}. Nominal horizons are normalized separately for
$p=1$ and $p=\infty$. These curves do not establish an intrinsic
convergence-rate or wall-clock ordering.}
\label{fig:mmd-loss}
\end{figure}

\subsection{Couplings and Interpolations}
This subsection complements Figures~\ref{fig:static-pair} and
\ref{fig:static-interp-dynamic} with additional planar shape pairs.
Each example uses an image-defined source and a translated image-defined
target. Trace couplings are exact assignment solutions. Operator couplings
are obtained by alternating a projected assignment with a rank-one cost
update, for at most 40 iterations from a horizontal-axis initialization.
This heuristic supplies no primal--dual gap certificate and is not asserted
to solve the full convex SDP. Accordingly, its interpolations are
illustrative displacement paths, not certified $\Wg$-geodesics. Rotation
invariance holds for the underlying Schatten transport problem, not
necessarily for this finitely iterated, fixed-initialization solver.

Figure~\ref{fig:extra-static-pairs} shows additional static couplings for three image-defined source and target pairs, complementing the annulus-to-disks example in the main text.

\begin{figure}[!ht]
\centering
\begin{tabular}{@{}cc@{}}
\includegraphics[width=.47\linewidth]{figures/fig-static/bunny-trefle_static_pairing_p1.pdf}
&
\includegraphics[width=.47\linewidth]{figures/fig-static/bunny-trefle_static_pairing_pinf.pdf}
\\[-.2em]
\small Bunny to clover, $p=1$
&
\small Bunny to clover, $p=\infty$
\\[.5em]
\includegraphics[width=.47\linewidth]{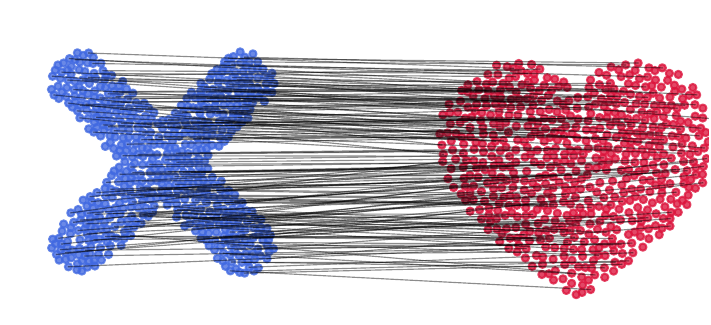}
&
\includegraphics[width=.47\linewidth]{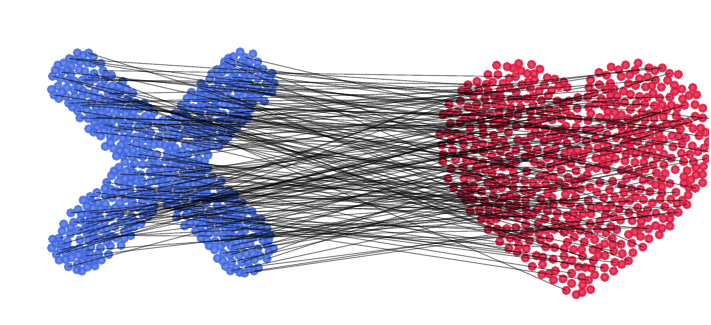}
\\[-.2em]
\small Cross to heart, $p=1$
&
\small Cross to heart, $p=\infty$
\\[.5em]
\includegraphics[width=.47\linewidth]{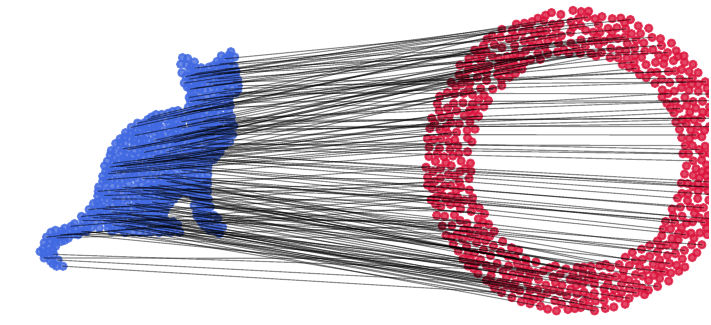}
&
\includegraphics[width=.47\linewidth]{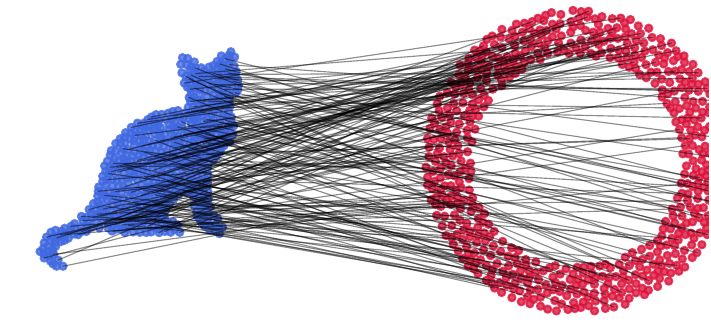}
\\[-.2em]
\small Cat to annulus, $p=1$
&
\small Cat to annulus, $p=\infty$
\end{tabular}
\caption{Additional static Spectral Wasserstein couplings. Only a subsample of matching segments is displayed.}
\label{fig:extra-static-pairs}
\end{figure}

\FloatBarrier

Figure~\ref{fig:extra-static-interp} displays the corresponding displacement-interpolation density snapshots, using the same computed couplings as in Figure~\ref{fig:extra-static-pairs}.
\begin{figure}[H]
\centering
\resizebox{\linewidth}{!}{%
\begin{tabular}{@{}cccccccccc@{}}
\includegraphics[width=.098\linewidth]{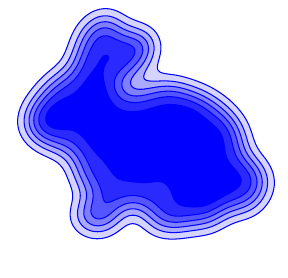}
&
\includegraphics[width=.098\linewidth]{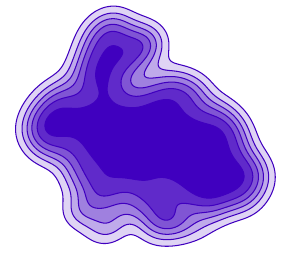}
&
\includegraphics[width=.098\linewidth]{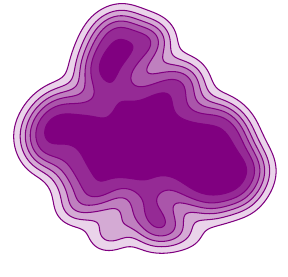}
&
\includegraphics[width=.098\linewidth]{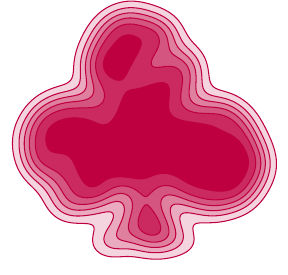}
&
\includegraphics[width=.098\linewidth]{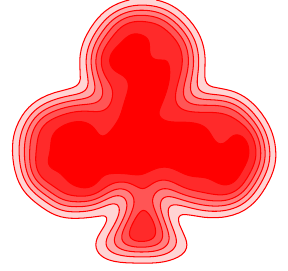}
&
\includegraphics[width=.098\linewidth]{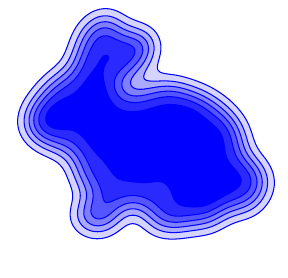}
&
\includegraphics[width=.098\linewidth]{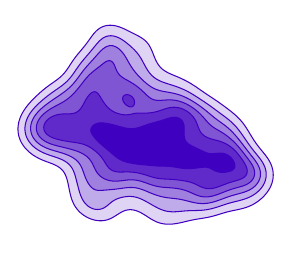}
&
\includegraphics[width=.098\linewidth]{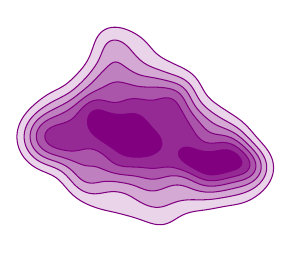}
&
\includegraphics[width=.098\linewidth]{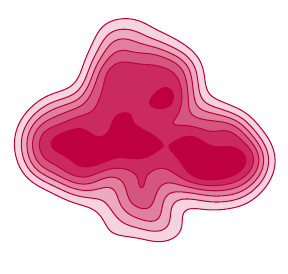}
&
\includegraphics[width=.098\linewidth]{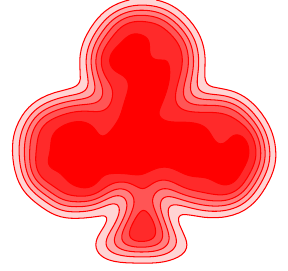}
\\[-.2em]
& & \small $p=1$ & & & & & \small $p=\infty$ & &
\\[.35em]
\includegraphics[width=.098\linewidth]{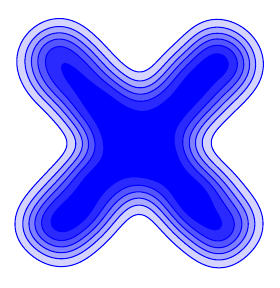}
&
\includegraphics[width=.098\linewidth]{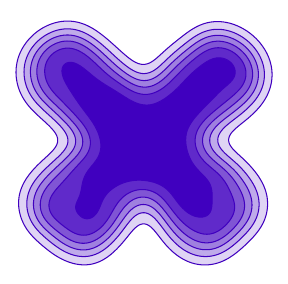}
&
\includegraphics[width=.098\linewidth]{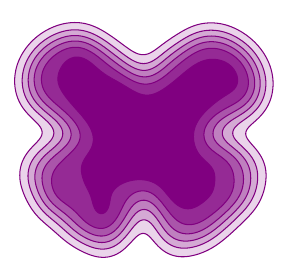}
&
\includegraphics[width=.098\linewidth]{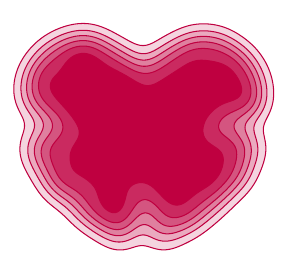}
&
\includegraphics[width=.098\linewidth]{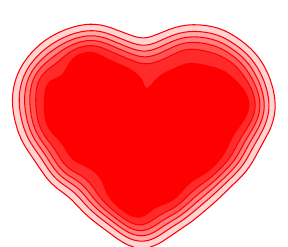}
&
\includegraphics[width=.098\linewidth]{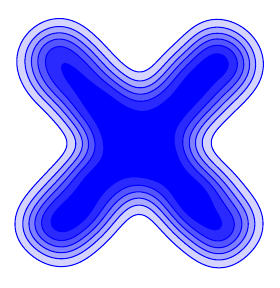}
&
\includegraphics[width=.098\linewidth]{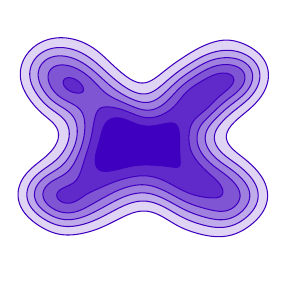}
&
\includegraphics[width=.098\linewidth]{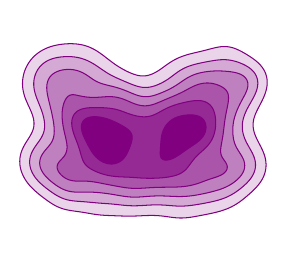}
&
\includegraphics[width=.098\linewidth]{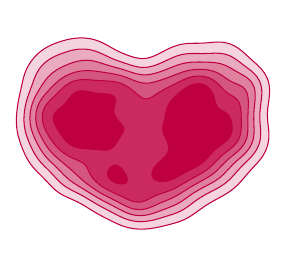}
&
\includegraphics[width=.098\linewidth]{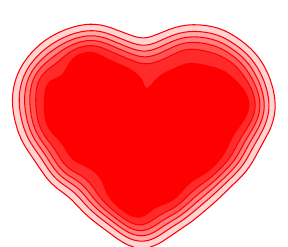}
\\[-.2em]
& & \small $p=1$ & & & & & \small $p=\infty$ & &
\\[.35em]
\includegraphics[width=.098\linewidth]{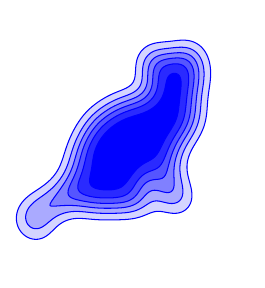}
&
\includegraphics[width=.098\linewidth]{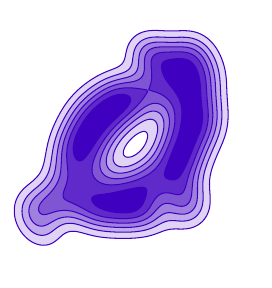}
&
\includegraphics[width=.098\linewidth]{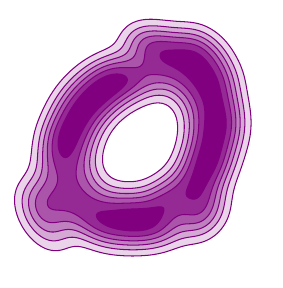}
&
\includegraphics[width=.098\linewidth]{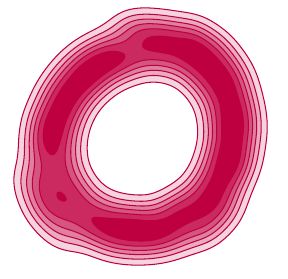}
&
\includegraphics[width=.098\linewidth]{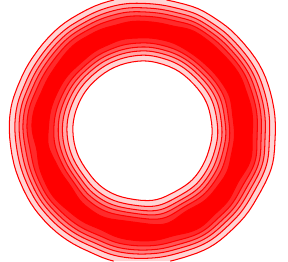}
&
\includegraphics[width=.098\linewidth]{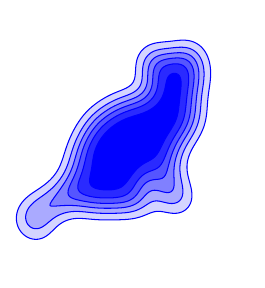}
&
\includegraphics[width=.098\linewidth]{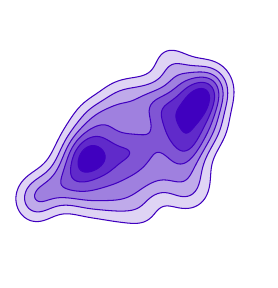}
&
\includegraphics[width=.098\linewidth]{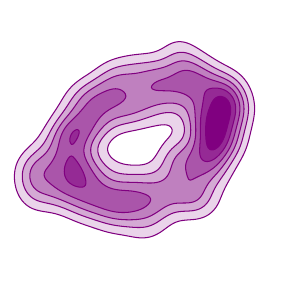}
&
\includegraphics[width=.098\linewidth]{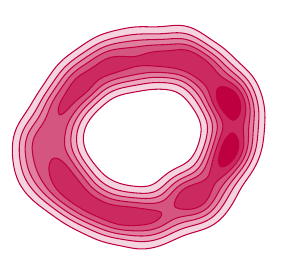}
&
\includegraphics[width=.098\linewidth]{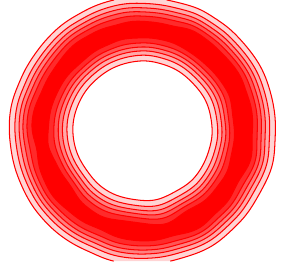}
\\[-.2em]
& & \small $p=1$ & & & & & \small $p=\infty$ & &
\end{tabular}
}
\caption{Additional displacement-interpolation density snapshots for the same shape pairs as Figure~\ref{fig:extra-static-pairs}.}
\label{fig:extra-static-interp}
\end{figure}

\section{Geodesic Convexity}
\label{app:convexity}
This section studies one of the basic structural notions behind metric gradient flows, namely geodesic convexity, for the Spectral Wasserstein geometries introduced above. Geodesic convexity is desirable because it supports stability and global convergence statements \citep{AmbrosioGigliSavare2008}, but for many machine-learning objectives it is arguably too strong, so here it is used as a structural benchmark rather than as a blanket modeling assumption.

\begin{proposition}[Basic displacement-convexity examples]
\label{prop:basic-geodesic-convexity}
Assume that $\gamma$ is monotone.
A linear functional $\mathcal F(\mu)=\int h\,d\mu$, with $h:\R^d\to\R$, is
convex along every static displacement geodesic if and only if $h$ is convex.
A quadratic interaction functional
$\mathcal F(\mu)=\int h(x,y)\,d\mu(x)d\mu(y)$, with
$h:\R^d\times\R^d\to\R$, is convex along every static displacement geodesic
if $h$ is convex.
Finally, let $d\nu=Z^{-1}e^{-V}dx$, with $V\in C^2(\R^d)$ convex, and define
the relative entropy by $\mathcal F(\mu)=\int\log(d\mu/d\nu)\,d\mu$ if
$\mu\ll\nu$, and $\mathcal F(\mu)=+\infty$ otherwise. If the representing set
$\cK$ contains at least one positive definite matrix, then for every
$\mu_0,\mu_1\in\cP_2(\R^d)$ there exists a constant-speed $\Wg$-geodesic
$(\mu_t)_{t\in[0,1]}$ joining them along which $\mathcal F$ is convex.
\end{proposition}

\subsection{Definition and General Setup}
This subsection recalls the relevant notion of convexity along Spectral Wasserstein geodesics and explains why it matters for the variational analysis of the flow.
When one studies gradient flows in metric spaces, geodesic convexity plays the
same role as ordinary convexity in Euclidean optimization: it governs
uniqueness, stability, and the variational structure of the dynamics.
Corollary~\ref{cor:geodesic} provides a distinguished class of geodesics from
optimal couplings, for which convexity inequalities can be tested directly
along Euclidean segments.

The next definition fixes the notion used in the remainder of this section.
\begin{definition}[Geodesic convexity]
\label{def:geoconvex}
A functional $F:\cP_2(\R^d)\to(-\infty,+\infty]$ is called $\kappa$-geodesically convex for $\Wg$ if along every constant-speed $\Wg$-geodesic $(\mu_t)_{t\in[0,1]}$ one has
$$
F(\mu_t)
\le
(1-t)F(\mu_0)+tF(\mu_1)
-\frac{\kappa}{2}t(1-t)\Wg(\mu_0,\mu_1)^2.
$$
\end{definition}

We call a geodesic a \emph{static displacement geodesic} if it has the form
$\mu_t=((1-t)x+ty)_\#\pi_*$ for a static optimal coupling $\pi_*$, as in
Corollary~\ref{cor:geodesic}. Full geodesic convexity concerns every
constant-speed geodesic and is stronger when the metric admits branching.

\begin{remark}[Branching at the operator endpoint]
This distinction is genuine for the Muon geometry. Let
$\gamma(S)=\lambda_{\max}(S)$ on $\mathbb S_+^2$, choose a nonzero
$1$-Lipschitz
function $f:[0,1]\to\R$ with $f(0)=f(1)=0$, and set
\[
 \mu_t=\tfrac12\delta_{(t,f(t))}+\tfrac12\delta_{(t,-f(t))}.
\]
For $s<t$, matching equal signs gives displacement covariance
$\diag((t-s)^2,(f(t)-f(s))^2)$ and therefore cost $(t-s)^2$. Conversely, the
change of the first-coordinate mean gives the same lower bound for every
coupling. Hence $(\mu_t)$ is a constant-speed $\Wg$-geodesic from
$\delta_{(0,0)}$ to $\delta_{(1,0)}$, but it is not a static displacement
geodesic when $f\ne0$. For the convex function $h(x)=x_2^2$, one has
$\int h\,d\mu_0=\int h\,d\mu_1=0$ but
$\int h\,d\mu_t=f(t)^2>0$ whenever $f(t)\ne0$. Thus convexity of $h$ ensures
convexity along every static displacement geodesic, but not full geodesic
convexity at this endpoint.
\end{remark}

\subsection{Linear Functionals}
This subsection shows that for linear functionals the qualitative and quantitative convexity criteria reduce to explicit pointwise conditions on the integrand.
For a Borel function $h:\R^d\to\R$, define the linear functional
$$
F_h(\mu)\coloneqq \int_{\R^d} h(x)\,d\mu(x).
$$
Because static displacement geodesics are given by affine interpolation in
space, the convexity question for $F_h$ along this class is especially
transparent.

The next theorem gives the exact criterion.
\begin{theorem}[Linear functionals along displacement geodesics]
\label{thm:linearconvex}
Assume $\gamma$ is monotone and let $\cK$ be a convex compact PSD representing set for $\gamma$ as in Proposition~\ref{prop:psdpolar}.
\begin{enumerate}
\item The functional $F_h$ is convex along every static displacement geodesic
if and only if $h$ is convex on $\R^d$.
\item If in addition $h\in C^2(\R^d)$ and $\kappa\ge0$, then $F_h$ satisfies
the $\kappa$-convexity inequality along every static displacement geodesic if
and only if
$$
\nabla^2 h(z)\succeq \kappa Q
\qquad
\text{for every }z\in\R^d,\ Q\in\cK.
$$
Equivalently,
$$
\xi^\top \nabla^2 h(z)\,\xi
\ge
\kappa\max_{Q\in\cK}\xi^\top Q\,\xi
\qquad
\text{for every }z,\xi\in\R^d.
$$
For $\kappa<0$, this displayed scalar reformulation with a maximum is not equivalent to the matrix inequality; multiplication by the negative scalar reverses the extremum, so the corresponding pointwise scalar condition would involve $\min_{Q\in\cK}\xi^\top Q\xi$. No equivalence for negative $\kappa$ is asserted here.
\end{enumerate}
\end{theorem}

\begin{proof}
Assume first that $h$ is convex and let
$$
\mu_t=((1-t)x+ty)_\#\pi_*
$$
be any static displacement geodesic. Then
$$
F_h(\mu_t)
=
\int h((1-t)x+ty)\,d\pi_*(x,y).
$$
By the convexity of $h$,
$$
h((1-t)x+ty)\le (1-t)h(x)+th(y),
$$
and integrating yields geodesic convexity of $F_h$.

Conversely, apply the displacement-convexity assumption to Dirac masses
$\mu_0=\delta_x$ and $\mu_1=\delta_y$. Their static displacement geodesic is
$$
\mu_t=\delta_{(1-t)x+ty},
$$
so
$$
h((1-t)x+ty)=F_h(\mu_t)\le (1-t)h(x)+th(y),
$$
which proves convexity of $h$.

Assume now that $h\in C^2$ and that
$$
\nabla^2 h(z)\succeq \kappa Q
\qquad
\text{for every }z\in\R^d,\ Q\in\cK.
$$
For each fixed $Q\in\cK$, the function
$z\mapsto h(z)-\frac{\kappa}{2}z^\top Qz$ is convex. Hence
\[
 h((1-t)x+ty)\le (1-t)h(x)+th(y)
 -\frac{\kappa}{2}t(1-t)(y-x)^\top Q(y-x).
\]
Integrate against any static optimal coupling $\pi_*$ and then maximize
$\tr(Q\Sigma_{\pi_*})$ over $Q\in\cK$. Since $\kappa\ge0$, this gives
\[
 F_h(\mu_t)\le (1-t)F_h(\mu_0)+tF_h(\mu_1)
 -\frac{\kappa}{2}t(1-t)\Wg(\mu_0,\mu_1)^2.
\]
No differentiation under the integral is needed. Convexity of $h$ gives an
affine lower bound, so these integrals are well-defined on $\cP_2$;
the argument applies to finite endpoint energies, with the other cases
understood in the extended-value sense.

Conversely, test the midpoint inequality between
$\delta_{z-\varepsilon\xi}$ and $\delta_{z+\varepsilon\xi}$. It gives
\[
 h(z)
 \le
 \frac{h(z-\varepsilon\xi)+h(z+\varepsilon\xi)}{2}
 -\frac{\kappa}{2}\varepsilon^2\gamma(\xi\xi^\top).
\]
Subtracting $h(z)$, dividing by $\varepsilon^2/2$, and sending
$\varepsilon\to0$ yields
$$
\xi^\top \nabla^2 h(z)\,\xi
\ge
\kappa\gamma(\xi\xi^\top)
=
\kappa\max_{Q\in\cK}\xi^\top Q\,\xi.
$$
Because $\kappa\ge0$, this scalar inequality is equivalent to the matrix inequality
$$
\nabla^2 h(z)\succeq \kappa Q
\qquad
\text{for every }Q\in\cK.
$$
\end{proof}

The next remark interprets the condition of Theorem~\ref{thm:linearconvex} for the Schatten family.
\begin{remark}[Schatten norms]
For the Schatten geometries used throughout the paper and for $\kappa\ge0$, the condition
$$
\nabla^2 h(z)\succeq \kappa Q
\qquad
\text{for every }Q\in\cK
$$
is equivalent to
$$
\nabla^2 h(z)\succeq \kappa \Id.
$$
Indeed, for $p=1$ we may choose $\cK=\{\Id\}$. For $1<p\le \infty$, a convenient choice is
$$
\cK=\{Q\succeq 0:\ \norm{Q}_{S_{p'}}\le 1\},
\qquad
p'=\frac{p}{p-1},
$$
with the usual interpretation $p'=1$ when $p=\infty$. Under the standing
condition $\kappa\ge0$ in Theorem~\ref{thm:linearconvex}, every such $Q$
satisfies $Q\preceq\Id$, and every rank-one projector belongs to $\cK$; hence
the condition is equivalent to $\nabla^2h(z)\succeq\kappa\Id$, exactly as in
the classical $\Wtwo$ geometry. For negative $\kappa$, this simplification is
not asserted.
\end{remark}

\subsection{Interaction Functionals}

The same segmentwise argument applies to convex interactions.
\begin{proposition}[Convex interactions]
Let $h:\R^d\times\R^d\to\R$ be jointly convex and define
\[
 F_h^{\rm int}(\mu)=\int h(x,x')\,d\mu(x)d\mu(x').
\]
Then $F_h^{\rm int}$ is convex along every static displacement geodesic.
\end{proposition}

\begin{proof}
Let $\mu_t=((1-t)x+ty)_\#\pi_*$ and take two independent copies
$(x,y),(x',y')$ with law $\pi_*$. Joint convexity gives
\[
 h((1-t)x+ty,(1-t)x'+ty')
 \le (1-t)h(x,x')+t h(y,y').
\]
Integrating with respect to $\pi_*\otimes\pi_*$ proves the claim.
\end{proof}

\subsection{Relative Entropy}

This subsection examines the geodesic convexity of relative entropy, which is a fundamental nonlinear example underlying diffusion-type gradient flows.
Section~\ref{sec:relative-entropy-gaussians} treats in detail the case of a
Gaussian target $\nu$. For the canonical selector, the Gaussian ansatz closes
as an ODE for the mean and covariance.
Let
$$
d\nu(x)=Z^{-1}e^{-V(x)}\,dx,
\qquad
V\in C^2(\R^d),
$$
and define the relative entropy
$$
\operatorname{Ent}_\nu(\mu)\coloneqq 
\begin{cases}
\displaystyle \int \log\!\left(\frac{d\mu}{d\nu}\right)\,d\mu, & \mu\ll \nu,\\[0.6em]
+\infty, & \text{otherwise}.
\end{cases}
$$
The trace case is the classical displacement-convexity theory of entropy. For general spectral geometries the same argument works only when the active quadratic costs remain uniformly elliptic, and otherwise one only gets a weaker statement by approximation.

The next theorem records the all-geodesics result in the uniformly elliptic
regime.
\begin{theorem}[Full entropy convexity under uniform ellipticity]
\label{thm:entropyfull}
Assume $\gamma$ is monotone and that $\cK$ can be chosen so that
$$
\cK\subset \Sppd.
$$
Assume moreover that
$$
\nabla^2 V(x)\succeq \kappa Q
\qquad
\text{for every }x\in\R^d,\ Q\in\cK.
$$
Then $\operatorname{Ent}_\nu$ is $\kappa$-geodesically convex on $(\cP_2(\R^d),\Wg)$.
\end{theorem}

\begin{proof}
Let $(\mu_t)$ be any constant-speed $\Wg$-geodesic and set
$D=\Wg(\mu_0,\mu_1)$. Choose a dual maximizer $Q_*$ for its endpoints.
Then $Q_*\succ0$ and $\WQ{Q_*}(\mu_0,\mu_1)=D$. For $0\le s\le t\le1$,
the triangle inequality for $\WQ{Q_*}$ and the bound $\WQ{Q_*}\le\Wg$ give
\[
\begin{aligned}
 D&\le \WQ{Q_*}(\mu_0,\mu_s)+\WQ{Q_*}(\mu_s,\mu_t)
       +\WQ{Q_*}(\mu_t,\mu_1)\\
  &\le sD+(t-s)D+(1-t)D=D.
\end{aligned}
\]
Equality throughout implies $\WQ{Q_*}(\mu_s,\mu_t)=(t-s)D$.
Thus the given curve is also a constant-speed $\WQ{Q_*}$-geodesic.
With $L=Q_*^{1/2}$, set $\widetilde\mu_t=L_\sharp\mu_t$ and
$\widetilde\nu=L_\sharp\nu$. The transformed curve is a $\Wtwo$-geodesic and
$$
\Wtwo(\widetilde\mu_0,\widetilde\mu_1)^2=\Wg(\mu_0,\mu_1)^2.
$$

The reference measure $\widetilde\nu$ has density
$$
d\widetilde\nu(z)=\widetilde Z^{-1}e^{-\widetilde V(z)}\,dz,
\qquad
\widetilde V(z)=V(L^{-1}z)+\text{constant},
$$
so
$$
\nabla^2\widetilde V(z)=L^{-1}\nabla^2V(L^{-1}z)L^{-1}.
$$
Because $Q_*\in\cK$, the curvature assumption gives
$$
\nabla^2V(x)\succeq \kappa Q_*=\kappa L^2,
$$
hence
$$
\nabla^2\widetilde V(z)\succeq \kappa \Id.
$$
The classical entropy convexity theorem for $\Wtwo$ therefore yields
$$
\operatorname{Ent}_{\widetilde\nu}(\widetilde\mu_t)
\le
(1-t)\operatorname{Ent}_{\widetilde\nu}(\widetilde\mu_0)
+t\operatorname{Ent}_{\widetilde\nu}(\widetilde\mu_1)
-\frac{\kappa}{2}t(1-t)\Wtwo(\widetilde\mu_0,\widetilde\mu_1)^2.
$$
Finally, relative entropy is invariant under the invertible change of variables $L$, so
$$
\operatorname{Ent}_{\widetilde\nu}(\widetilde\mu_t)=\operatorname{Ent}_\nu(\mu_t),
$$
and similarly at the endpoints. This gives the result.
\end{proof}

The next theorem explains what remains true in the more relevant but singular regime where the condition $\cK\subset \Sppd$ fails. This lack of ellipticity is precisely what happens for Schatten-$p$ geometries with $p>1$, because the corresponding representing sets contain singular matrices, but one still retains a weaker geodesic-convexity statement by approximation.
\begin{theorem}[Weak entropy convexity]
\label{thm:entropyweak}
Assume $\gamma$ is monotone, that $\cK\subset \mathbb S_+^d$ is convex compact, and that
$$
\cK\cap \Sppd\neq \varnothing.
$$
Assume moreover that
$$
\nabla^2V(x)\succeq \kappa Q
\qquad
\text{for every }x\in\R^d,\ Q\in\cK.
$$
Then for every $\mu_0,\mu_1\in\cP_2(\R^d)$ there exists at least one constant-speed $\Wg$-geodesic $(\mu_t)_{t\in[0,1]}$ from $\mu_0$ to $\mu_1$ such that
$$
\operatorname{Ent}_\nu(\mu_t)
\le
(1-t)\operatorname{Ent}_\nu(\mu_0)+t\operatorname{Ent}_\nu(\mu_1)
-\frac{\kappa}{2}t(1-t)\Wg(\mu_0,\mu_1)^2.
$$
\end{theorem}

\begin{proof}
Choose $Q_0\in\cK\cap \Sppd$ and define
$$
\mathcal K_{\gamma,\varepsilon}\coloneqq (1-\varepsilon)\cK+\varepsilon Q_0,
\qquad
\varepsilon\in(0,1).
$$
Then every element of $\mathcal K_{\gamma,\varepsilon}$ is positive definite. Let $\gamma_\varepsilon$ be the support function of $\mathcal K_{\gamma,\varepsilon}$, namely
$$
\gamma_\varepsilon(S)=\max_{Q\in\mathcal K_{\gamma,\varepsilon}}\tr(QS)
=(1-\varepsilon)\gamma(S)+\varepsilon\tr(Q_0S).
$$
Hence
$$
(1-\varepsilon)\gamma(S)\le \gamma_\varepsilon(S)\le \gamma(S)
$$
for all $S\succeq 0$, and therefore
$$
(1-\varepsilon)\Wg(\mu_0,\mu_1)^2
\le
\mathsf W_{\gamma_\varepsilon}(\mu_0,\mu_1)^2
\le
\Wg(\mu_0,\mu_1)^2.
$$
Moreover, convexity of $\cK$ gives
$\mathcal K_{\gamma,\varepsilon}\subset\cK$, so the curvature assumption is
unchanged.

By Theorem~\ref{thm:entropyfull}, each regularized geometry has full entropy convexity. Choose an optimal coupling $\pi_\varepsilon$ for $\mathsf W_{\gamma_\varepsilon}$ and set
$$
\mu_t^\varepsilon=((1-t)x+ty)_\#\pi_\varepsilon.
$$
Then
$$
\operatorname{Ent}_\nu(\mu_t^\varepsilon)
\le
(1-t)\operatorname{Ent}_\nu(\mu_0)+t\operatorname{Ent}_\nu(\mu_1)
-\frac{\kappa}{2}t(1-t)\mathsf W_{\gamma_\varepsilon}(\mu_0,\mu_1)^2.
$$
By compactness of couplings, one may extract a subsequence $\pi_{\varepsilon_n}\rightharpoonup \pi_*$. Setting
$$
\mu_t=((1-t)x+ty)_\#\pi_*,
$$
the lower semicontinuity of relative entropy and the convergence
$$
\mathsf W_{\gamma_{\varepsilon_n}}(\mu_0,\mu_1)^2\to \Wg(\mu_0,\mu_1)^2
$$
yield the claimed inequality along the limit geodesic. The limit coupling $\pi_*$ is $\Wg$-optimal by continuity of the displacement covariance and of the support functions on trace-bounded sets.
\end{proof}

The next proposition shows that, in the operator case, full geodesic convexity of relative entropy fails in a very concrete way.
\begin{proposition}[Operator-norm counterexample]
\label{prop:entropy-counterexample}
Let $\gamma(S)=\lambda_{\max}(S)$ on $\mathbb S_+^2$, and let
$$
d\nu(x)=\frac{1}{2\pi}e^{-|x|^2/2}\,dx.
$$
Then $\operatorname{Ent}_\nu$ is not geodesically convex on $(\cP_2(\R^2),\mathsf W_{\gamma})$.
\end{proposition}

\begin{proof}
Fix parameters $L>0$ and $a>0$ such that
$$
\frac{4a^2}{3}\le L^2.
$$
Set
$$
R_0=[0,1]\times[-a,a],
\qquad
R_1=[L,L+1]\times[-a,a],
$$
and let
$$
\mu_i=\frac{1}{2a}\mathbf 1_{R_i}\,dx,
\qquad i=0,1.
$$
Both endpoints are absolutely continuous with bounded support, hence satisfy
$$
\operatorname{Ent}_\nu(\mu_0)<+\infty,
\qquad
\operatorname{Ent}_\nu(\mu_1)<+\infty.
$$

For any coupling $\pi$ between $\mu_0$ and $\mu_1$, define
$$
S_\pi=\int (y-x)(y-x)^\top\,d\pi(x,y).
$$
Then
$$
\lambda_{\max}(S_\pi)\ge e_1^\top S_\pi e_1
=\int (y_1-x_1)^2\,d\pi(x,y)
\ge
\left(\int (y_1-x_1)\,d\pi(x,y)\right)^2
=L^2.
$$
Hence every admissible coupling has operator cost at least $L^2$.

The translation map
$$
T(x_1,x_2)=(x_1+L,x_2)
$$
pushes $\mu_0$ to $\mu_1$ and has covariance
$$
S_T=\diag(L^2,0),
$$
so $\mathsf W_{\gamma}(\mu_0,\mu_1)^2=L^2$.

Now consider the reflected transport
$$
R(x_1,x_2)=(x_1+L,-x_2).
$$
It also pushes $\mu_0$ to $\mu_1$ because the second marginal of $\mu_0$ is symmetric on $[-a,a]$. Its displacement is
$$
R(x)-x=(L,-2x_2),
$$
hence
$$
S_R=\diag\!\left(L^2,\frac{4a^2}{3}\right).
$$
By the choice of $a$,
$$
\lambda_{\max}(S_R)=L^2,
$$
so the coupling induced by $R$ is again optimal for $\mathsf W_{\gamma}$.

Consider the corresponding constant-speed geodesic
$$
\mu_t=((1-t)\Id+tR)_\#\mu_0.
$$
At time $t=\frac12$ one gets
$$
\mu_{1/2}
=
\left(x_1+\frac{L}{2},0\right)_\#\mu_0,
$$
which is the uniform measure on the horizontal segment
$$
\left[\frac{L}{2},\frac{L}{2}+1\right]\times\{0\}.
$$
Therefore $\mu_{1/2}$ is singular with respect to Lebesgue measure, hence also singular with respect to $\nu$, and thus
$$
\operatorname{Ent}_\nu(\mu_{1/2})=+\infty.
$$
So along this optimal $\mathsf W_{\gamma}$-geodesic the endpoints have finite entropy while the midpoint has infinite entropy, which rules out geodesic convexity.
\end{proof}

The next remark explains why Proposition~\ref{prop:entropy-counterexample} does not contradict Theorem~\ref{thm:entropyweak}.
\begin{remark}[Why there is no contradiction]
Theorem~\ref{thm:entropyweak} is an existence statement: for each pair of
endpoints, it guarantees at least one $\Wg$-geodesic along which the entropy
inequality holds. Proposition~\ref{prop:entropy-counterexample} proves that in
the operator geometry there also exist optimal geodesics along which entropy
blows up at an interior time. Thus the counterexample rules out convexity
along every geodesic, but not the existence of a selected geodesic satisfying
the estimate.
\end{remark}

The following remark explains the scope of the entropy results for the Schatten family.
\begin{remark}[Schatten norms]
The full theorem applies to the trace case $p=1$, where one may choose $\cK=\{\Id\}$ and recover the classical $\Wtwo$ displacement-convexity theory. For every Schatten geometry with $1<p\le \infty$, the natural representing sets contain singular matrices, so the uniformly elliptic argument above does not apply directly. Theorem~\ref{thm:entropyweak} nevertheless yields a weak geodesic-convexity statement as soon as $\cK$ contains one positive definite matrix, which is the case for the standard Schatten choices.
\end{remark}

%%%

\paragraph{Metric slope of relative entropy.}
Theorem~\ref{prop:formalGF} does not directly apply to entropy: atomic
measures have infinite entropy and are dense in $\cP_2$, so its
finite-neighborhood hypothesis fails. The following one-sided argument
provides the slope identity needed for Gaussian states and smooth local
perturbations, also for root-norm gauges.

\begin{proposition}[Entropy slope at a regular relative density]
\label{prop:entropy-slope}
Let $\gamma$ be a root-norm gauge and let $\alpha,\beta\in\cP_2(\R^d)$ have
strictly positive $C^1$ densities. Assume that
$\ell=\log(d\alpha/d\beta)\in C^2$ has bounded Hessian. Then
\[
 |\partial\KL(\cdot\mid\beta)|_\gamma(\alpha)
 =\MN_\alpha^*(\nabla\ell).
\]
\end{proposition}

\begin{proof}
The Hessian bound implies quadratic growth of $\ell$ and linear growth of
$g=\nabla\ell$, so the entropy is finite and $g\in L^2(\alpha)$.
The supporting inequality for relative entropy gives
\[
 \KL(\nu\mid\beta)-\KL(\alpha\mid\beta)
 \ge\int\ell\,d(\nu-\alpha).
\]
It follows either by convexity or from the nonnegativity of
$\KL(\nu\mid\alpha)$; it is trivial when the left side is infinite.
For any coupling $\pi$ of $\alpha,\nu$, Taylor's lower bound therefore gives
\[
 \KL(\nu\mid\beta)-\KL(\alpha\mid\beta)
 \ge\int g(x)\cdot(y-x)\,d\pi
       -\frac M2\int|y-x|^2\,d\pi,
 \qquad M=\sup_x\|D^2\ell(x)\|_{\rm op}.
\]
Choose $c_\gamma>0$ with $c_\gamma\tr(S)\le\gamma(S)$, as provided by
Corollary~\ref{cor:root-norm-metric}. Apply the inequality to an optimal
spectral coupling. Conditional Jensen and monotonicity of a root-norm gauge give
\[
 \KL(\alpha\mid\beta)-\KL(\nu\mid\beta)
 \le\MN_\alpha^*(g)\Wg(\alpha,\nu)
       +\frac{M}{2c_\gamma}\Wg(\alpha,\nu)^2.
\]
This proves the slope upper bound.

For $w\in C_c^\infty(\R^d;\R^d)$, the map $\Id+hw$ is a diffeomorphism
for small $|h|$. Change of variables in the entropy and integration by parts
on the compact support give
\[
 \left.\frac{d}{dh}\KL((\Id+hw)_\sharp\alpha\mid\beta)\right|_{h=0}
 =-\int(\divv w+\nabla\log\beta\cdot w)\,d\alpha
 =\int g\cdot w\,d\alpha,
\]
where $\beta$ inside the logarithm denotes its density. Together with
$\Wg(\alpha,(\Id+hw)_\sharp\alpha)\le |h|\MN_\alpha(w)$, this proves the
opposite slope bound by duality and density of $C_c^\infty$ in $L^2(\alpha)$.
\end{proof}

\subsection{Relative-entropy flows over Gaussians}
\label{sec:relative-entropy-gaussians}

The $\Wtwo$-gradient flow of the relative entropy
$\operatorname{Ent}_\nu(\mu) = \KL(\mu\,|\,\nu)$ is the Langevin/Fokker--Planck equation, whose stochastic
single-particle realization is the overdamped Langevin diffusion. This makes
it a basic tool for sampling. We now describe what happens for the spectral
Schatten geometries. For $p>1$, the resulting evolution is still explicit on
Gaussian measures, but the PDE is generally nonlocal: its velocity depends
on global statistics of the current law. In general, a stochastic
representation would therefore need law-dependent coefficients, rather than
a fixed autonomous single-particle diffusion. In dimension one all these
geometries reduce to $\Wtwo$.

\begin{proposition}[Gaussian KL flow for Schatten geometries]
\label{prop:gaussian-kl-schatten-flow}
Let $\nu=\mathcal N(\bar m,\bar\Sigma)$ with $\bar\Sigma\succ0$, and consider
$\mathcal F(\mu)=\KL(\mu\,|\,\nu)$ with the canonical Schatten selector of
Corollary~\ref{cor:schatten-measure-selectors}. Given
$\mu_0=\mathcal N(m_0,\Sigma_0)$ with $\Sigma_0\succ0$, any locally absolutely
continuous solution $(m_t,\Sigma_t)$ of the equations below, with these initial
data and $\Sigma_t\succ0$, defines a Gaussian solution
$\mu_t=\mathcal N(m_t,\Sigma_t)$ of the normalized continuity equation.
On every compact time interval on which $\Sigma_t\succ0$, this solution is
also a curve of maximal slope for $\KL(\cdot\mid\nu)$ in the ambient
$\Wg$ metric.
Conversely, every sufficiently regular Gaussian solution with this selector
satisfies these equations. This is a conditional reduction, not an assertion
of existence or uniqueness through changes in the rank of the force covariance.
Define the classical $p=1$ Gaussian velocity by
$\dot m_t^{(1)}\coloneqq-\bar\Sigma^{-1}(m_t-\bar m)$ and
$\dot\Sigma_t^{(1)}\coloneqq
2I-\bar\Sigma^{-1}\Sigma_t-\Sigma_t\bar\Sigma^{-1}$.
Then the trace geometry $p=1$ satisfies
$\dot m_t=\dot m_t^{(1)}$ and $\dot\Sigma_t=\dot\Sigma_t^{(1)}$.

For a general Schatten geometry $p\in[1,\infty]$, set
$q\coloneqq 2p/(2p-1)$, with the convention $q=1$ when $p=\infty$, and define
\[
    F_t
    \coloneqq
    (\bar\Sigma^{-1}-\Sigma_t^{-1})\Sigma_t
    (\bar\Sigma^{-1}-\Sigma_t^{-1})
    +
    \bar\Sigma^{-1}(m_t-\bar m)(m_t-\bar m)^\top\bar\Sigma^{-1}.
\]
If $F_t\ne0$, define
\[
    \mathcal N_t^q\coloneqq \tr(F_t^{q/2}),
    \qquad
    P_{p,t}\coloneqq \mathcal N_t^{2-q}F_t^{q/2-1}.
\]
For $1<p\le\infty$, the negative power is taken on
$\operatorname{range}(F_t)$ and set to zero on its kernel; for $p=1$, take
$P_{1,t}=I$. If $F_t=0$, the velocity and both right-hand sides below are set
to zero. Then the Schatten-$p$ Gaussian KL flow is
\[
    \dot m_t=P_{p,t}\dot m_t^{(1)},
    \qquad
    \dot\Sigma_t
    =
    P_{p,t}(I-\bar\Sigma^{-1}\Sigma_t)
    +(I-\Sigma_t\bar\Sigma^{-1})P_{p,t}.
\]
If $\bar\Sigma^{-1}$ and $\Sigma_t$ commute, the covariance equation reduces
to
\[
 \dot\Sigma_t
 =\frac12\left(
   P_{p,t}\dot\Sigma_t^{(1)}
   +\dot\Sigma_t^{(1)}P_{p,t}
 \right).
\]
In particular, for $p=1$, one has $q=2$, $P_{1,t}=I$, and the classical
Langevin covariance equation is recovered. For $p=\infty$, one has
\[
    P_{\infty,t}=\tr(F_t^{1/2})F_t^{\dagger/2}.
\]
\end{proposition}

\begin{proof}
For $\mu_t=\mathcal N(m_t,\Sigma_t)$ and
$\nu=\mathcal N(\bar m,\bar\Sigma)$, the first variation of
$\KL(\mu\,|\,\nu)$ has Wasserstein gradient
\[
    g_t(x)
    =
    \nabla_x\log\frac{d\mu_t}{d\nu}(x)
    =
    (\bar\Sigma^{-1}-\Sigma_t^{-1})(x-m_t)
    +
    \bar\Sigma^{-1}(m_t-\bar m).
\]
Thus $g_t$ is affine. Set
$H_t\coloneqq\bar\Sigma^{-1}-\Sigma_t^{-1}$ and
$b_t\coloneqq\bar\Sigma^{-1}(m_t-\bar m)$, so that
$g_t(x)=H_t(x-m_t)+b_t$. Since
$\int(x-m_t)\,d\mu_t(x)=0$ and
$\int(x-m_t)(x-m_t)^\top\,d\mu_t(x)=\Sigma_t$, the covariance matrix of
$g_t$ is exactly
\[
    \int g_t(x)g_t(x)^\top\,d\mu_t(x)
    =
    H_t\Sigma_tH_t+b_tb_t^\top
    =
    F_t.
\]

For the Schatten geometry associated with $p$, the Measure-LMO gives the
velocity
\[
    v_t(x)
    =
    -P_{p,t}g_t(x),
    \qquad
    P_{p,t}=\mathcal N_t^{2-q}F_t^{q/2-1},
    \qquad
    \mathcal N_t^q=\tr(F_t^{q/2}).
\]
This formula is the Schatten duality map applied to the vector field $g_t$.
When $F_t=0$, one has $g_t=0$ $\mu_t$-a.e., and the flow is stationary.
Because $g_t$ is affine and $P_{p,t}$ is independent of $x$, the selected
velocity $v_t$ is affine in $x$. The continuity equation for this prescribed
affine velocity transports Gaussian laws to Gaussian laws. The moment
equations computed below therefore verify the claimed Gaussian solution;
they do not establish uniqueness among all solutions of the nonlinear PDE.

We now compute the induced ODE on mean and covariance. Since
$v_t(x)=-P_{p,t}H_t(x-m_t)-P_{p,t}b_t$, the mean satisfies
\[
    \dot m_t
    =
    \int v_t(x)\,d\mu_t(x)
    =
    -P_{p,t}b_t.
\]
But
$b_t=\bar\Sigma^{-1}(m_t-\bar m)$, hence
$-b_t=\dot m_t^{(1)}$. Therefore
$\dot m_t=P_{p,t}\dot m_t^{(1)}$.

For the covariance, using
$\dot\Sigma_t=\int \bigl(v_t(x)(x-m_t)^\top
+(x-m_t)v_t(x)^\top\bigr)\,d\mu_t(x)$, we obtain
\[
    \dot\Sigma_t
    =
    -P_{p,t}H_t\Sigma_t-\Sigma_tH_tP_{p,t}.
\]
On the other hand,
\[
    \dot\Sigma_t^{(1)}
    =
    2I-\bar\Sigma^{-1}\Sigma_t-\Sigma_t\bar\Sigma^{-1}.
\]
Since $H_t=\bar\Sigma^{-1}-\Sigma_t^{-1}$ and both
$\bar\Sigma^{-1}$ and $\Sigma_t^{-1}$ are symmetric,
\[
    -H_t\Sigma_t
    =
    I-\bar\Sigma^{-1}\Sigma_t,
    \qquad
    -\Sigma_tH_t
    =
    I-\Sigma_t\bar\Sigma^{-1}.
\]
Therefore
\[
    -P_{p,t}H_t\Sigma_t-\Sigma_tH_tP_{p,t}
    =
    P_{p,t}(I-\bar\Sigma^{-1}\Sigma_t)
    +
    (I-\Sigma_t\bar\Sigma^{-1})P_{p,t}.
\]
This is the intrinsic covariance equation stated above. Whenever
$\bar\Sigma^{-1}\Sigma_t=\Sigma_t\bar\Sigma^{-1}$, this reduces to the
displayed symmetrized formula because
$\dot\Sigma_t^{(1)}=2(I-\bar\Sigma^{-1}\Sigma_t)$.

Finally, for $p=1$ one has $q=2$, hence
$\mathcal N_t^{2-q}=1$ and $F_t^{q/2-1}=I$, so $P_{1,t}=I$. The equations
therefore reduce to
$\dot m_t=-\bar\Sigma^{-1}(m_t-\bar m)$ and
$\dot\Sigma_t=2I-\bar\Sigma^{-1}\Sigma_t-\Sigma_t\bar\Sigma^{-1}$, which is
the Gaussian restriction of the classical Langevin/Fokker--Planck flow. For
$p=\infty$, $q=1$, giving
$P_{\infty,t}=\tr(F_t^{1/2})F_t^{\dagger/2}$.

To verify the metric interpretation, work on a compact time interval with
$\Sigma_t\succ0$. The moments and $\Sigma_t^{-1}$ are bounded there, so the
force norm and the selected velocity norm are bounded. For an affine field
$v_t(x)=A_t(x-m_t)+a_t$, its $L^2(\mu_t)$ norm controls $A_t,a_t$ on such an
interval. Thus the affine coefficients are locally integrable, even if they
jump at a force-rank change, and the Gaussian continuity equation is valid.
The Gaussian entropy is a smooth function of $(m,\Sigma)$; its derivative is
$\int g_t\cdot v_t\,d\mu_t=-\MN_{\mu_t}^*(g_t)^2$ almost everywhere.
Proposition~\ref{prop:entropy-slope} identifies the ambient slope because
the log density ratio is quadratic. BB bounds the metric speed by
$\MN_{\mu_t}(v_t)=\MN_{\mu_t}^*(g_t)$. The chain rule and slope inequality
give the reverse speed bound, exactly as in the proof of
Theorem~\ref{prop:formalGF}. The energy is absolutely continuous, hence the
integrated energy identity follows. No existence or uniqueness claim for
the moment ODE is needed for this verification.
\end{proof}

%%%

\section{P-L and Poincar\'e Inequalities}
\label{app:pl-poincare}

The preceding section shows that geodesic convexity is useful but sometimes
too rigid for Spectral Wasserstein geometries.  A more direct route to
convergence rates is a Polyak--\L ojasiewicz (P-L) inequality: a first-order
estimate requiring the metric slope to control the energy gap.

\subsection{P-L inequalities}

Let $\MF_\star\coloneqq\inf_{\mu\in\cP_2(\R^d)}\MF(\mu)>-\infty$.
We use the descending metric slope defined in Section~\ref{sec:wasserstein-flows}.
Whenever its identification with the dual force norm holds, as under
Theorem~\ref{prop:formalGF}, it can be computed from
$g_\mu=\nabla\delta\MF/\delta\mu$ by
\[
    |\partial\MF|_\gamma(\mu)^2
    =
    \sup_{\MN_\mu(v)\le 1}
    \left(\int g_\mu(x)\cdot v(x)\,d\mu(x)\right)^2 .
\]

\begin{definition}[$\Wg$--P-L inequality]
\label{def:gamma-pl}
We say that $\MF$ satisfies the $\Wg$--P-L inequality with constant
$\kappa>0$ if
\[
    |\partial\MF|_\gamma(\mu)^2
    \ge
    2\kappa\bigl(\MF(\mu)-\MF_\star\bigr)
\]
for every $\mu$ in the domain of $\MF$.
\end{definition}

\begin{proposition}[$\Wg$--P-L convergence]
\label{prop:gamma-pl-convergence}
Assume that $\MF$ satisfies Definition~\ref{def:gamma-pl} and is lower
semicontinuous for $\Wg$ convergence. Let $(\mu_t)_{t\ge0}$ be a global,
locally absolutely continuous $\Wg$ gradient flow with $\MF(\mu_0)<+\infty$.
Assume that $t\mapsto\MF(\mu_t)$ is locally absolutely continuous and satisfies
the energy dissipation identity
\[
    \frac{d}{dt}\MF(\mu_t)
    =
    -|\dot\mu_t|_\gamma^2
    =
    -|\partial\MF|_\gamma(\mu_t)^2
    \qquad\text{for a.e. }t>0 .
\]
Set $E(t)=\MF(\mu_t)-\MF_\star$. Then, for $0\le s\le t$,
\[
    E(t)\le e^{-2\kappa(t-s)}E(s),
\]
and
\[
    \operatorname{Length}_{\Wg}\bigl((\mu_r)_{r\in[s,t]}\bigr)
    \le
    \sqrt{\frac{2}{\kappa}}\,
    \bigl(\sqrt{E(s)}-\sqrt{E(t)}\bigr).
\]
Consequently, $\mu_t$ converges to some
$\mu_\infty\in\operatorname*{argmin}\MF$, and
\[
    \Wg(\mu_t,\mu_\infty)
    \le
    \sqrt{\frac{2}{\kappa}}\sqrt{E(t)}
    \le
    \sqrt{\frac{2E(0)}{\kappa}}\,e^{-\kappa t}.
\]
In particular,
\[
    \operatorname{dist}_{\Wg}\bigl(\mu_t,\operatorname*{argmin}\MF\bigr)
    \le
    \sqrt{\frac{2E(0)}{\kappa}}\,e^{-\kappa t}.
\]
\end{proposition}

\begin{proof}
The energy identity and the P-L inequality give
$E'(t)\le -2\kappa E(t)$, hence the exponential bound by Gronwall's lemma.  On
the set where $E(t)>0$, the same two identities give
\[
    |\dot\mu_t|_\gamma
    =
    |\partial\MF|_\gamma(\mu_t)
    \le
    \frac{|\partial\MF|_\gamma(\mu_t)^2}
         {\sqrt{2\kappa E(t)}}
    =
    \frac{-E'(t)}{\sqrt{2\kappa E(t)}} .
\]
Integrating from $s$ to $t$ gives the length estimate; intervals on which
$E=0$ are stationary by the energy identity.  The tail length is finite, so
$(\mu_t)$ is Cauchy.  Proposition~\ref{prop:bounds} and completeness of
$\Wtwo$ imply completeness of $\Wg$, hence $\mu_t\to\mu_\infty$.  Lower
semicontinuity gives
$\MF(\mu_\infty)=\MF_\star$, and the last bound follows by sending the final
time to $+\infty$ and using that distance is bounded by metric length.
\end{proof}

The following comparison explains why P-L estimates transfer more readily
than geodesic convexity: only the tangent norms and their dual slopes enter.

\begin{proposition}[Transfer from $\Wtwo$ to $\Wg$]
\label{prop:pl-transfer}
Assume that
\[
    \gamma(S)\le C_\gamma\tr(S)
    \qquad
    \text{for every }S\succeq0 .
\]
If $\MF$ satisfies the $\Wtwo$--P-L inequality with constant $\kappa$, then it
satisfies the $\Wg$--P-L inequality with constant $\kappa/C_\gamma$.
\end{proposition}

\begin{proof}
The static comparison gives $\Wg\le\sqrt{C_\gamma}\Wtwo$.
Both distances induce the same topology, so the definitions of descending
slope yield $|\partial\MF|_\gamma\ge
C_\gamma^{-1/2}|\partial\MF|_2$. Combining this with the
$\Wtwo$--P-L inequality proves the claim, without a differentiability
assumption on $\MF$.
\end{proof}

The normalization of the gauge matters.  Replacing $\gamma$ by $a\gamma$,
$a>0$, multiplies squared distances and squared tangent norms by $a$, while
dividing squared slopes, P-L constants, and the Poincar\'e constants below by
$a$.
All comparisons in this section use the standard Schatten normalization
$\gamma_p(aa^\top)=|a|^2$.

For the Schatten gauges $\gamma_p(S)=\norm{S}_{S_p}$ with $p\ge1$, including
the Muon endpoint $p=\infty$, one has
$\gamma_p(S)\le\tr(S)$ on $\mathbb S_+^d$, so $C_\gamma=1$.  Thus every
$\Wtwo$--P-L inequality transfers to these spectral geometries with no loss of
constant.  A basic instance is the relative entropy
$\MF(\alpha)=\KL(\alpha\,|\,\beta)$: if
$d\beta=Z^{-1}e^{-V}dx$ and $\nabla^2V\succeq\lambda I$ for some $\lambda>0$,
then the classical
log-Sobolev inequality gives the $\lambda$--P-L inequality in $\Wtwo$
\citep{BakryGentilLedoux2014,Villani2009}. Hence the same functional is also $\lambda$--P-L for
the Muon geometry. Thus any global Muon entropy flow satisfying the energy
identity above inherits at least the same guaranteed exponential rate in
function value as the $\Wtwo$ flow. This is a comparison
of certified rates, not a pointwise ordering of the two trajectories.

\subsection{Spectral Poincar\'e inequalities}

Throughout this subsection, assume that $\beta\in\cP_2(\R^d)$ has a strictly
positive $C^1$ density. Test functions are taken in the weighted Sobolev space
$H^1(\beta)$, the completion of smooth compactly supported functions for the
norm $\left(\int (u^2+|\nabla u|^2)\,d\beta\right)^{1/2}$.
These assumptions make the weak gradient
well defined and the smooth and Sobolev infima below agree.

The comparison above gives a simple way to transfer a $\Wtwo$--P-L estimate,
but it does not identify the sharp $\Wg$--P-L constant in general.  Deciding
whether a geometry such as Muon yields a strictly better global constant is
difficult because its nonlinear slope must be controlled over the entire
energy landscape. We therefore study a small-amplitude expansion at a
minimizer, measuring a variational local dissipation gap in the spectral
geometry. The leading evolution need not be linear: a spectral duality map
is generally nonlinear even though it is one-homogeneous.
For concreteness, consider the relative entropy
\[
    \MF(\alpha)=\KL(\alpha\,|\,\beta).
\]
Its $\Wtwo$ gradient flow is the Fokker--Planck, or drift--diffusion, process
with invariant law $\beta$. Its variational local gap in the $\Wg$ geometry is a
$\gamma$-Poincar\'e constant, generalizing the classical Poincar\'e constant of
$\beta$.

Let $u=\varphi-\int\varphi\,d\beta$ with
$\varphi\in C_c^\infty(\R^d)$ and set
$\alpha_\varepsilon=(1+\varepsilon u)\beta$ for sufficiently small
$|\varepsilon|$.  Then
\[
    \KL(\alpha_\varepsilon\,|\,\beta)
    =
    \frac{\varepsilon^2}{2}\operatorname{Var}_\beta(u)
    +O(\varepsilon^3),
    \qquad
    \nabla\log\frac{d\alpha_\varepsilon}{d\beta}
    =
    \frac{\varepsilon\nabla u}{1+\varepsilon u}
    =\varepsilon\nabla u+O_{L^2(\beta)}(\varepsilon^2).
\]
Proposition~\ref{prop:entropy-slope} applies because
$\log(1+\varepsilon u)$ has bounded Hessian. These expansions therefore give
\[
    |\partial\MF|_\gamma(\alpha_\varepsilon)^2
    =\MN_{\alpha_\varepsilon}^*\!\left(
      \nabla\log\frac{d\alpha_\varepsilon}{d\beta}\right)^2
    =
    \varepsilon^2\mathcal E_\gamma^\beta(u)+o(\varepsilon^2),
\]
where the spectral Dirichlet energy is
\[
    \mathcal E_\gamma^\beta(u)
    \coloneqq
    \sup_{\MN_\beta(v)\le1}
    \left(\int \nabla u\cdot v\,d\beta\right)^2 .
\]
Indeed, the pairing is exactly
$\int\nabla\log(d\alpha_\varepsilon/d\beta)\cdot v\,d\alpha_\varepsilon
=\varepsilon\int\nabla u\cdot v\,d\beta$.
Writing $\delta=|\varepsilon|\|u\|_\infty<1$, monotonicity and homogeneity give
$(1-\delta)\MN_\beta(v)^2\le\MN_{\alpha_\varepsilon}(v)^2
\le(1+\delta)\MN_\beta(v)^2$. Duality then bounds the squared slope between
$\varepsilon^2\mathcal E_\gamma^\beta(u)/(1+\delta)$ and
$\varepsilon^2\mathcal E_\gamma^\beta(u)/(1-\delta)$, justifying the expansion
despite the changing reference measure.
Thus a $\kappa$--P-L inequality implies
$\kappa\operatorname{Var}_\beta(u)\le\mathcal E_\gamma^\beta(u)$, first for
these tests and then on $H^1(\beta)$ by density and norm comparison.
The converse is not a global P-L theorem, nor a uniform local P-L statement
without specifying a neighborhood and uniform remainder bounds.
The energy $\mathcal E_\gamma^\beta$ is convex and two-homogeneous but is
generally not quadratic. We therefore use ``Dirichlet energy'' rather than
a bilinear Dirichlet form. For the canonical Schatten selection ($p\ge1$),
the formal leading perturbation equation is
\[
 \partial_tu=-\beta^{-1}\divv\!\left(
       \beta\MJ_\gamma^\beta(\nabla u)\right),
 \qquad
 \frac12\frac{d}{dt}\int u^2\,d\beta=-\mathcal E_\gamma^\beta(u),
\]
where $\beta$ in the PDE denotes its density and the identity holds for
sufficiently regular solutions. This explains the dissipation meaning of
the local constant without asserting a linear generator or existence of
the nonlinear entropy flow.

\begin{definition}[$\gamma$-Poincar\'e constant]
The $\gamma$-Poincar\'e constant of $\beta$ is
\[
    \lambda_{\mathrm P}^\gamma(\beta)
    \coloneqq
    \inf_{\operatorname{Var}_\beta(u)>0}
    \frac{\mathcal E_\gamma^\beta(u)}
         {\operatorname{Var}_\beta(u)} .
\]
Equivalently,
\[
    \lambda_{\mathrm P}^\gamma(\beta)\operatorname{Var}_\beta(u)
    \le
    \mathcal E_\gamma^\beta(u)
\]
for all smooth $u$.
\end{definition}

Equivalently, one may take the infimum over the corresponding Sobolev form
domain; the finite-dimensional norm comparisons make the two formulations
identical by density.

\begin{proposition}[Robust Poincar\'e representation]
\label{prop:robust-poincare-representation}
Let $\cK\subset\mathbb S_+^d$ be a compact convex representing set for
$\gamma$, and assume that $\cK\cap\Sppd\ne\emptyset$.  Then, for every smooth
$u$,
\[
    \mathcal E_\gamma^\beta(u)
    =
    \inf_{Q\in\cK\cap\Sppd}
    \int \nabla u(x)^\top Q^{-1}\nabla u(x)\,d\beta(x).
\]
The infimum need not be attained; a singular optimizer is understood through
the lower-semicontinuous limit from $\cK\cap\Sppd$.  Consequently,
\[
    \lambda_{\mathrm P}^\gamma(\beta)
    =
    \inf_{Q\in\cK\cap\Sppd}
    \lambda_{\mathrm P}^Q(\beta),
    \qquad
    \lambda_{\mathrm P}^Q(\beta)
    \coloneqq
    \inf_{\operatorname{Var}_\beta(u)>0}
    \frac{\int \nabla u^\top Q^{-1}\nabla u\,d\beta}
         {\operatorname{Var}_\beta(u)} .
\]
\end{proposition}

\begin{proof}
Let $\ell_u(v)=\int\nabla u\cdot v\,d\beta$ and
$\MN_\beta(v)^2=\gamma(\int vv^\top d\beta)$.  Fenchel duality for a squared
norm gives
\[
    \inf_v\left\{\ell_u(v)+\frac12\MN_\beta(v)^2\right\}
    =
    -\frac12
    \sup_{\MN_\beta(v)\le1}\ell_u(v)^2
    =
    -\frac12\mathcal E_\gamma^\beta(u).
\]
On the other hand, set
$\Phi_Q(v)=\ell_u(v)+\frac12\int v^\top Qv\,d\beta$.  The support
representation of $\gamma$ and the inverse-trace minimax argument of
Proposition~\ref{prop:dual-inverse-trace-selector} give
\[
\begin{aligned}
    \inf_v\left\{\ell_u(v)+\frac12\MN_\beta(v)^2\right\}
    &=
    \inf_v\max_{Q\in\cK}\Phi_Q(v)
    =
    \max_{Q\in\cK}\inf_v\Phi_Q(v) \\
    &=
    -\frac12\inf_{Q\in\cK\cap\Sppd}
    \tr\!\left(
        Q^{-1}\int\nabla u\nabla u^\top\,d\beta
    \right).
\end{aligned}
\]
Here any $Q_0\in\cK\cap\Sppd$ provides the coercivity needed for the minimax
exchange, and convexity of $\cK$ lets one approximate every singular $Q$ by
$(1-\varepsilon)Q+\varepsilon Q_0\succ0$.  Comparing the two displays proves
the formula for $\mathcal E_\gamma^\beta(u)$.  Taking the
infimum over all $u$ with positive variance gives the formula for
$\lambda_{\mathrm P}^\gamma(\beta)$; the two infima commute because this is an
infimum over the product of admissible functions and matrices.
\end{proof}

Equivalently, $\lambda_{\mathrm P}^Q(\beta)$ is the usual Poincar\'e constant
after the linear change of variables $(Q^{1/2})_\#\beta$.  Hence the spectral
Poincar\'e constant is the worst anisotropic Poincar\'e constant over the dual
ball $\cK$.  At singular endpoints, this worst case may be approached by
positive definite matrices without being attained in $\Sppd$.

For the Schatten family $\gamma_p$, $p\ge1$, this becomes especially explicit.
With $p'=p/(p-1)$, using the conventions $p'=\infty$ for $p=1$ and $p'=1$ for
$p=\infty$, one may take
\[
    \mathcal K_p=\{Q\succeq0:\ \norm{Q}_{S_{p'}}\le1\}.
\]
Hence
\[
    \lambda_{\mathrm P}^{p}(\beta)
    =
    \inf_{\substack{Q\succ0\\ \norm{Q}_{S_{p'}}\le1}}
    \lambda_{\mathrm P}\bigl((Q^{1/2})_\#\beta\bigr).
\]
For $p=\infty$, this is the Muon formula
\[
    \lambda_{\mathrm P}^{\infty}(\beta)
    =
    \inf_{\substack{Q\succ0\\ \tr Q\le1}}
    \lambda_{\mathrm P}\bigl((Q^{1/2})_\#\beta\bigr).
\]
Since every admissible $Q$ satisfies $Q\preceq I$, one has
$Q^{-1}\succeq I$ and therefore
\[
    \lambda_{\mathrm P}^{p}(\beta)
    \ge
    \lambda_{\mathrm P}^{1}(\beta)
    \qquad (p\ge1).
\]
Thus the variational local dissipation constant is no smaller than the
classical $\Wtwo$ constant for all Schatten geometries from the trace case
to the Muon endpoint. Along solutions for which the preceding perturbation
identity is justified, it controls exponential decay of the perturbation's
squared $L^2(\beta)$ norm.

The same local computation also gives a simple way to compare examples across
the full root-norm Schatten range $p\ge1/2$ discussed in
Appendix~\ref{app:root-norm-gauges}.  For a smooth test function $u$, set
\[
    G_u\coloneqq \int \nabla u\,\nabla u^\top\,d\beta .
\]
For $\gamma_p(S)=\norm{S}_{S_p}$, $p\ge1/2$, the tangent norm is the
Schatten $S_{2p}$ norm of the operator
$T_v:a\mapsto a\cdot v$. By Schatten duality, the spectral Dirichlet energy is
\[
    \mathcal E_p^\beta(u)
    =
    \norm{G_u^{1/2}}_{S_{\theta_p}}^2,
    \qquad
    \theta_p=(2p)',
\]
with the conventions $\theta_{1/2}=+\infty$ and
$\theta_\infty=1$.  Relative to the trace case,
\[
    \mathcal E_p^\beta(u)\ge\mathcal E_1^\beta(u)
    \quad (p>1),
    \qquad
    \mathcal E_p^\beta(u)\le\mathcal E_1^\beta(u)
    \quad (1/2\le p<1),
\]
and, away from $p=1$, equality holds precisely when $G_u$ has rank at most
one.  Thus the difference between the geometries is visible exactly through
perturbations with several active gradient directions.

\begin{proposition}[Gaussian targets]
\label{prop:gaussian-poincare-schatten}
Let $\beta=\mathcal N(m,\Sigma)$ be a nondegenerate Gaussian measure on
$\R^d$, $d\ge2$.  Then, for every $p\ge1$,
\[
    \lambda_{\mathrm P}^{p}(\beta)
    =
    \lambda_{\mathrm P}^{1}(\beta)
    =
    \frac{1}{\lambda_{\max}(\Sigma)} .
\]
For $1/2\le p<1$, the constant is strictly smaller:
\[
    d^{1-1/p}\lambda_{\mathrm P}^{1}(\beta)
    \le
    \lambda_{\mathrm P}^{p}(\beta)
    <
    \lambda_{\mathrm P}^{1}(\beta).
\]
\end{proposition}

\begin{proof}
For $p\ge1$, $\theta_p\le2$, so
$\norm{G^{1/2}}_{S_{\theta_p}}^2\ge\tr G$.  The classical Gaussian Poincar\'e
inequality gives
$\lambda_{\mathrm P}^{p}(\beta)\ge1/\lambda_{\max}(\Sigma)$.  Equality follows
by testing with a linear function $u(x)=a^\top x$, where $a$ is an eigenvector
of $\Sigma$ associated with $\lambda_{\max}(\Sigma)$: then $G_u=aa^\top$ is
rank one, so all Schatten norms coincide on $G_u^{1/2}$.

Let now $1/2\le p<1$, so $\theta_p>2$.  The Schatten comparison gives
\[
    \norm{G^{1/2}}_{S_{\theta_p}}^2
    \ge
    d^{1-1/p}\tr G,
\]
which proves the lower bound.  For the strict upper bound, translate the
measure so that $m=0$, diagonalize $\Sigma$, and write
$\sigma_1=\lambda_{\max}(\Sigma)$ and $\sigma_2>0$ for another eigenvalue.  In
these coordinates, take
\[
    u_\varepsilon(x)=x_1+\varepsilon(x_2^2-\sigma_2).
\]
Then
\[
    \operatorname{Var}_\beta(u_\varepsilon)
    =
    \sigma_1+2\varepsilon^2\sigma_2^2,
    \qquad
    G_{u_\varepsilon}
    =
    e_1e_1^\top+4\varepsilon^2\sigma_2\,e_2e_2^\top .
\]
Thus $\mathcal E_p^\beta(u_\varepsilon)=1+o(\varepsilon^2)$ as
$\varepsilon\to0$ for every $\theta_p>2$ (and it is exactly $1$ for small
$\varepsilon$ when $\theta_p=\infty$).  Therefore
\[
    \frac{\mathcal E_p^\beta(u_\varepsilon)}
         {\operatorname{Var}_\beta(u_\varepsilon)}
    =
    \frac{1}{\sigma_1}
    -\frac{2\sigma_2^2}{\sigma_1^2}\varepsilon^2
    +o(\varepsilon^2),
\]
which is strictly smaller than $1/\sigma_1$ for all sufficiently small
$\varepsilon>0$.
\end{proof}

\begin{proposition}[Generalized Gaussian targets]
\label{prop:power-poincare-schatten}
Let $d\ge2$, let $r>1$, $r\ne2$, and let
\[
    d\beta_r(x)=Z_{r,d}^{-1}e^{-|x|^r}\,dx .
\]
Then the Schatten Poincar\'e constant is strictly improved for $p>1$ and
strictly decreased for $1/2\le p<1$:
\[
    \lambda_{\mathrm P}^{p}(\beta_r)
    >
    \lambda_{\mathrm P}^{1}(\beta_r)
    \quad\text{for }p>1,
    \qquad
    \lambda_{\mathrm P}^{p}(\beta_r)
    <
    \lambda_{\mathrm P}^{1}(\beta_r)
    \quad\text{for }1/2\le p<1 .
\]
Thus, for the family $e^{-|x|^r}$ in dimension $d>1$, the Gaussian exponent
$r=2$ is the only case with no strict improvement of the local Poincar\'e rate
for $p>1$.
\end{proposition}

\begin{proof}
The reversible generator is
\[
    L_r=\Delta-r |x|^{r-2}x\cdot\nabla .
\]
Writing $V_r(x)=|x|^r$, the ground-state transform of $-L_r$ has effective
potential
\[
    \frac14|\nabla V_r|^2-\frac12\Delta V_r
    =
    \frac{r^2}{4}|x|^{2r-2}
    -\frac{r(r+d-2)}{2}|x|^{r-2},
\]
which tends to $+\infty$ as $|x|\to\infty$. The transform is understood at
the level of the closed quadratic form. For $1<r<2$, its negative singularity
at zero is a multiple of $|x|^{-(2-r)}$, which belongs locally to $L^a$ for
some $a>d/2$ and $a>1$. Sobolev estimates make this term infinitesimally
form-bounded relative to the Laplacian. For $r\ge2$ the potential is locally
bounded. In both cases, the confining part controls tails and local Rellich
compactness gives compact embedding of the form domain. Thus the classical
Poincar\'e problem has compact resolvent and its spectral gap is attained.

We first prove that no first eigenfunction has rank-one gradient covariance.
Indeed, if $G_u$ has rank one, then the positive density of $\beta_r$ implies
that $\nabla u$ is a.e. parallel to some unit vector $a$.  Hence
$u(x)=\psi(a\cdot x)$, and, with $t=a\cdot x$, the eigenvalue equation
$-L_r u=\lambda u$ becomes
\[
    -\psi''(t)+r |x|^{r-2}t\psi'(t)=\lambda\psi(t).
\]
Elliptic regularity away from the origin justifies this identity pointwise
for $t\ne0$. Since $u$ is nonconstant, one can choose such a $t$ with
$\psi'(t)\ne0$. Setting
$x=ta+z$ and varying $z\perp a$ forces $|ta+z|^{r-2}$ to be independent of
$|z|$, which is possible only for $r=2$.  This contradiction proves the
claim.

Let $\lambda_1=\lambda_{\mathrm P}^{1}(\beta_r)$. If $p>1$, then
$\theta_p<2$ and $\mathcal E_p^{\beta_r}\ge\mathcal E_1^{\beta_r}$, so
$\lambda_{\mathrm P}^{p}(\beta_r)\ge\lambda_1$. Suppose equality holds and
take a minimizing sequence $(u_n)$ with zero mean and unit variance.  Since
\[
    \lambda_1
    \le \mathcal E_1^{\beta_r}(u_n)
    \le \mathcal E_p^{\beta_r}(u_n)
    \longrightarrow \lambda_1,
\]
its classical Dirichlet energy also converges to $\lambda_1$.  Compactness of
the resolvent and the spectral decomposition then give, after extraction,
convergence in the form norm to a first classical eigenfunction $u_\star$.
Moreover,
$\mathcal E_p^{\beta_r}(u_n)-\mathcal E_1^{\beta_r}(u_n)\to0$, and the
continuity of $u\mapsto G_u$ along this convergence gives
\[
    \norm{G_{u_\star}^{1/2}}_{S_{\theta_p}}
    =
    \norm{G_{u_\star}^{1/2}}_{S_2}.
\]
For $\theta_p<2$, equality in the Schatten comparison is possible only when
$G_{u_\star}$ has rank one, impossible by the rigidity fact above.  Hence the
inequality is strict for $p>1$.

If $1/2\le p<1$, take any first classical eigenfunction $u$.  The previous
paragraph shows that $G_u$ has rank at least two.  Since now $\theta_p>2$,
\[
    \mathcal E_p^{\beta_r}(u)
    =
    \norm{G_u^{1/2}}_{S_{\theta_p}}^2
    <
    \norm{G_u^{1/2}}_{S_2}^2
    =
    \mathcal E_1^{\beta_r}(u)
    =
    \lambda_1\operatorname{Var}_{\beta_r}(u),
\]
which gives
$\lambda_{\mathrm P}^{p}(\beta_r)<\lambda_{\mathrm P}^{1}(\beta_r)$.
\end{proof}

The restriction $d\ge2$ is essential: in dimension one, every gradient
covariance is scalar, so all the Schatten Poincar\'e constants coincide.

%%%

\section{Positively \texorpdfstring{$2$}{2}-Homogeneous Case and a Generalized Unbalanced Geometry}
\label{app:homogeneous}
This section continues the two-layer mean-field discussion of
Section~\ref{sec:mlp-meanfield} in the positively two-homogeneous case.
The energy factors through a weighted spherical projection, and the action
of a homogeneous velocity has a spherical continuity--reaction form. These
are energy and action identities, not an isometry between fixed ambient
lifts and spherical measures.
Let
$$
\MF(\mu)=R\!\left(\int_{\R^d}\Phi(x)\,d\mu(x)\right),
\qquad
\Phi(\lambda x)=\lambda^2\Phi(x).
$$
Here $\Phi:\R^d\to\R^k$ is continuous and $R$ is defined on the relevant
moment vectors. The energy identity below also applies to nonsmooth features.
For the subsequent force construction we assume $\Phi\in C^1$ and $R\in C^1$;
applying it to ReLU requires a separate justification at its kinks.
Writing $x=r\omega$ with $\omega\in\Sd$, define the weighted spherical projection
$$
\int_{\Sd}\psi(\omega)\,d\Pi_2(\mu)(\omega)
\coloneqq 
\int_{\R^d}\abs{x}^2\psi\!\left(\frac{x}{\abs{x}}\right)\,d\mu(x).
$$
The integrand on the right is defined to be zero at $x=0$.
The first point is that two-homogeneous models only depend on this weighted spherical projection.
\begin{proposition}[Energy factorization]
\label{prop:quotient}
There exists a functional $\overline f$ on $\cM_+(\Sd)$ such that
$$
\MF(\mu)=\overline f(\Pi_2(\mu)).
$$
\end{proposition}

\begin{proof}
The homogeneity identity $\Phi(r\omega)=r^2\Phi(\omega)$ immediately gives
$$
\int_{\R^d}\Phi(x)\,d\mu(x)
=
\int_{\Sd}\Phi(\omega)\,d\Pi_2(\mu)(\omega),
$$
which defines $\overline f$.
\end{proof}

The second point is that the ambient normalized velocity splits into radial and tangential parts on the sphere. If a vector field $v$ is $1$-homogeneous, namely $v(\lambda x)=\lambda v(x)$ for every $\lambda>0$, then for every $x=r\omega\neq 0$ it can be written uniquely as
$$
v(r\omega)=r\bigl(b(\omega)\omega+\tau(\omega)\bigr),
\qquad
\tau(\omega)\in T_\omega \Sd.
$$
Indeed, one defines
$$
b(\omega)\coloneqq v(\omega)\cdot \omega,
\qquad
\tau(\omega)\coloneqq v(\omega)-\bigl(v(\omega)\cdot \omega\bigr)\omega,
$$
so that $\tau(\omega)\cdot \omega=0$, and then $1$-homogeneity gives
$$
v(r\omega)=r\,v(\omega)=r\bigl(b(\omega)\omega+\tau(\omega)\bigr).
$$
The splitting is orthogonal and unique. In the smooth two-homogeneous model,
the first variation is two-homogeneous, so its force is one-homogeneous.
Choose the projected Measure-LMO of
Proposition~\ref{prop:measure-lmo-reduction}, using a fixed measurable matrix
selection $A_g=A(S_g)$, and set the velocity to zero at the origin.
This selected velocity is one-homogeneous because it is a spatially constant
matrix times the force. The canonical Schatten selector is one such choice.
Other minimizing velocities need not be homogeneous when the oracle is
nonunique, so this selection requirement cannot be omitted.

The coefficients close on $\nu=\Pi_2(\mu)$: the force depends on $\mu$ only
through $\int\Phi\,d\nu$, and homogeneity gives
\[
 S_g=\int_{\Sd}g_\mu(\omega)g_\mu(\omega)^\top\,d\nu(\omega).
\]
Consequently the fixed selection, and hence $b$ and $\tau$, depend only on
$\nu$, not on its radial lift.
\begin{proposition}[Projected continuity-reaction equation]
\label{prop:sphere}
Let $(\mu_t,v_t)$ solve the continuity equation on a finite interval, with
$\mu_t$ continuous in $\Wtwo$, $\int_0^T\int|v_t|^2d\mu_tdt<\infty$, and
$v_t$ a one-homogeneous Borel field with $v_t(0)=0$. If
$\nu_t=\Pi_2(\mu_t)$, then
$$
\partial_t\nu_t+\divS(\nu_t\tau_t)=2b_t\nu_t.
$$
\end{proposition}

\begin{proof}
Test the ambient continuity equation against
$\widetilde\psi(r\omega)=r^2\psi(\omega)$. For $r>0$,
\[
 \nabla\widetilde\psi(r\omega)
 =r\bigl(2\psi(\omega)\omega+\nablaS\psi(\omega)\bigr).
\]
Using $v(r\omega)=r(b(\omega)\omega+\tau(\omega))$ and
$\omega\cdot\tau(\omega)=0$, we obtain
\[
 \nabla\widetilde\psi(r\omega)\cdot v(r\omega)
 =r^2\bigl(2b(\omega)\psi(\omega)
            +\nablaS\psi(\omega)\cdot\tau(\omega)\bigr).
\]
The weak form of the ambient continuity equation and the definition of
$\Pi_2$ therefore give
\[
 \frac{d}{dt}\int_{\Sd}\psi\,d\nu_t
 =\int_{\Sd}
   \bigl(2b_t\psi+\nablaS\psi\cdot\tau_t\bigr)\,d\nu_t,
\]
which is precisely the weak form of the stated continuity--reaction equation.
The test function is $C^1$ at the origin with zero gradient. Spatial cutoffs
justify its quadratic growth: the error is controlled by
$\int_0^T\int |x||v_t(x)|\,d\mu_tdt<\infty$, using the uniform second-moment
bound and finite kinetic action. The $\Wtwo$ continuity also implies narrow
continuity of $\nu_t$, by uniform integrability of the weighted projection.
\end{proof}

Motivated by Proposition~\ref{prop:sphere}, define for finite nonnegative
measures $\nu_0,\nu_1$ on $\Sd$
$$
\UWg(\nu_0,\nu_1)^2
\coloneqq 
\inf_{(\nu_t,b_t,\tau_t)}
\int_0^1
\gamma\!\left(
\int_{\Sd} (b_t(\omega)\omega+\tau_t(\omega))(b_t(\omega)\omega+\tau_t(\omega))^\top\,d\nu_t(\omega)
\right)\,dt,
$$
over narrowly continuous finite-measure curves with the prescribed endpoints
and Borel fields $b_t,\tau_t$, with $\tau_t(\omega)\perp\omega$ almost
everywhere and $\int_0^1\int(b_t^2+|\tau_t|^2)d\nu_tdt<\infty$, under the
distributional continuity--reaction constraint
$$
\partial_t\nu_t+\divS(\nu_t\tau_t)=2b_t\nu_t.
$$

The next proposition identifies the ambient homogeneous action with this spherical unbalanced action.
\begin{proposition}[Ambient action equals spherical action]
\label{prop:unbalanced}
For each ambient curve and homogeneous velocity satisfying
Proposition~\ref{prop:sphere} on $[0,T]$, $T>0$, its ambient action equals the
action of its spherical projection on the same interval. In particular,
\[
 \UWg(\nu_0,\nu_T)^2
 \le T\int_0^T\gamma\!\left(\int v_t v_t^\top\,d\mu_t\right)dt.
\]
\end{proposition}

\begin{proof}
Write $x=r\omega$ and $v(r\omega)=r(b(\omega)\omega+\tau(\omega))$. By definition of $\Pi_2(\mu)$,
$$
\int_{\R^d} v(x)v(x)^\top\,d\mu(x)
=
\int_{\Sd}(b(\omega)\omega+\tau(\omega))(b(\omega)\omega+\tau(\omega))^\top\,d\Pi_2(\mu)(\omega).
$$
Integrating yields equality of the two actions. Rescaling $[0,T]$ to
$[0,1]$ multiplies velocities by $T$, hence multiplies the total action by
$T$; the definition of $\UWg$ then gives the inequality.
\end{proof}

This identity does not equate the distances of fixed lifts. For example,
$\mu_0=\delta_{e_1}$ and
$\mu_1=\tfrac12\delta_0+\tfrac12\delta_{\sqrt2e_1}$ both project to
$\delta_{e_1}$, while
$\Wg(\mu_0,\mu_1)^2=(2-\sqrt2)\gamma(e_1e_1^\top)>0$.
A metric-quotient theorem would require an additional minimization over
lifts and a converse lifting result, neither of which is asserted here.

The following remark explains how the trace-norm case collapses to the classical Wasserstein--Fisher--Rao geometry.
\begin{remark}
When $\gamma(S)=\tr(S)$, the integrand becomes
$$
\int_{\Sd}\bigl(\abs{\tau_t(\omega)}^2+b_t(\omega)^2\bigr)\,d\nu_t(\omega).
$$
Since the reaction rate is $\alpha_t=2b_t$, the reaction term is
$\alpha_t^2/4$. Thus the trace case is the Wasserstein--Fisher--Rao action
with this transport/reaction normalization, whose static formulations are
developed in \citet{ChizatPeyreSchmitzerVialard2018,LieroMielkeSavare2018}.
Writing $\mathsf{WFR}$ for this normalization, trace comparison gives
\[
 \sqrt{c_\gamma}\,\mathsf{WFR}\le\UWg
 \le\sqrt{C_\gamma}\,\mathsf{WFR}.
\]
Thus $\UWg$ is finite and separates measures. Symmetry follows by reversing
time, and concatenating curves with durations proportional to the square
roots of their actions proves the triangle inequality. Hence it is a metric
on finite spherical measures. Whether it admits an analogous static
formulation is an interesting open direction.
\end{remark}

% Requires: \usepackage{longtable,booktabs,array}
\section{Notation}
\label{sec:notation}

This section collects the main notation used throughout the paper. The first table lists standard background notation, while the second table lists the notation introduced for the Spectral Wasserstein geometry.

\begin{longtable}{@{}p{0.23\linewidth}p{0.39\linewidth}p{0.28\linewidth}@{}}
\caption{Classical notation.}\label{tab:notation-classical}\\
\toprule
Notation & Name & First use \\
\midrule
\endfirsthead
\toprule
Notation & Name & First use \\
\midrule
\endhead
\bottomrule
\endfoot
$\R^d$ & Euclidean space & Sec.~\ref{sec:flows} \\
$\mathbb S^d$ & Symmetric matrices & Sec.~\ref{sec:flows} \\
$\mathbb S_+^d$ & PSD matrices & Sec.~\ref{sec:flows} \\
$\mathbb S_{++}^d$ & Positive definite matrices & Prop.~\ref{prop:dual-inverse-trace-selector} \\
$S\preceq T$ & Loewner order & Sec.~\ref{sec:flows} \\
$I$, $\Id$ & Identity matrix/map & Sec.~\ref{sec:flows} \\
$\tr(S)$ & Trace & Def.~\ref{def:monotone-norm} \\
$\diag(a_i)$ & Diagonal matrix & Sec.~\ref{sec:flows} \\
$\lambda(S)$ & Eigenvalues of $S$ & Sec.~\ref{sec:flows} \\
$\|\cdot\|_F$ & Frobenius norm & Def.~\ref{def:matrix-lmo} \\
$\langle\cdot,\cdot\rangle_F$ & Frobenius product & Def.~\ref{def:matrix-lmo} \\
$\|A\|_{S_p}$ & Schatten $p$-norm & Sec.~\ref{sec:flows} \\
$A^\dagger$ & Pseudoinverse & Eq.~\eqref{eq:muon-selector} \\
$\cP_2(\R^d)$ & Measures with finite second moment & Def.~\ref{def:duality} \\
$\delta_x$ & Dirac mass & Sec.~\ref{sec:flows} \\
$T_\#\mu$ & Pushforward & Sec.~\ref{sec:dynamic} \\
$\Pi(\mu,\nu)$ & Couplings & Def.~\ref{def:static} \\
$L^2(\mu;\R^d)$ & Square-integrable vector fields & Def.~\ref{def:duality} \\
$\nabla$, $\nabla_x$ & Euclidean gradient & Sec.~\ref{sec:flows} \\
$\divv$ & Divergence & Eq.~\eqref{eq:measure-spectral-flow} \\
$\delta\MF/\delta\mu$ & First variation & Sec.~\ref{sec:flows} \\
$\Wtwo$ & Quadratic Wasserstein distance & Sec.~\ref{sec:intro} \\
$\mathcal N(m,\Sigma)$ & Gaussian law & Prop.~\ref{thm:gaussian} \\
$\Sigma$ & Covariance matrix & Prop.~\ref{thm:gaussian} \\
$\Sd$ & Unit sphere & App.~\ref{app:homogeneous} \\
$\nablaS$, $\divS$ & Spherical gradient/divergence & Prop.~\ref{prop:sphere} \\
\end{longtable}

\begin{longtable}{@{}p{0.23\linewidth}p{0.39\linewidth}p{0.28\linewidth}@{}}
\caption{Main notation introduced in the paper.}\label{tab:notation-paper}\\
\toprule
Notation & Name & Introduction \\
\midrule
\endfirsthead
\toprule
Notation & Name & Introduction \\
\midrule
\endhead
\bottomrule
\endfoot
$\gamma$ & Monotone matrix gauge on $\mathbb S_+^d$ & Sec.~\ref{sec:flows} \\
$\gamma_p$ & Schatten-$p$ gauge & Sec.~\ref{sec:flows} \\
$p'$ & Conjugate exponent of $p$ & Rem.~\ref{rem-kgamma-shatten} \\
$q=(2p)'$ & Tangent-space dual exponent & Prop.~\ref{prop:schatten} \\
$\cK$ & Representing set & Def.~\ref{def:representing-set} \\
$\mathcal P_\gamma$ & Polar set & App.~\ref{app:matrix-recap} \\
$\mathcal K_{\gamma_p}$ & Schatten representing set & Rem.~\ref{rem-kgamma-shatten} \\
$\mN(V)$ & Matrix tangent norm & Sec.~\ref{sec:flows} \\
$\mJ_\gamma(G)$ & Matrix LMO & Def.~\ref{def:matrix-lmo} \\
$\mJ_p(G)$ & Schatten matrix LMO & Prop.~\ref{prop:schatten} \\
$X$ & Particle matrix & Sec.~\ref{sec:flows} \\
$\mu_X$ & Empirical particle measure & Sec.~\ref{sec:flows} \\
$\mF(X)$ & Particle objective & Sec.~\ref{sec:flows} \\
$\MF(\mu)$ & Mean-field objective & Sec.~\ref{sec:flows} \\
$\MN_\mu(v)$ & Measure tangent norm & Def.~\ref{def:duality} \\
$\MJ_\gamma^\mu(g)$ & Measure-LMO & Def.~\ref{def:duality} \\
$\mathcal S_\mu(g)$ & Uncentered force second moment & Prop.~\ref{prop:measure-lmo-reduction} \\
$\mathfrak N_\gamma(T)$ & Root norm of an operator & Def.~\ref{def:root-norm-gauge} \\
$Q_\mu^*$ & Inverse-trace minimizer & Prop.~\ref{prop:dual-inverse-trace-selector} \\
$g_\mu$ & Wasserstein force (before normalization) & Sec.~\ref{sec:flows} \\
$\Wg$ & Static Spectral Wasserstein distance & Def.~\ref{def:static} \\
$\WQ{Q}$ & Fixed-metric transport cost & Thm.~\ref{thm:minmax} \\
$c_\gamma$, $C_\gamma$ & Trace comparison constants & Prop.~\ref{prop:bounds} \\
$\Wgbb$ & Dynamic Spectral Wasserstein distance & Def.~\ref{def:dynamic} \\
$\cA(\mu,m)$ & Momentum action & Prop.~\ref{prop:momentum} \\
$\WgM$ & Monge-restricted cost & Def.~\ref{def:monge} \\
$\phi(z,x)$ & Feature map & Sec.~\ref{sec:mlp-meanfield} \\
$u$, $v$ & Outer and inner weights (MLP context) & Sec.~\ref{sec:mlp-meanfield} \\
$H_\mu$ & Mean-field predictor & Sec.~\ref{sec:mlp-meanfield} \\
$R(H)$ & Risk functional & Sec.~\ref{sec:mlp-meanfield} \\
$\operatorname{MMD}(\mu,\nu)^2$ & Squared MMD & Sec.~\ref{sec:wasserstein-flows} \\
$(Q,K,V,O)$ & Attention weights & Sec.~\ref{sec:wasserstein-flows} \\
$N_{\mathrm{eff}}$ & Effective tokens & Sec.~\ref{sec:wasserstein-flows} \\
$C_\rho$ & Data covariance & App.~\ref{app:gaussian-linear} \\
$B_\rho$ & Teacher cross-moment & App.~\ref{app:gaussian-linear} \\
$M_\mu$ & Learned linear predictor & App.~\ref{app:gaussian-linear} \\
$A_\star$ & Teacher matrix & App.~\ref{app:gaussian-linear} \\
$E_t$ & Linear prediction error & Prop.~\ref{app-prop:gaussian-linear-covariance} \\
$r_i(t)$ & Modal predictor coefficient & Prop.~\ref{prop:modal-main} \\
$s_i(t)$ & Modal variance & Prop.~\ref{prop:modal-main} \\
$\beta_i$ & Teacher mode & Prop.~\ref{prop:modal-main} \\
$\mathcal N_t$ & Modal normalization & Prop.~\ref{prop:modal-main} \\
$\gamma_{\rm blk}$ & Sum block gauge (separate LMOs) & App.~\ref{sec:muon-extensions} \\
$\gamma_{\rm max}$ & Maximum block gauge (common polar clock) & App.~\ref{sec:muon-extensions} \\
$L_\gamma$ & Matrix smoothness constant & Eq.~\eqref{eq:spectral-smoothness} \\
$r_n=\min(n,d)$ & Matrix rank bound & App.~\ref{sec:smoothness-discretization} \\
$\mI_p(V)$ & Momentum constitutive map & Eq.~\eqref{eq:matrix-momentum} \\
$F_h(\mu)$ & Linear measure functional & Thm.~\ref{thm:linearconvex} \\
$\kappa$ & Convexity or P-L constant (by context) & Def.~\ref{def:geoconvex} \\
$\operatorname{Ent}_\nu(\mu)$ & Relative entropy & App.~\ref{app:convexity} \\
$\KL(\mu\,|\,\nu)$ & Relative entropy / KL divergence & Sec.~\ref{sec:relative-entropy-gaussians} \\
$|\partial\MF|_\gamma$ & Metric slope of $\MF$ & Sec.~\ref{sec:wasserstein-flows} \\
$\mathcal E_\gamma^\beta$ & Spectral Dirichlet energy & App.~\ref{app:pl-poincare} \\
$\lambda_{\mathrm P}^\gamma(\beta)$ & Spectral Poincar\'e constant & App.~\ref{app:pl-poincare} \\
$\theta_p=(2p)'$ & Dirichlet-energy Schatten exponent & App.~\ref{app:pl-poincare} \\
$\Phi$ & Two-homogeneous feature map & App.~\ref{app:homogeneous} \\
$\Pi_2(\mu)$ & Spherical projection & App.~\ref{app:homogeneous} \\
$\UWg$ & Unbalanced spherical action & App.~\ref{app:homogeneous} \\
\end{longtable}

\end{document}